\documentclass[11pt]{amsart}

\usepackage[utf8]{inputenc}
\usepackage[a4paper,twoside]{geometry}

\usepackage[usenames,dvipsnames]{xcolor}

\usepackage{enumitem}
\usepackage{amsmath, amssymb, amsthm,amsaddr}

\usepackage[english]{babel}
\usepackage{fancyhdr}
\usepackage[all]{xy}

\usepackage{tikz}
\usetikzlibrary{calc,decorations.markings,patterns}

\usepackage[framemethod=tikz]{mdframed}

\usepackage{float}
\usepackage{multirow}
\usepackage[colorlinks=true,linkcolor=blue,citecolor=red,backref=page,pdftex]{hyperref}

\newcommand{\retrait}{\hspace{1.7cm}}

\newcommand{\unp}{\mathbf{\mathrm{1 \kern-0.25em I}}}
\newcommand{\un}{\mathbf{1}}

\newcommand{\X}{\mathbf X}
\newcommand{\Abf}{\mathbf A}

\newcommand{\Sb}{\mathbb S}

\newcommand{\R}{\mathbb R}
\newcommand{\N}{\mathbb N}
\newcommand{\C}{\mathbb C}
\newcommand{\T}{\mathbb T}

\newcommand{\Z}{\mathbb Z}
\newcommand{\di}{{\mathrm d}}
\newcommand{\G}{\mathbf G}
\newcommand{\Hb}{\mathbf H}

\newcommand{\e}{\mathbf e}
\newcommand{\V}{\mathbb V}

\newcommand{\U}{\mathbb U}

\newcommand{\lib}{[} 
\newcommand{\rib}{]} 

\newcommand{\supp}{\mathrm{supp}\,}

\newcommand{\prob}{\mathbb P}
\newcommand{\esp}{\mathbb E}

\newcommand{\Msig}{[\sigma]_{_M}}
\newcommand{\Isig}{[\sigma]_{\infty}}

\newenvironment{miniabstract}{%
\begin{center}\begin{minipage}{0.8\linewidth} %
}{%
\end{minipage} \end{center}\medskip %
}

\newcommand{\convetoile}{\stackrel{ \ast }{ \rightharpoonup }}

\newcommand{\cal}{\mathcal}

\newcommand{\trace}{\mathrm{tr}}

\newtheorem{theorem}{Theorem}[section]
\newtheorem*{theorem*}{Theorem}
\newtheorem*{lemma*}{Lemma}
\newtheorem*{corollary*}{Corollary}
\newtheorem{lemma}[theorem]{Lemma}
\newtheorem{proposition}[theorem]{Proposition}
\newtheorem*{proposition*}{Proposition}
\newtheorem{corollary}[theorem]{Corollary}
\newtheorem{theorem_annexe}{Th\'eor\`eme}[section]
\newtheorem{lemma_annexe}[theorem_annexe]{Lemme}

\newtheorem{corollary_annexe}[theorem_annexe]{Corollaire}
\newtheorem*{corollary_annexe*}{Corollaire}

\theoremstyle{definition}
\newtheorem{definition}[theorem]{Definition}
\newtheorem*{definition*}{Definition}
\newtheorem{definition_annexe}[theorem_annexe]{D\'efinition}

\newtheorem{example}[theorem]{Example}

\newtheorem*{notations*}{Notations}
\newtheorem*{notation*}{Notation}

\theoremstyle{remark}
\newtheorem{remark}[theorem]{Remark}
\newtheorem{remark_annexe}[theorem_annexe]{Remarque}

\numberwithin{equation}{section}
\title{The rate of convergence for the renewal theorem in $\mathbb{R}^{\lowercase{d}}$}

\author{Jean-Baptiste Boyer}
\email{jeaboyer@math.cnrs.fr}

\address{IMB, Université de Bordeaux / MODAL'X, Université Paris-Ouest Nanterre}


\date{\today}

\begin{document}
\pagestyle{headings}
\thispagestyle{plain}

\begin{abstract}
Let $\rho$ be a borelian probability measure on $\mathrm{SL}_d(\mathbb{R})$. Consider the random walk $(X_n)$ on $\mathbb{R}^d\setminus\{0\}$ defined by $\rho$ : for any $x\in \mathbb{R}^d\setminus\{0\}$, we set $X_0 =x$ and $X_{n+1} = g_{n+1} X_n$ where $(g_n)$ is an iid sequence of $\mathrm{SL}_d(\mathbb{R})-$valued random variables of law $\rho$. Guivarc'h and Raugi proved that under an assumption on the subgroup generated by the support of $\rho$ (strong irreducibility and proximality), this walk is transient.

In particular, this proves that if $f$ is a compactly supported continuous function on $\mathbb{R}^d$, then the function $Gf(x) :=\mathbb{E}_x \sum_{n=0}^{+\infty} f(X_n)$ is well defined for any $x\in \mathbb{R}^d \setminus\{0\}$.

Guivarc'h and Le Page proved the renewal theorem in this situation : they study the possible limits of $Gf$ at $0$ and in this article, we study the rate of convergence in their renewal theorem.

To do so, we consider the family of operators $(P(it))_{t\in \mathbb{R}}$ defined for any continuous function $f$ on the sphere $\mathbb{S}^{d-1}$ and any $x\in \mathbb{S}^{d-1}$ by
\[
P(it) f(x) = \int_{\mathrm{SL}_d(\mathbb{R})} e^{-it \ln \frac{ \|gx\|}{\|x\|}} f\left(\frac{gx}{\|gx\|}\right) \mathrm{d}\rho(g)
\]
And we prove that, for some $L\in \mathbb{R}$ and any $t_0 \in \mathbb{R}_+^\ast$,
\[
\sup_{\substack{t\in \mathbb{R}\\ |t| \geqslant t_0}} \frac{ 1 }{|t|^L} \left\| (I_d-P(it))^{-1} \right\|  \text{ is finite}
\]
where the norm is taken in some space of hölder-continuous functions on the sphere.
\end{abstract}

\maketitle

\tableofcontents

\section{Introduction} \label{section:introduction_renouvellement_Rd}

\subsection{Main results}

Let $\rho$ be a borelian probability measure on $\mathrm{SL}_d(\R)$ ($2 \leqslant d\in \N$) and let $x\in \R^d \setminus\{0\}$. We define a random walk on $\R^d \setminus \{0\}$ starting at $x$ by
\[
\left\{\begin{array}{ccc} X_0 & =& x \\ X_{n+1} &=& g_{n+1} X_n \end{array} \right.
\]
where $(g_n)\in \mathrm{SL}_d(\R)^\N$ is an iid sequence of $\mathrm{SL}_d(\R)-$valued random variables of law $\rho$.

\medskip
In the sequel, we will say that a closed subgroup of $\mathrm{SL}_d(\R)$ acts strongly irreducibly and proximally on $\R^d$ if it doesn't fix any finite union of non trivial subspaces of $\R^d$ and if it contains an element which has an eigenvalue that is strictly larger than the other ones and whose eigenspace has dimension $1$. We will also that a borelian probability measure on $\mathrm{SL}_d(\R)$ is strongly irreducible and proximal it it's support generates a group that has these two properties.

\medskip
If $\rho$ is strongly irreducible, proximal and has a moment of order $1$\footnote{i.e. $\int_\G |\ln\|g\| |\di\rho(g)$ is finite.}, then a result by Furstenberg and Kesten (see~\cite{Fur63} and~\cite{GuRa85}) shows that, if $\|\,.\,\|$ is a norm on $\R^d$, then, for any $x\in \R^d \setminus\{0\}$,
\begin{equation}\label{equation:lambda_rho}
\frac 1 n \ln \| g_n \dots g_1 x \| \xrightarrow\, \lambda_\rho := \int_\G \int_{\prob(\R^d)} \ln \|gx\| \di\rho(g) \di \nu(x)>0\;\; \rho^{\otimes \N}-\text{a.e.}
 \end{equation}
where $\nu$ is a stationary measure on $\prob(\R^d)$ (which is unique as we will see in proposition~\ref{proposition:unique_mesure_invariante_contraction}).

In particular, this implies that the walk on $\R^d\setminus\{0\}$ is transient.

\medskip
Given a continuous function $f$ on $\R^d$, such that for some $\alpha \in \R_+^\ast$,
\[
\sup_{x\in \R^d} \frac{ |f(x)|}{\|x\|^\alpha} <+\infty \text{ and }\sup_{x\in \R^d} \|x\|^\alpha |f(x)| <+\infty,
\]
we study the function
\begin{equation} \label{equation:G}
\left(x\mapsto Gf(x) := \sum_{n=0}^{+\infty} \esp_x f(X_n) \right)
\end{equation}
Since the walk is transient and our assumptions on $f$, this function is well defined and even continuous on $\R^d \setminus\{0\}$ (the series converges uniformly on every compact subset of $\R^d \setminus\{0\}$) and we would like to study it's behaviour at $0$. This is what we call the renewal theorem in $\R^d$ by analogy to the situation in $\R$ (see for instance~\cite{Bla48}).

\bigskip
Guivarc'h and Le Page proved in~\cite{GuLe15} that if $T_\rho$, the sub-semigroup generated by the support of $\rho$, fixes a non trivial convex cone in $\R^d$ then, there are two stationary borelian probability measures $\nu_1$ and $\nu_2$ on the sphere $\Sb^{d-1}$ and the space of $P-$invariant continuous functions on the sphere has dimension $2$, a basis being given by two non negative functions $p_1,p_2$ such that $p_1+ p_2=1$ and ${p_i}{|_{\supp \nu_j}} = \delta_{i,j}$ where $\delta$ is Kronecker's symbol; on the other hand, if $T_\rho$ doesn't fix any non trivial convex cone in $\R^d$, then, there is a unique stationary borelian probability measure $\nu_1$ on $\Sb^{d-1}$ and we note $p_1$ the constant function taking the value $1$ on the sphere.

In both cases, we define an operator on the space of continuous functions on $\R^d$ that decay with a polynomial rate at infinity\footnote{There is $\alpha \in \R_+^\ast$ such that $\sup_{x\in \R^d} \|x\|^\alpha |f(x)|$ is finite.} setting, for such a function $f$ and $x\in \R^d \setminus\{0\}$,
\begin{equation} \label{equation:definition_Pi_0}
\Pi_0 f(x) =  \sum_{i=1}^r p_i\left(\frac x{\|x\|}\right) \int_{\Sb^{d-1}} \int_{\|x\|}^{+\infty} f(u y) \frac{\di u}{u} \di\nu_i(y) 
\end{equation}
where, $r\in\{1,2\}$ is the number of $T_\rho-$invariant closed subsets on the sphere.

\medskip
They also proved the following
\begin{theorem_annexe}[Guivarc'h - Le Page in~\cite{GuLe15}] \label{theorem:guivarch_lepage}
Let $\rho$ be a strongly irreducible and proximal borelian probability measure on $\mathrm{SL}_d(\R)$.

Then, for any $\gamma \in \R_+^\ast$ and any continuous function $f$ on $\R^d$ such that
\[
\sup_{v\in \R^d\setminus\{0\}}\frac{|f(v)|}{\|v\|^\gamma} \text{ and }\sup_{v\in \R^d} \|v\|^\gamma |f(v)| \text{ are finite}
\]
we have that
\[
\lim_{x\to 0} \left(G- \frac 1 {\lambda_\rho}\Pi_0\right)f(x) = 0
\]
Where $\lambda_\rho$, $G$ and $\Pi_0$ are defined in equations~\eqref{equation:lambda_rho}, \eqref{equation:G} and~\eqref{equation:definition_Pi_0}.
\end{theorem_annexe}

In particular, this theorem shows that if $f$ is a compactly supported hölder-continuous function on $\R^d$ such that $f(0)=0$, then the function $(G-\frac 1 {\lambda_\rho}\Pi_0)f$ can be extended at $0$ to a continuous function on $\R^d$.

So, the continuity of $Gf$ at $0$ is equivalent to the one of $ \Pi_0 f$.

Thus, in the case of a unique invariant closed subset on the sphere, we have that
\[
\lim_{x\to 0} \sum_{n=0}^{+\infty} P^n f(x) = \frac 1 {\lambda_\rho} \int_{0}^{+\infty} \int_{\Sb^{d-1}} f(uy) \di \nu(y) \frac{\di u}u
\]
and in the other case, we only have a ``directional limit'' : for any $x\in \R^d\setminus\{0\}$,
\[
\lim_{t\to 0^+} \sum_{n=0}^{+\infty} P^n f(tx) = \frac 1 {\lambda_\rho}\sum_{i=1}^2 p_i\left( \frac x {\|x\|} \right)\int_{0}^{+\infty} \int_{\Sb^{d-1}} f(uy) \di\nu_i (y) \frac{\di u} u
\]
And, in particular, the function $Gf$ cannot be extended to a continuous function at $0$ in general.

\begin{example}
If $T_\rho$ contains only matrices having positive coefficients, then it fixes the cone $\mathcal{C}$ of the vectors having only positive coefficients and it's opposite. Therefore, taking a positive regular function $f$ supported in $\mathcal{C}$, we see that $G f=0$ on $-\mathcal{C}$ whereas $Gf(x)$ will eventually be non negative on $\mathcal{C}$. Thus, we won't be able to extend it to a continuous function at $0$.
\end{example}

We would like to compute the modulus of continuity of $Gf$ at $0$ and, to do so, we want to study the rate of convergence in Guivarc'h and Le Page's result. To simplify our study, we will only consider $(G-\frac 1 {\lambda_\rho}\Pi_0)f$ and that will allow us to make no distinction between the number of closed invariant subsets on the sphere (and we will see in proposition~\ref{proposition:representation_renouvellement_fonctions_regulieres} that it is more than a computational trick). Thus, we will only have to study the modulus of continuity of $\Pi_0 f$ to get the one of $Gf$ and, as we have an easy formula for $\Pi_0 f$, it will be easy to find conditions that guarantee that $Gf$ can be extended to a continuous function at $0$ and to get it's modulus of continuity.

\smallskip
In~\cite{BDP15}, Buraczewski, Damek and Przebinda considered the case where $T_\rho$ is (conjuguated to) a subgroup of $\R^\ast_+ \times \cal O(d)$ and a diophantine condition is satisfied by the projection of $\rho$ on $\R_+^\ast$. They prove their result by going back to the $1-$dimensional case (this is why they need this diophantine condition that is necessary in this case (see for instance~\cite{Car83}) ; this hypothesis will always be satisfied in our case as wee will see in section~\ref{section:proprietes_diophantiennes_SL_d}).

Our study (and the one of Guivarc'h and Le Page) takes place in the opposite case where the subgroup generated by the support of $\rho$ contains an element having a strictly dominant eigenvalue (this is our proximality assumption).

\medskip
More specifically, we will prove the following
\begin{theorem_annexe}\label{theorem:renouvellement_0}
Let $\rho$ be a strongly irreducible and proximal borelian probability measure on $\mathrm{SL}_d(\R)$ having an exponential moment\footnote{There is $\varepsilon \in \R_+^\ast$ such that $\int_\G \|g\|^{\varepsilon} \di\rho(g)$ is finite.}.

Then, for any $\gamma>0$ small enough and any compact subset $K$ of $\R^d$, there are $C,\alpha\in \R$ such that for any continuous function $f\in \cal C^{0,\gamma}(\R^d)$ supported in $K$ and such that $f(0)=0$ and for any $x \in \R^d$,
\[
\left|\left(G-\frac 1 {\lambda_\rho}\Pi_0\right) f(x)\right| \leqslant \frac C {1+|\ln\|x\||^\alpha} \|f\|_\gamma
\]
Where $\lambda_\rho$, $G$ and $\Pi_0$ are defined in equations~\eqref{equation:lambda_rho}, \eqref{equation:G} and~\eqref{equation:definition_Pi_0}.
\end{theorem_annexe}

\label{discussion:A}
If one studies the linear random walk on the torus $\T^d:=\R^d/\Z^d$ defined by a probability measure on $\mathrm{SL}_d(\Z)$ (see for instance~\cite{BFLM11}), it appears that there are finite invariant subsets (e.g. the set $\{0\}$). If $\Abf$ is one of them that is also minimal, then one can identify a neighbourhood of $\Abf$ in the torus to a neighbourhood of  $\{0\} \times \Abf $ in $ \R^d \times \Abf$.

\medskip
This is why, from now on, noting $\Gamma_\rho$ the subgroup of $\mathrm{SL}_d(\R)$ generated by the support of $\rho$, we study the renewal theorem on the product of $\R^d$ and a finite $\Gamma_\rho-$set $\Abf$ on which the walk defined by $\rho$ is irreducible and aperiodic and we consider hölder continuous functions $f$ on $\R^d\times \Abf$.

\medskip
Remark that if $Gf(x,a) = \sum_{n=0}^{+\infty} P^nf(x,a)$ has a limit $g(a)$ when $x$ goes to $0$, then $(I_d-P)g(a) = f(0,a)$ and so $g$ is a solution to Poisson's equation for $f$ restricted to $\Abf$ (in particular, this implies that $\sum_{a\in \Abf} f(0,a) = 0$).

\medskip
Remark also that for any $f \in \cal C^0(\R^d\times \Abf)$ such that $\sum_{a\in \Abf} f(0,a)=0$ and any $a\in \Abf$, $ \sum_{n=0}^{+\infty} P^n f(0,a) $ is well defined since the random walk on $\Abf$ is irreducible and aperiodic.

\medskip
We modify our operator $\Pi_0$ to account of the dependence in $\Abf$ and we note, for any continuous function $f$ on $\R^d \times \Abf$ that decays at polynomial rate at infinity\footnote{There is some $\alpha \in \R_+^\ast$ such that $\sup_{(x,a) \in \R^d \times \Abf} \|x\|^\alpha |f(x)|$ is finite.}, any $x\in \R^d\setminus\{0\}$ and any $a\in \Abf$,
\begin{equation} \label{equation:definition_Pi_0_2}
\Pi_0 f(x,a) = \frac 1 {|\Abf|} \sum_{a'\in \Abf} \sum_{i=1}^r p_i\left(\frac x{\|x\|}\right) \int_{\Sb^{d-1}} \int_{\|x\|}^{+\infty} f(u y,a') \frac{\di u}{u} \di\nu_i(y) 
\end{equation}
Remark that if $\sum_{a\in \Abf} f(0,a) =0$ and for any $a\in \Abf$, the function $f(.,a)$ is hölder-continuous at $0$ then the limit of $\Pi_0 f(x,a)$ at $0$ is well defined (only radially if $r=2$).

\medskip
The main result of this article is the following
\begin{theorem_annexe} \label{theorem:renouvellement_Rd}
Let $\rho$ be a strongly irreducible and proximal borelian probability measure on $\mathrm{SL}_d(\R)$ having an exponential moment.

Let $\Abf$ be a finite $\Gamma_\rho-$set such that the random walk on $\Abf$ defined by $\rho$ is irreducible and aperiodic.

\medskip
Then, for any $\gamma>0$ small enough, there are constants $C \in \R$ and $\alpha \in \R_+^\ast$ such that for any function $f$ on $\R^d \times \Abf$ such that
\[
\|f\|_\gamma := \sup_{\substack{x,y\in \R^d \setminus\{0\} \\x\not=y \\a\in \Abf}} (1+\|x\|)^\gamma(1+\|y\|)^\gamma \frac{|f(x,a) - f(y,a)|}{\|x-y\|^\gamma} <+\infty,
\]
and such that for any $a\in \Abf$,
\[
\lim_{x\to \infty} f(x,a) = 0 \text{ and }\sum_{a\in \Abf} f(0,a) = 0
\]
We have that for any $a\in \Abf$,
\[
\lim_{x\to 0} \left(G- \frac 1 {\lambda_\rho} \Pi_0\right) f(x,a) = \sum_{n=0}^{+\infty} P^n f(0,a)
\]
Moreover, for any $x,y \in \R^d \setminus\{0\}$ and any $a\in \Abf$,
\[
\left|\left(G-\frac 1 {\lambda_\rho}\Pi_0\right)f(x,a) - \left(G-\frac 1 {\lambda_\rho}\Pi_0\right) f(y,a) \right| \leqslant C \omega_0(x,y)^\alpha \|f\|_\gamma
\]
where $\Pi_0$ is the operator defined in equation~\eqref{equation:definition_Pi_0_2} and where we noted, for any $x,y \in \R^d \setminus\{0\}$,
\[
\omega_0(x,y) =  \frac{\sqrt{\left|\ln \|x\| - \ln\|y\|\right|^2 + \left\|\frac x{\|x\|} - \frac y {\| y\|} \right\|^2}}{(1+|\ln\|x\||)(1+ |\ln \|y\||)}  
\]
\end{theorem_annexe}

\begin{remark_annexe}
The definition of the function $\omega_0$ may seem complicate but we will see that it is a kind of conical distance on $\R^d$. We will give more details about this function in section~\ref{section:renouvellement_holder}.
\end{remark_annexe}

\begin{remark_annexe}
The assumption on $f$ guarantees that $\lim_{\infty} f(.,a)=0$ and that there is a constant $C$ such that for any $x,y \in \R^d$ and any $a\in \Abf$,
\[
|f(x,a) - f(y,a)| \leqslant C \left(\frac{\|x-y\|}{(1+\|x\|)(1+\|y\|)} \right)^\gamma
\]
In particular, compactly supported hölder-continuous functions on $\R^d\times \Abf$ satisfy this assumption. Moreover, letting $y$ go to infinity, the equation shows that for any $x\in \R^d$,
 \[
 |f(x)| \leqslant \frac{ C}{(1+\|x\|)^\gamma}
 \]
So these functions vanish at polynomial speed at infinity.

We will not only consider compactly supported functions because our assumption will become very natural after identifying $\R^d\setminus\{0\}$ and $\Sb^{d-1} \times \R$ in chapter~\ref{section:renouvellement}.
\end{remark_annexe}

\begin{remark_annexe}
As we already said, it is the continuity of $G f$ that interests us but it is very easy to have the one of $\Pi_0 f$.
\end{remark_annexe}

To prove this theorem, we will study an analytic family of operators (see section~\ref{section:renouvellement}) defined on $\cal C^{0,\gamma}(\Sb^{d-1} \times \Abf)$ for $z\in \C$ with $|\Re(z)|$ small enough, a function $f\in \cal C^{0,\gamma}(\Sb^{d-1} \times \Abf)$ and some point $(x,a)$ of $\Sb^{d-1} \times \Abf$ by
\[
P(z) f(x,a) = \int_\G e^{-z \ln \frac{\|gx\|}{\|x\|}} f(gx,ga) \di \rho(g)
\]
Indeed, we will prove in section~\ref{section:renouvellement} that the rate of convergence in the renewal theorem is linked to the growth of the norm of $(I_d-P(z))^{-1}$ along the imaginary axis.

\bigskip
To get a control of $\|(I_d-P(it))^{-1}\|_{\cal C^{0,\gamma}(\Sb^{d-1} \times \Abf)}$ for large values of $t$, we will adapt in~\ref{section:Dolgopyat_Markov} the arguments developed by Dolgopyat in~\cite{Dol98} for Ruelle operators and we will prove proposition~\ref{proposition:Dolgopyat_markov} which links the norm of $\|(I_d-P(it))^{-1}\|$ to the diophantine properties of the logarithms of the spectral radii of elements of $\Gamma_\rho$.

\medskip
Then, we will prove that in a strongly irreducible and proximal subgroup of $\mathrm{SL}_d(\R)$, we can construct elements for which the logarithm of the spectral radius is very well controlled. This is what we will do in section~\ref{section:proprietes_diophantiennes_SL_d} and more specifically in proposition~\ref{proposition:rayons_spectraux_diophantiens}.

\subsection{Proofs}

\begin{miniabstract}
In this paragraph, we prove the results that we stated in this introduction from the ones we will prove in more general settings in the following ones.
\end{miniabstract}

\begin{proof}[Proof of theorem~\ref{theorem:renouvellement_0} from theorem~\ref{theorem:renouvellement_Rd}]~

Let $\gamma \in ]0,1]$ and $K$ a compact subset of $\R^d$. Then, there is a constant $C_0$ such that for any $\gamma-$hölder continuous function $f$ on $\R^d$ such that $\supp f\subset K$,
\[
\sup_{\substack{x,y \in \R^d\setminus\{0\}\\x\not=y}} (1+\|x\|)^\gamma(1+\|y\|)^\gamma\frac{|f(x)-f(y)|}{\|x-y\|^\gamma} \leqslant C_0 \|f\|_\gamma
\]
We can now apply theorem~\ref{theorem:renouvellement_Rd} to find constants $C,\alpha$ such that for any $\gamma-$hölder-continuous function $f$ with $\supp f\subset K$ and any $x,y\in \R^d\setminus\{0\}$,
\[
\left|\left(G-\frac 1 {\lambda_\rho}\Pi_0\right)f(x) - \left(G-\frac 1 {\lambda_\rho}\Pi_0\right)f(y) \right| \leqslant C \|f\|_\gamma \omega_0(x,y)^\alpha
\]
and
\[
\lim_{y\to 0} \left(G-\frac 1 {\lambda_\rho}\Pi_0 \right)f(y)=0
\]
But, we also have that
\[
\lim_{y\to 0} \omega_0(x,y) = \frac 1 {1+|\ln \|x\||}
\]
and this proves theorem~\ref{theorem:renouvellement_0}.
\end{proof}

\begin{proof}[Proof of theorem~\ref{theorem:renouvellement_Rd}]~

This is a direct application of our theorem~\ref{theorem:renouvellement_general}.

Indeed, noting $\X = \Sb^{d-1} \times \Abf$ and $\Hb= \{I_d,\vartheta\}$ where $\vartheta$ is the antipodal application on the sphere and identity on $\Abf$, we have that $\Hb$ acts by isometries on $\X\times \Abf$ and $(\X\times \Abf) / \Hb $, that we identify with the product of the projective space and $\Abf$ is well $(\rho,\gamma ,M,N)-$contracted over $\Abf$ (see example~\ref{exemple:espace_proj_contracte} and lemma~\ref{lemma:contraction_indeed}). Moreover, in section~\ref{section:proprietes_diophantiennes_SL_d}, we saw that the cocycle $\sigma$ defined on $\G\times\prob(\R^d)$ by $\sigma(g,X) = \ln \frac{\|gx\|}{\|x\|}$ for $x\in \X\setminus\{0\}$ also belong to $\cal Z^M(\prob(\R^d))$ and the result by Furstenberg that we already gave in this introduction proves that $\sigma_\rho >0$. 

Moreover, we saw in theorem~\ref{theorem:controle_resolvante_SL_d} that for any $t_0\in \R_+^\ast$ there are constants $C,L$ such that for any $t\in \R$ with $|t|\geqslant t_0$,
\[
\|(I_d-P(it))^{-1} \|\leqslant C|t|^L
\]

This proves that we can actually apply theorem~\ref{theorem:renouvellement_general} to any function $f$ that satisfies the assumption of theorem~\ref{theorem:renouvellement_Rd} since such a function can be identified to a function $\tilde f$ in $\mathcal C^{\gamma}_\omega(\X\times \R)$ such that $\sum_{a\in \cal A} \lim_{x\to -\infty} \tilde f(x,a) = 0$ and $\lim_{x\to +\infty }\tilde f(x,a) = 0$ by the map $(x,t) \mapsto e^t x$ from $\Sb^{d-1} \times \R $ to $ \R^d \setminus\{0\}$.
\end{proof}

\subsection{Notations and conventions}

For any continuous function $f$ on a topological space $\X$, we note $\supp f$ the support of $f$. In the same way, if $\nu$ is a borelian measure on $\X$, we note $\supp \nu$ it's support.

Moreover, we note
\[
\|f\|_\infty = \sup_{x\in \X} |f(x)|
\]

For any complex-valued function $f$ on a metric space $(X,d)$ and any $\gamma\in \mathopen{]}0,1\mathclose{]}$, we note
\[
m_\gamma(f) = \sup_{x\not=y} \frac{ |f(x)-f(y)|}{d(x,y)^\gamma} \text{ and } \|f\|_\gamma = \|f\|_\infty + m_\gamma(f)
\]
Moreover, we note $\cal C^{0,\gamma}(\X)$ the space of $\gamma-$hölder-continuous functions on $\X$ that we endow with the norm $\|\,.\,\|_\gamma$.

For $\eta \in \R_+^\ast$, we note
\[
\C_\eta = \{z\in \C| |\Re(z)|<\eta \}\text{ and }\overline{\C_\eta} =  \{z\in \C| |\Re(z)|\leqslant\eta \}\]

For any $A,B\subset \R$, we note $A\oplus iB = \{a+ib| a\in A,\; b\in B\}$ and in particular, if $A\subset \R$, then $A \oplus i\R = \{z\in \C| \Re(z) \in A\}$

\medskip
For $f\in \mathrm{L}^1(\R)$, we note $\widehat f$ the Fourier-Laplace transform of $f$ that is defined for any $z\in \C$ such that the integral is absolutely convergent by
\[
\widehat f(z) = \int_\R f(x) e^{-zx} \di x
\]
\section{Unitary perturbations of Markov operators}\label{section:Dolgopyat_Markov}

\begin{miniabstract}
In this section, we study the perturbation of Markov operators coming from group actions by kernels of modulus one given by cocycles. The aim is to prove proposition~\ref{proposition:Dolgopyat_markov} that shows that if the perturbated operator has an eigenvalue close to $1$, then the cocycle is close to a coboundary.
\end{miniabstract}

Let $\rho$ be a borelian probability measure on $\R$ having an exponential moment and a drift $\lambda = \int_\R y\di\rho(y)>0$.

In \cite{Car83}, Carlsson proved that to obtain the rate of convergence for the renewal theorem, we have to find some constant $l\in \R_+$ such that
\[
\liminf_{t\to \pm \infty}{|t|^l} \left| 1 - \int_\R e^{it y} \di \rho(y) \right| >0.
\]
This condition is linked to the diophantine properties of the $\rho-$generic elements (see for instance~\cite{Bre05} where a slightly different but similar condition is studied).

More specifically, is such a parameter $l$ exists, then the rate of convergence in the renewal theorem is polynomial and if we can even take $l=0$ (which is always the case if $\rho$ is spread-out as proved by Riemann-Lebesgue's lemma) then we can obtain an exponential rate of convergence (see~\cite{BG07}).

\bigskip
In this section, $\G$ will be a second countable locally compact group acting continuously on a compact metric space $(\X,d)$. We will fix a function $\sigma: \G\times \X\to \R$ (that will be a cocycle) and we will study the family of operators $(P(it))_{t\in \R}$ defined for any continuous function $f$ on $\X$ and any $x\in \X$ by
\[
P(it)f(x) = \int_\G e^{-it\sigma(g,x)} f(gx) \di \rho(g)
\]
To simplify notations, we simply note $P$ (or sometimes $P_\rho$ to insist on the measure $\rho$) the operator $P(0)$.  It is clear that if $\G$ acts continuously on $\X$, then $P$ preserves the space of continuous functions on $\X$.

\medskip
What corresponds to the diophantine condition for measures on $\R$ will be the existence of a constant $l\in \R_+$ such that
\[
\limsup_{t\to \pm\infty} \frac{1}{|t|^l}\|(I_d-P(it))^{-1}\| \text{ is finite.}
\]
Where the norm is taken in some Banach space (the space of hölder-continuous functions in our study).

\bigskip
To obtain this kind of control, we adapt a theorem proved for Ruelle operators by Dolgopyat in~\cite{Dol98} : this will be our proposition~\ref{proposition:Dolgopyat_markov} which is the aim of this section.

\subsection{Preliminaries}

\begin{miniabstract}
Before we state proposition~\ref{proposition:Dolgopyat_markov} properly, we introduce in this section a few technical notions.
\end{miniabstract}

\subsubsection{Contracting actions}
From now on, we assume that $\X$ fibers $\G-$equivariently over a finite $\G-$set $\Abf$. This means that we have a continuous map $\pi_{\Abf}: \X\to \Abf$ that is $\G-$equivariant : for any $x$ in $\X$ and any $g$ in $\G$,
\[
\pi_\Abf(gx) = g\pi_\Abf(x)
\]

\begin{definition}[Contracting action] Let $\G$ be a second countable locally compact group, $N:\G\to [1,+ \infty[$ a submultiplicative function on $\G$ and let $(\X,d)$ be a compact metric space endowed with a continuous action of $\G$.

We assume that $\X$ fibers $\G-$equivariantly over a finite {$\G-$set} $\Abf$.

Let $\rho$ be a borelian probability measure on $\G$ and $\gamma,M\in \R_+^\star$.

We say that $\X$ is $(\rho,\gamma,M,N)-$contracted over $\Abf$ if
\begin{enumerate}
\item For any $g\in \G$ and any $x,y\in \X$,
\begin{equation*}
d(g x,g y ) \leqslant MN(g)^M d(x,y)
\end{equation*}
\item
\begin{equation}\label{equation:exponential_moment}
\int_G N(g)^{M\gamma} \di \rho(g) \text{ is finite}
\end{equation}
\item For some $n_0\in \N^\star$ we have that
\[
\sup_{\substack{x,y\in \X \\ x\not=y\\ \pi_\Abf(x) = \pi_\Abf(y)}} \int_\G \frac{d(gx,gy)^\gamma}{d(x,y)^\gamma} \di\rho^{\star n_0}(g) <1
\]
where $\pi_\Abf: \X\to \Abf$ is the $\G-$equivariant projection.
\end{enumerate}
\end{definition}

\begin{remark}If $\X$ is $(\rho,\gamma, M,N)-$contracted over $\Abf$, then $P$ preserves the space $\cal C^{0,\gamma}(\X)$ of $\gamma-$Holder-continuous functions on $\X$.
\end{remark}

\begin{remark} \label{remark:contraction}
This notion is used for instance by Bougerol and Lacroix in~\cite{BL85} to study random walks on the projective space but the definition with such a generality is given in~\cite{BQred} where the reader will find more details.

We could have defined $N(g)$ as the maximum $d(gx,gy)/d(x,y)$ (assuming that it is finite) since this is a submultiplicative function on $\G$ ; however, in our applications, there will be a natural function $N$ associated to $\G$.
\end{remark}

\begin{example}\label{exemple:espace_proj_contracte}
Our main example will be the case where $\G$ is a strongly irreducible and proximal subgroup of $\mathrm{SL}_d(\R)$, $\rho$ is a borelian probability measure on $\G$ having an exponential moment and whose support generates $\G$ and $\X$ will be the product of the projective space $\prob^d(\R)$ (which is contracted according to the theorem 2.3 in chapter $\MakeUppercase{\romannumeral 5}$ in~\cite{BL85}) and of a finite $\G-$set $\Abf$ endowed with the discrete distance (for any $s,s'\in \Abf$, $d(s,s')=0$ if $s=s'$ and $1$ otherwise).
\end{example}

Remark that the sequence $(u_n)$ defined for any $n\in \N$ by
\[
u_n = \sup_{\substack{x,y\in \X \\ x\not=y\\ \pi_\Abf(x) = \pi_\Abf(y)}} \int_\G \frac{d(gx,gy)^\gamma}{d(x,y)^\gamma} \di\rho^{\star n}(g)
\]
is submultiplicative. Therefore, if $\X$ is $(\rho,\gamma,M,N)-$contracted over $\Abf$, then there are constants $C_1,\delta\in \R_+^\star$ such that for any $n\in \N$ and any $x,y\in \X$ such that $\pi_\Abf(x) = \pi_\Abf(y)$,
\begin{equation}\label{equation:contraction}
\int_\G d(gx,gy)^\gamma \di\rho^{\star n}(g) \leqslant C_1 e^{-\delta n} d(x,y)^\gamma
\end{equation}

Remark also that if $\gamma'\in\left]0,\gamma\right]$ then the function $t\mapsto t^{\gamma'/\gamma}$ is concave on $[0, \mathrm{Diam}(\X)]$ so if the space $\X$ is $(\rho,\gamma,M,N)-$contracted, it is also $(\rho,\gamma',M,N)-$contracted.

Let $\X$ be a compact metric space and $P$ a positive operator\footnote{For any non negative continuous function $f$ on $\X$, $Pf$ is non negative.} on $\cal C^0(\X)$. We say that the operator $P$ is \emph{equicontinuous} if it is power-bounded and if for any $f\in \cal C^0(\X)$, the sequence $(P^n f)_{n\in \N}$ is equicontinuous. We refer to~\cite{Rau92} for the properties of equicontinuous operators.

\begin{proposition} \label{proposition:unique_mesure_invariante_contraction}
Let $\G$ be a second countable locally compact group, $N:\G\to [1,+\infty[$ a submultiplicative function on $\G$ and $\rho$ a borelian probability measure on $\G$.

Let $(\X,d)$ be a compact metric space endowed with a continuous action of $\G$ and which is $(\rho,\gamma,M,N)-$contracted over a finite $\G-$set $\Abf$. 

Then, the operator $P$ associated to $\rho$ is equicontinuous on $\cal C^0(\X)$.

Moreover, if the random walk defined by $\rho$ on $\Abf$ is irreducible and aperiodic then there is a unique probability measure $\nu$ on $\X$ which is $P_\rho-$invariant.

Finally $1$ is the unique eigenvalue of $P$ having modulus $1$ and the associated eigenspace has dimension $1$.
\end{proposition}

Before we prove this result, we state a lemma about Markov chains defined by group actions on finite sets.

\begin{lemma}\label{lemma:markov_chains_finite_space_state}
Let $\G$ be a second countable locally compact group acting on a finite set $\Abf$ and let $\rho$ be a borelian probability measure on $\G$ such that the random walk on $\Abf$ defined by $\rho$ is irreducible and aperiodic.

Then, $\nu_\Abf$, the uniform measure on $\Abf$, is the unique $P_\rho-$stationary probability measure on $\Abf$ and $P_\rho$ has a spectral radius strictly smaller than $1$ in the orthogonal of constant functions in $\mathrm{L}^2(\Abf,\nu_\Abf)$.
\end{lemma}

\begin{proof}
According to the theory of Markov chains on finite state spaces (or more specifically Perron-Frobénius's theorem), we only have to remark that the measure $\nu_\Abf$ is stationary.
\end{proof}

\begin{proof}[Proof of proposition~\ref{proposition:unique_mesure_invariante_contraction}]
The equicontinuity of $P$ in the space $\cal C^0(\X)$ can be proved as in the case of a subgroup of $\mathrm{SL}_d(\R)$ acting on $\prob(\R^d)$ given in~\cite{BQproj}. We will give more details in the proof of proposition~\ref{proposition:quasi_compact_localement_contracte} where the space is only locally contracted.

\medskip
Let $f\in \cal C^0(\X)$ and $\lambda$ a complex number of modulus $1$. Assume that $Pf=\lambda f$.

For any $x,y \in \X$ such that $\pi_\Abf(x) = \pi_\Abf(y)$, we have that
\[
\lambda^n (f(x)-f(y)) = P^n f(x) - P^n f(y) = \int_\G f(gx) - f(gy) \di\rho^{\star n}(g)
\]
But, the space is contracted over $\Abf$ and $|\lambda |=1$ so, we get that for any $x,y\in \X$ with $\pi_\Abf(x) = \pi_\Abf(y)$, $f(x)=f(y)$.

Thus, eigenvectors of $P$ in $\cal C^0(\X)$ associated to eigenvalues of modulus $1$ can be identified to functions on $\Abf$. As we assumed that the Markov chain defined by $\rho$ on $\Abf$ is irreducible and aperiodic, we have that the only eigenvectors of $P$ associated to eigenvalues of modulus $1$ are constants (cf lemma~\ref{lemma:markov_chains_finite_space_state}). Using proposition~3.2 and 3.3 in~\cite{Rau92}, this proves that the measure $\nu$ is unique, that $1$ is a simple eigenvalue and that there is no other eigenvalue of modulus $1$.
\end{proof}

We can now extend to our context the theorem 2.5 of chapter $\MakeUppercase{\romannumeral 5}$ in~\cite{BL85} that proves that, when the space is contracted, the operator $P$ has a spectral gap in the space of hölder-continuous functions. This is the following

\begin{proposition}\label{proposition:spectral_gap_P0}
Let $\G$ be a second countable locally compact group, $N:\G\to [1,+\infty[$ a submultiplicative function on $\G$ and $\rho$ a borelian probability measure on $\G$.

Let $(\X,d)$ be a compact metric space endowed with an action of $\G$ and which is $(\rho,\gamma,M,N)-$contracted over a finite $\G-$set $\Abf$ on which the random walk defined by $\rho$ is irreducible and aperiodic.

Note $e^{-\kappa_\Abf}\in \left]0,1\right[$ and $C_\Abf\in\left[1,+\infty\right[$ such that for any function $f$ on $\Abf$ and any $n\in \N$,
\[
\left\|P^n f -\int f\di\nu_\Abf\right\|_\infty \leqslant C_\Abf e^{-\kappa_\Abf n} \|f\|_\infty
\]
where $\nu_\Abf$ is the uniform measure on $\Abf$ (see lemma~\ref{lemma:markov_chains_finite_space_state} for the existence of $\kappa_\Abf, C_\Abf$).

Let $\nu$ be the unique $P_\rho-$invariant borelian probability measure on $\X$ (given by proposition~\ref{proposition:unique_mesure_invariante_contraction}).

Then, there are constants $\kappa,C_0 \in \R_+^\star$ that don't depend on $C_\Abf$ and such that for any $n\in \N$,
\[
\left\| P^n_\rho -\Pi_\nu \right\|_{\cal C^{0,\gamma}(\X)} \leqslant C_0C_\Abf e^{-\kappa n}
\]
where we noted $\Pi_\nu$ the operator of integration against the measure $\nu$.
\end{proposition}

\begin{remark}
We quantify the spectral gap assumption in $\mathrm{L}^\infty(\Abf,\nu_{\Abf})$ since this will allow us to take a family $(\Abf_i,\nu_i)$ of $\G-$finite $\G-$sets on which $P$ has a uniform spectral gap.
\end{remark}

\begin{remark}
This proposition can be seen has a corollary of the quasicompacity of $P$ in $\cal C^{0,\gamma}(\X)$ that we will prove in proposition~\ref{proposition:quasi_compact_localement_contracte} and of the fact that, in $\cal C^0(\X)$, $1$ is the only eigenvalue of modulus $1$ and it's associated eigenspace has dimension $1$. However, what interests us is the dependence between the spectral gap in $\mathrm{L}^{\infty} (\Abf,\nu_\Abf)$ and the one in $\cal C^{0,\gamma}(\X)$. 
\end{remark}

\begin{proof}
Let $f\in \cal C^{0,\gamma}(\X)$, $x,y\in \X$ such that $\pi_\Abf(x) = \pi_\Abf(y)$ and $n\in \N$. Then, for any $n\in \N$, we can compute
\[
\left|P^nf(x) - P^nf(y) \right|  \leqslant m_\gamma(f)\int_\G d(gx,gy)^\gamma
\di\rho^{\star n} (g) \leqslant m_\gamma(f) C_1e^{-\delta n} d(x,y)^\gamma
\]
where we noted $C_1,\delta$ the constants given by equation~\eqref{equation:contraction}.

This proves that for any $n\in \N$,
\[
m_\gamma(P^n f) \leqslant  C_1e^{-\delta n} m_\gamma(f)
\]
We recall that we noted $\nu$ the unique $P-$invariant borelian probability measure on $\X$ (given by proposition~\ref{proposition:unique_mesure_invariante_contraction}).

Moreover, for any $x\in \X$ and any non zero integer $n$, we note $\nu_x$ the measure defined by
\[
\int \varphi(y)\di\nu_x(y) = |\Abf|\int_{\X} \un_{\pi_\Abf(x)=\pi_\Abf(y)} \varphi(y)\di\nu(y)
\]
Then, for any function $f\in \cal C^{0,\gamma}(\X)$, we note
\[
f_1^n(x)  = \int_\X P^nf(y) \di\nu_x(y) \text{ and }f_2^n(x) = P^nf(x) - f_1^n(x)
\]
Thus, for any $x,y\in \X$, we have that
\[
f_2^n(y) - m_\gamma(f_2^n) \mathrm{Diam}(\X) \leqslant f_2^n (x) \leqslant f_2^n (y) + m_\gamma(f_2^n) \mathrm{Diam}(\X)
\]
where we noted $\mathrm{Diam}(\X)$ the diameter of $\X$.

Therefore, integrating in the $y$ variable and using the fact that $\int_{\X} f_2^n(y)\di\nu_x(y) = 0$, we get that
\[
\|f_2^n\|_\infty \leqslant \mathrm{Diam}(\X)^\gamma m_\gamma(f_2^n) = \mathrm{Diam}(\X)^\gamma m_\gamma(P^nf)
\]
And as,
\[
P^{2n} f(x) = P^n( P^n f)(x) = P^n (f_2^n + f_1^n)(x) = P^n f_2^n(x) + P^n f_1^n(x)
\]
we also get that
\begin{align*}
\left|P^{2n} f(x) - \int_\Abf f_1^n(a) \di\nu_\Abf(a)\right| &\leqslant \|f_2^n\|_\infty + \left|P^n f_1^n(x) - \int_\Abf f_1^n (a) \di\nu_\Abf(a) \right| \\
& \leqslant \mathrm{Diam}(\X)^\gamma C_1e^{-\delta n} m_\gamma(f) + C_\Abf e^{-\kappa_\Abf n} \|P^nf_1^n \|_\infty \\
& \leqslant \left(\mathrm{Diam}(\X)^\gamma C_1 e^{-\delta n} + C_\Abf e^{-\kappa_\Abf n} \right) \|f\|_\gamma
\end{align*}

Finally, using Fubini's lemma, we obtain that
\[
\int_\Abf f_1^n(a) \di \nu_\Abf(a)= \int_{\X} f(y)\di\nu(y)
\]

This last equality ends the proof of the lemma since we also have that
\[
m_\gamma(P^n f) \leqslant C_1 e^{-\delta n} m_\gamma(f)
\]
And so,
\[
\left\|P^{2n} f- \int f\di\nu \right\|_\gamma \leqslant \left(CC_1 e^{-\delta n} + C_1 e^{-2\delta n}+ C_\Abf e^{-\kappa_\Abf n} \right) \|f\|_\gamma
\]
So we note $\kappa  = \frac 1 2\min(\delta,\kappa_\Abf)$ and $C_0 = (1+C)C_1 + 1$.
\end{proof}

\subsubsection{Fibered contracting actions}

\begin{miniabstract}
In this paragraph, we study the case where the space is only locally contracted and we recover some results of the previous paragraph.
\end{miniabstract}

To study the action of $\mathrm{SL}_d(\R)$ on the sphere and not only on the projective space, the notion of contractivity of the action is not enough (since the sphere isn't contracted as $x$ and $-x$ stays at fixed distance). However, this is the only obstruction since if we note $\theta$ the application on the sphere that sends any point $x$ onto $-x$, then it commutes to the action of $\G$ and so, noting $\Hb=\{I_d,\theta\}$, we have the identification $\Sb^{d-1}/\Hb \sim \prob^d$ and the projective space is $(\rho,\gamma,M,N)-$contracted (if $\rho$ has an exponential moment and is strongly irreducible and proximal) as we already noted in example~\ref{exemple:espace_proj_contracte}.

\medskip
This is why, from now on, we will consider compact metric $\G-$spaces $(\X,d)$ endowed with an action of a finite group $\Hb$ that commutes to the action of $\G$ and such that the quotient (endowed with it's quotient metric) is contracted. To simplify the lecture, the reader may keep $\G=\mathrm{SL}_d(\R)$, $\X=\Sb^{d-1}$, $\Hb=\{I_d,\theta\}$ and $\X/\Hb = \prob^d$.

\medskip
Our first step is then to recover an equivalent of lemma~\ref{proposition:unique_mesure_invariante_contraction} and proposition~\ref{proposition:spectral_gap_P0}.

To do so, we will use the following
\begin{theorem}[Ionescu-Tulcea and Marinescu in \cite{ITM50}]\label{theorem:ITM}
Let $(\cal B, \|\,.\,\|_{\cal B})$ be a Banach space and $P$ a continuous operator on $\cal B$.

Assume there is a norm $\|\,.\,\|$ on $\cal B$ such that the identity map between the spaces $(\cal B,\|\,.\,\|_{\cal B} )$ and $ (\cal B,\|\,.\,\|)$ is compact and that there are two constants $r,R \in \R_+$ such that for any $f\in \cal B$,
\[
\|Pf\|_{\cal B} \leqslant r\|f\|_{\cal B} + R \|f\| 
\]
Then, the essential spectral radius of $P$ in $(\cal B,\|\,.\,\|_{\cal B})$ is bounded by $r$. 
\end{theorem}

\begin{example}
In our examples, $(\cal B, \|\,.\,\|_{\cal B})$ will be a space of hölder-continuous functions endowed with it's Banach-space norm and $\|\,.\,\|$ will be the uniform norm.
\end{example}

\begin{proposition}\label{proposition:quasi_compact_localement_contracte}
Let $\G$ be a second countable locally compact group, $N:\G\to [1,+\infty[$ a submultiplicative function on $\G$ and $\rho$ a borelian probability measure on $\G$.

Let $(\X,d)$ be a compact metric $\G-$space endowed with an action of a finite group $\Hb$ that commutes to the one of $\G$ and such that $\X/\Hb$ is $(\rho,\gamma,M,N)-$contracted over a finite $\G-$set $\Abf$.

Then, there are $C',\delta' \in \R_+^\star$ such that for any $f\in \cal C^{0,\gamma}(\X)$ and any $n\in \N$,
\[
m_\gamma(P^n f) \leqslant C' \left( e^{-\delta' n} m_\gamma(f) + \|f\|_\infty \right)
\]
In particular, $P$ est is equicontinuous on $\cal C^0(\X)$ and it's spectral radius in $\cal C^{0,\gamma}(\X)$ is strictly smaller than $1$.
\end{proposition}

\begin{proof}
We do not prove this result here but later, in lemma~\ref{lemme:controle_P_z} when the operator is perturbated by a cocycle.
\end{proof}

Finally, we study the eigenvalues of $P$ in $\cal C^{0}(\X)$ having modulus $1$. To do so, we begin by studying the $P-$invariant borelian probability measures and then, we will see that, contrary to what happened when the space was contracted, there can be eigenvalues of modulus $1$ and different from $1$ and even non constant $P-$invariant functions.

This study will allow us to understand why we have to make an assumption about a cone being fixed or not in the renewal theorem.

\begin{lemma}\label{lemma:loi_grands_nombres_localement_contracte}
Let $\G$ be a second countable locally compact group, $N:\G\to [1,+\infty[$ a submultiplicative function on $\G$ and $\rho$ a borelian probability measure on $\G$.

Let $(\X,d)$ be a compact metric $\G-$space endowed with an action of a finite group $\Hb$ that commutes to the $\G-$action and such that $\X/\Hb$ is $(\rho,\gamma,M,N)-$contracted over a finite $\G-$set $\Abf$ on which the random walk defined by $\rho$ is irreducible and aperiodic.

Then, there are at most $|\Hb|$ minimal closed invariant subsets (for the action of $T_\rho$ the subsemigroup generated by the support of $\rho$) that we note $\Lambda_1, \dots ,\Lambda_r$. Each one is associated to a $P-$invariant $P-$ergodic borelian probability measure $\nu_i$ with $\supp\nu_i=\Lambda_i$.

Moreover, for any $x\in \X$ and $\rho^{\otimes \N}-$a.e. $(g_n) \in \G^\N$, the sequence
\[
\frac 1 n \sum_{k=0}^{n-1} \delta_{g_k \dots g_1 x}
\]
converges to one of the $\nu_i$ and if we note, for $i\in \lib 1,r\rib$,
\[
p_i(x) = \rho^{\otimes \N} \left(\left\{ (g_n) \middle| \frac 1 n \sum_{k=0}^{n-1} \delta_{g_k \dots g_1 x} \convetoile \nu_i \right\}\right)
\]
we have that the function $p_i$ is continuous, $P-$invariant, $\sum_i p_i =1$, $p_i  = \delta_{i,j}$ on $\Lambda_j$ (where $\delta_{i,j}$ Kronecker's symbol).

Finally, for any continuous function $f$ on $\X$ and any $x\in \X$,
\[
\frac 1 n \sum_{k=0}^{n-1} P^kf(x) \xrightarrow[n\to +\infty]\, \sum_{i=1}^r p_i(x) \int_\X f\di\nu_i
\]
\end{lemma}

\begin{proof}
Let $\Lambda$ be a minimal closed invariant subset (there is at least one since $\X$ is compact) and let $h\in \Hb$. Then, $h\Lambda$ is still a closed invariant subset since the actions of $\G$ and $\Hb$ commute. Moreover, it is also minimal since $h$ is invertible. 

This proves that, $\Hb\Lambda$ is again a closed $P-$invariant subset. But this time, it is also $\Hb-$invariant and so $\pi_\Hb(\Hb\Lambda)$ is an invariant subset of $P$ seen as an operator on $\cal C^{0,\gamma}(\X/\Hb)$. But this closed invariant subset is unique since $P$ is contracting on $\X/\Hb$ and the random walk on $\Abf$ defined by $\rho$ is irreducible and aperiodic (see proposition~\ref{proposition:unique_mesure_invariante_contraction}). This proves that $\Hb\Lambda$ is unique and that there are at most $|\Hb|$ minimal closed invariant subsets and that $\Hb$ acts transitively on them. We note them $\Lambda_1,\dots,\Lambda_r$ and we note ${\Lambda}$ their union. 

We saw in proposition~\ref{proposition:quasi_compact_localement_contracte} that $P$ is equicontinuous and using the propositions 3.2 and 3.3 of~\cite{Rau92}, we get that there are at most $r$ continuous $P-$invariant functions $p_1, \dots p_r$ forming a free family, and as each one is constant on $\Lambda_i$, we can always assume that $p_j=\delta_{i,j}$ on $\Lambda_i$. Thus, noting $\nu_i$ the $P-$invariant measure on $\Lambda_i$, we have that for any continuous function $f\in \cal C^{0}(\X)$,
\[
\lim_{n\to +\infty} \frac 1 n \sum_{k=0}^{n-1} P^kf(x) = \sum_{i=1}^r p_i(x) \int f\di\nu_i
\]
To conclude, we only need to check that $p_i$ really is the function we defined.

\medskip
First of all, the fact that for any $x\in \X$, $\frac 1 n \sum_{k=0}^{n-1} \delta_{g_k \dots g_1 x}$ converges a.e. to an ergodic measure depending continuously on $x$ is a consequence of the equicontinuity of $P$ and of the propositions of Raugi that we already used.

The fact that the function $p_i$ that we defined is $P-$invariant also comes from these propositions (see also equality $2.11$ in~\cite{BQproj}). So we can conclude using the unicity of these functions $p_1, \dots, p_r$.
\end{proof}

\subsubsection{Lazy random walk} \label{subsubsection:marche_ralentie}

Let $\G$ be a topological group. If $\rho$ is a borelian probability measure on $\G$, we will have to introduce the lazy random walk associated to $\rho$ : this is the walk associated to the measure 
\begin{equation}
\rho_\e = \frac 1 2 \delta_{\e} + \frac 1 2 \rho
\end{equation}
The main interest of this measure is that the sequence $(\supp \rho_{\e}^{\star n})_{n\in \N}$ is non decreasing. Moreover, for any $\lambda \in \C$,
\[
\lambda I_d-P_{\rho_{\e}} = \frac 1 2 ((2\lambda-1)I_d - P_\rho)
\] 
and so the spectral values of $P_{\rho_{\e}}$ and the ones of $P_\rho$ are linked (in particular, for $\lambda =1$, we get that $I_d-P_{\rho_{\e}} = \frac 1 2(I_d-P_\rho)$).

The following lemma proves that this measure keeps other properties of $\rho$.

\begin{lemma}\label{lemma:contraction_marche_ralentie}
Let $\G$ be a second countable locally compact group and $\rho$ a borelian probability measure on $\G$.

Let $(\X,d)$ be a compact metric $\G-$space endowed with an action of a finite group $\Hb$ that commutes to the $\G-$action and such that $\X/\Hb$ is $(\rho,\gamma,M,N)-$contracted over a finite $\G-$set $\Abf$.

Then, $\X/\Hb$ is also $(\rho_{\e},\gamma,M,N)-$contracted over $\Abf$.
\end{lemma}

\begin{proof}
It is clear that the first two properties are satisfied by $\rho_\e$.

Moreover, for any $n\in \N$, we have that
\[
\rho_{\e}^{\star n} = \frac 1 {2^n} \sum_{k=0}^{n} \binom n k \rho^{\star k}
\]
And so, for any $x,y \in \X$ such that $x\not=y$ and $\pi_\Abf(x) = \pi_\Abf(y)$ and any $n\in \N$, we have that
\begin{align*}
\int_\G \frac{d(gx,gy)^\gamma}{d(x,y)^\gamma} \di\rho_{\e}^{\star n}(g) &= \frac 1 {2^n} \sum_{k=0}^n \binom  n k \int_\G \frac{d(gx,gy)^\gamma}{d(x,y)^\gamma} \di\rho^{\star k}(g) \\&\leqslant \frac 1 {2^n} \sum_{k=0}^{n} \binom n k C_1 e^{-\delta k} \leqslant C_1 \left( \frac {1+e^{-\delta}} 2 \right)^n \qedhere
\end{align*}
\end{proof}

In the same way, we prove the following
\begin{lemma}
Let $\G$ be a second countable locally compact group and $\rho$ a borelian probability measure on $\G$.

Let $(\cal B,\|\,.\,\|_{\cal B})$ be a Banach space and $r:\G\to \mathrm{GL}(\cal B)$ a representation of $\G$ such that $\left\{\begin{array}{ccc}\G\times \cal B &\to& \cal B\\ (g,b) & \mapsto & r(g)b \end{array} \right.$ is continuous and $\int_\G \|r(g)\| \di\rho(g)$ is finite.

We note $P_\rho$ the operator $b\mapsto \int_\G r(g)(b)\di\rho(g)$.

We assume that there is an operator $N_0$ on $\cal{B}$ and $C,\kappa \in \R$ such that for any $n\in \N$, $\|P_\rho^n - N_0 \|_{\cal B} \leqslant Ce^{-\kappa n}$.

Then, for any $n\in \N$,
\[
\|P_{\rho_{\e}}^n -N_0 \|_{\cal B} \leqslant  C \left(\frac{1+e^{-\kappa}} 2 \right)^n
\]
where $P_{\rho_{\e}}$ is the operator associated to $\rho_{\e} = \frac 1 2 \delta_{\e} + \frac 1 2 \rho$.
\end{lemma}

\subsubsection{Perturbations of Markov operators by cocycles} \label{subsubsection:pertubation_cocycles}

In this paragraph, $\G$ still is a second countable locally compact group acting on a compact metric $\G-$space $(\X,d)$ that fibers $\G-$equivariantly over a finite $\G-$set $\Abf$ and $\rho$ a borelian probability measure on $\G$.

We are going to study a kind of perturbation of the Markov operator associated to $\rho$. To do so, we make the following

\begin{definition}[Cocycles] \label{definition:cocycle}
Let $\G$ be a topological group and $\X$ a topological space endowed with a continuous action of $\G$.

We say that a continuous function $\sigma : \G\times \X \to \R$ is a  \emph{cocycle} if for any $g_1,g_2\in \G$ and any $x\in \X$,
\[
\sigma(g_2 g_1, x) = \sigma(g_2, g_1 x) + \sigma(g_1,x) 
\]
Among the cocycles, we call \emph{coboundaries} the ones given by $\sigma(g,x) = \varphi(gx) - \varphi(x)$ where $\varphi:\X\to \R$ is a continuous function.
\end{definition}

\begin{remark}
Let $\sigma$ be a cocycle. Then, the operator defined for any $f\in \cal C^{0}(\X)$ and any $x\in \X$ by
\[
P(it)f(x) = \int_\G e^{-it\sigma(g,x)} f(gx) \di\rho(g)
\]
is continuous on $\cal C^{0}(\X)$ and for any $f\in \cal C^0(\X)$, any $x\in \X$ and any $n\in \N$, we have that
\[
P^n(it) f(x) = \int_\G e^{-it\sigma(g,x)} f(gx) \di\rho^{\star n}(g) \text{ and }\|P(it)^n f\|_\infty \leqslant \|f\|_\infty
\]
It is to have this equation that we only study cocycles and not more general functions on $\G\times \X$.
\end{remark}

As wa are going to study contracting actions (and thus holder-continuous functions) we want conditions that guarantee that $P(it)$ preserves the space of hölder-continuous functions on $\X$.

For a cocycle $\sigma$ and $g\in \G$, we note
\[
\sigma_{\mathrm{sup}}(g) = \sup_{x\in \X} |\sigma(g,x)| \text{ and }\sigma_{\mathrm{Lip}}(g) = \sup_{\substack{x,y\in \X\\ \pi_\Abf(x) = \pi_\Abf(y) \\ x\not=y}} \frac{|\sigma(g,x) - \sigma(g,y)|}{d(x,y)}
\] 

Then, for any $x,y\in \X$ with $x\not=y$ and $\pi_\Abf(x) = \pi_\Abf(y)$,
\begin{align*}
2^{\gamma-1}\left|e^{-it\sigma(g,x)}-e^{-it\sigma(g,y)}\right| &\leqslant\left|e^{-it\sigma(g,x)}-e^{-it\sigma(g,y)}\right|^\gamma \leqslant |t|^\gamma |\sigma(g,x) - \sigma(g,y)|^\gamma \\&\leqslant |t|^\gamma \sigma_{\mathrm{Lip}}^\gamma(g) d(x,y)^\gamma
\end{align*}
So, for any $g\in \G$, if $\sigma_{\mathrm{Lip}}(g)$ is finite, then the function $(x\mapsto e^{-it\sigma(g,x)})$ is hölder-continuous. 

We note, for $M\in \R_+$,
\[
\mathcal{Z}_N^M(\X) = \left\{\sigma\text{ is a cocycle}
\middle|  \sup_{g\in \G}\frac{\sigma_{\mathrm{Lip}}(g)}{N(g)^M } \text{ and }\sup_{g\in \G}\frac{e^{\sigma_{\mathrm{sup}}(g)}}{N(g)^M } \text{ are finite}\right\}
\]
and, for $\sigma \in \cal Z^{M}_N(\X)$, we note
\begin{equation} \label{equation:Msig_I_sig}
\Msig = \sup_{g\in \G} \frac{\sigma_{\mathrm{Lip}}(g)}{N(g)^M}  \text{ and }\Isig = \sup_{g\in \G}\frac{e^{\sigma_{\mathrm{sup}}(g)}}{N(g)^M } 
\end{equation}

The following proposition is an extension to our context of corollary~3.21 of Guivarc'h and Le~Page in~\cite{GuLe15}.

\begin{proposition}\label{proposition:essential_spectral_radius}
Let $\G$ be a second countable locally compact group, $N:\G\to [1,+\infty[$ a submultiplicative function on $\G$ and $\rho$ a borelian probability measure on $\G$.

Let $(\X,d)$ be a compact metric $\G-$space endowed with an action of a finite group $\Hb$ that commutes to the one of $\G$ and such that $\X/\Hb$ is $(\rho,\gamma_0,M,N)-$contracted over a finite $\G-$set $\Abf$ on which the random walk defined by $\rho$ is irreducible and aperiodic.

Let $\sigma \in \cal Z^{M}_N(\X/\Hb)$. Then there are $C_2, \delta_2 \in \R_+^\star$ such that for any $t\in \R$, any $n\in \N$ and any function $f\in \cal C^{0,\gamma}(\X)$, we have that
\[
m_\gamma(P^n(it) f) \leqslant C_2 \left( \|f\|_\infty (1+|t|) + e^{-\delta_2 n} m_\gamma(f) \right)
\]
In particular, the essential spectral radius of $P(it)$ is smaller than $e^{-\delta_2}$.
\end{proposition}

\begin{proof}
Let $f\in \cal C^{0,\gamma}(\X)$ and $x,y\in \X$ such that $x\not=y$ and $\pi_\Abf \circ \pi_{\Hb}(x) = \pi_\Abf \circ\pi_\Hb(y)$.

For any $n\in \N^\star$, we have that
\begin{align*}
\left|P^n(it) f(x) \right.&\left.- P^n(it) f(y)\right| = \left| \int_{\G} e^{-it\sigma(g,x)} f(gx) - e^{-it\sigma(g,y)} f(gy) \di\rho^{\star n} (g) \right| \\
& \leqslant \int_{\G} \left| f(gx)-f(gy)\right| \di\rho^{\star n}(g) + \|f\|_\infty \int_{\G} \left| e^{-it\sigma(g,x)} - e^{-it\sigma(g,y)} \right| \di\rho^{\star n}(g) \\
& \leqslant d(x,y)^\gamma m_\gamma(f)\int_{\G}  \frac{d(gx,gy)^{\gamma}}{d(x,y)^{\gamma}} \di\rho^{\star n}(g) \\
& \retrait + \|f\|_\infty 2^{1-\gamma}|t|^\gamma \Msig d(x,y)^\gamma\int_{\G} N^{\gamma M}(g)\di\rho^{\star n}(g)
\end{align*}
First of all, we note that as $N$ is submultiplicative, we have that
\[
\int_\G N(g)^{\gamma M} \di\rho^{\star n}(g)\leqslant \left( \int_{\G} N^{\gamma M}(g)\di\rho(g)\right)^n
\]

Moreover, since $\Hb$ is a finite group, there is $d_0 \in \R_+^\star$ such that for any $x,y\in \X$, is $d(x,y)\leqslant d_0$, then $d(x,y) = d(\pi_{\Hb}(x), \pi_{\Hb}(y))$.

Thus, for any $\varepsilon\in\left]0,1\right]$ and any $x,y \in \X$ such that $0<d(x,y) \leqslant \varepsilon d_0 $ and $\pi_\Abf \circ \pi_{\Hb}(x) = \pi_\Abf \circ\pi_\Hb(y)$, we have
\begin{align*}
I_n(x,y):&=\int_\G {d(gx,gy)^\gamma} \di\rho^{\star n}(g) \\
&= \int_\G \un_{d(gx,gy) \leqslant d_0} d(gx,gy)^\gamma + \un_{d(gx,gy) > d_0} d(gx,gy)^\gamma\di\rho^{\star n}(g) \\
&= \int_\G \un_{d(gx,gy) \leqslant d_0} d(g\pi_\Hb x,g\pi_\Hb y)^\gamma + \un_{d(gx,gy) > d_0} d(gx,gy)^\gamma\di\rho^{\star n}(g) \\
& \leqslant C_1 e^{-\delta n} d(x,y)^\gamma + d(x,y)^\gamma \int_\G \un_{d(gx,gy) > d_0} M^\gamma N(g)^{M\gamma} \di\rho^{\star n}(g) \\
& \leqslant  \left( C_1 e^{-\delta n} +M^\gamma \int_\G \un_{MN(g)^M \geqslant 1/\varepsilon} N(g)^{\gamma M} \di\rho^{\star n}(g) \right) d(x,y)^\gamma
\end{align*}
Thus, if $n_0$ is such that $C_1 e^{-\delta n_0} \leqslant 1/4$, as $\int_\G N(g)^{\gamma M} \di\rho^{\star n_0}(g)$ is finite, we can choose $\varepsilon$ such that
\[
\int_\G \un_{N(g)^M \geqslant 1/\varepsilon}M^\gamma N(g)^{\gamma M} \di\rho^{\star n_0}(g) \leqslant 1/4
\]
And so, for this choice of $\varepsilon$ and $n_0$, we have that for any $x,y \in \X$ such that $0<d(x,y) \leqslant \varepsilon d_0$ and $\pi_\Abf\circ\pi_\Hb(x) = \pi_\Abf \circ \pi_\Hb (y)$,
\[
\int_\G d(gx,gy)^\gamma \di\rho^{\star n_0}(g) \leqslant \frac 1 2 d(x,y)^\gamma
\]
This proves that for any $x,y\in \X$ with $\pi_\Abf \circ \pi_\Hb(x) = \pi_\Abf \circ \pi_\Hb(y)$ and $d(x,y) \leqslant \varepsilon d_0$ and any function $f\in \cal C^{0,\gamma}(\X)$,
\[
\frac{|P^{n_0}(it)f(x) - P^{n_0}(it)f(y)|}{d(x,y)^\gamma} \leqslant \frac 1 2 m_\gamma(f) + \|f\|_\infty 2^{1-\gamma}|t|^\gamma \Msig \left( \int_{\G} N^{\gamma_0 M}(g)\di\rho(g)\right)^n
\]
But, as we also have, for $x,y$ such that $\pi_\Abf\circ\pi_\Hb (x) = \pi_\Abf \circ \pi_\Hb (y)$ and $d(x,y) \geqslant \varepsilon d_0$, that
\[
\frac{|P^{n_0}(it)f(x) - P^{n_0}(it)f(y)|}{d(x,y)^\gamma} \leqslant \frac{2\|f\|_\infty}{(\varepsilon d_0)^\gamma}
\]
we get that for any function $f\in \cal C^{0,\gamma}(\X)$,
\[
m_\gamma(P^{n_0}(it) f) \leqslant \frac 1 2 m_\gamma(f) + \left( \frac{2}{(\varepsilon d_0)^\gamma} + 2^{1-\gamma} |t|^\gamma\Msig \left( \int_{\G} N^{\gamma_0 M}(g)\di\rho(g)\right)^{n_0}\right) \|f\|_\infty
\]
If we simplify the notations, what we just proves is that there is $n_0\in \N^\star$ and a constant $C\in \R_+$ (depending on $n_0$) such that for any $f\in \cal C^{0,\gamma}(\X)$,
\[
m_\gamma(P^{n_0}(it) f) \leqslant \frac 1 2 m_\gamma(f) + C(1+|t|) \|f\|_\infty
\]
Iterating this inequality, we find that there are constants $C_2,\delta_2 \in \R_+^\star$ such that for any $n\in \N$ and any $f\in \cal C^{0,\gamma}(\X)$,
\[
m_\gamma(P^n(it) f) \leqslant C_2 \left( e^{-\delta_2 n} m_\gamma(f) +(1+|t| )\|f\|_\infty \right)
\]
This proves, using theorem~\ref{theorem:ITM} that $P(it)$ has an essential spectral radius smaller than $e^{-\delta_2}$ and that it is equicontinuous.
\end{proof}

\subsubsection{Lower regularity of measures on metric spaces}

Guivarc'h proved (cf. \cite{BL85}) that if $\rho$ is a strongly irreducible and proximal borelian probability measure on $\mathrm{SL}_d(\R)$ having an exponential moment, then there is a unique $P-$invariant probability measure $\nu$ on $\prob(\R^d)$. Moreover, there are $\Delta^+, C\in\R_+^\star$ such that for any $x\in \prob(\R^d)$ and any $r\in\R_+$,
\[
\nu(B(x,r)) \leqslant Cr^{\Delta^+}
\]
(we refer to the chapter 12 in~\cite{BQred} for a proof of this result).

This property of upper regularity of the measure means that $\nu$ is not two much concentrated at neighbourhood of points in the projective space : if, for instance, it had an atom $x_0$ we would have that for any $\Delta^+\in \R_+^\star$, $\lim_{r\to 0^+} \nu(B(x_0,r))/r^{\Delta_+}=+\infty$.

\medskip
Here, we will have to use the lower regularity of the measure $\nu$ : at many steps we will have to use the fact that the measure of a ball of radius $r$ is polynomial in $r$. To study this, we make the following

\begin{definition}\label{definition:mesures_regulieres}
Let $(\X,d)$ be a compact metric space and $\nu$ a borelian probability measure on $\X$.

Let $\Delta \in \R_+$ and $t,r\in \R_+^\star$. 

We say that a point $x \in \X$ is \emph{$(\Delta,t)-\nu-$regular at scale $r$} if
\[
\nu(B(x,r)) \geqslant t r^\Delta
\]
In the same way, we will say that a point is $(\Delta,t)-\nu-$regular at any scale if
\[
\inf_{r\in \left]0,1\right]} \frac{ \nu(B(x,r))}{r^\Delta} \geqslant t
\]
Finally, we say that a point of $\X$ is $\Delta-\nu-$regular at scale $r$ if it is $(\Delta,1)-\nu$-regular at scale $r$.
\end{definition}

\begin{remark}
If $\X$ has an Hausdorff dimension smaller than $\Delta$ then we have (cf.~\cite{Rud87}) that
\[
\nu\left(\bigcup_{t\in \R_+^\star}\left\{x\in \X \middle| x\text{ is }(\Delta,t)-\nu-\text{regular at any scale} \right\}\right) =1
\]
\end{remark}

Sometimes, if $\pi_0:\X\to \X_0$ is a covering and $\nu$ is a probability measure on $\X_0$, we will say that $x\in \X$ is $\Delta-\nu$-regular at scale $r$ if $\pi_0(x)$ is.

\subsubsection{Isotypic decomposition}

\begin{miniabstract}
In this paragraph, we recall the isotypic decomposition that generalizes the decomposition of function on $\R$ between even and odd parts.
\end{miniabstract}

Let $\Hb$ be a finite group. For an irreducible unitary representation $\xi=(\rho,\V)$ of $\Hb$, we endow $\mathrm{End}(\V)$ of the Hilbert-Schmidt inner product defined for any $A,B\in\mathrm{End}(\V)$ by
\[
\langle A,B\rangle_{HS} := \trace{A^\star B}
\]
We note $|\,.\,|_{HS}$ the corresponding norm.

Let $(\X,d)$ be a compact metric space endowed with an action by isometries of $\Hb$ (this implies in particular that $\Hb$ preserves the space of hölder-continuous functions on $\X$).

The action of $\Hb$ on $\X$ gives a representation of $\Hb$ in $\cal C^{0}(\X)$ defined for any $h\in \Hb$,  $f\in \cal C^{0}(\X)$ and $x\in \X$ by
\[
\rho_0(h) f(x) = f(h^{-1} x)
\]
We note $\widehat{\Hb}$ a set of representatives of unitary irreducible representations of $\Hb$ up to isomorphism. This is a finite set.

For $\xi = (\rho,\V) \in \widehat \Hb$, $f\in \cal C^0(\X)$ and $x\in \X$, we note
\begin{equation} \label{equation:type_projection}
\widehat f(x,\xi) = \frac {\dim \V}{|\Hb|} \sum_{h\in \Hb}f(h^{-1} x) \rho(h)^\star
\end{equation}
Then, we have (see theorem 8 in~\cite{Ser78}) that for any $x\in \X$,
\begin{equation} \label{equation:type_projection_reciproque}
f(x) = \sum_{\xi \in \widehat{\Hb}} \trace\widehat f(x,\xi)
\end{equation}
However, wee will need the following relation : for any $x\in \X$ and any $h\in \Hb$,
\[
\widehat f(hx,\xi) = \frac {\dim \V}{|\Hb|} \sum_{h'\in \Hb} f((h^{-1}h')^{-1} x) \rho(h')^\star =  \widehat f(x,\xi) \rho(h)^\star
\]
Thus, for any $f\in \cal C^0(\X)$, any $x\in \X$, any $\xi \in \widehat \Hb$ and any $h \in \Hb$, we have that
\[
|\widehat f(hx,\xi)|_{HS} = |\widehat f(x,\xi)|_{HS}
\]
and the function $(x\mapsto |\widehat f(x,\xi)|_{HS})$ can be identified to a continuous function on $\X/\Hb$.

The norm $|\,.\,|_{HS}$ allows us to define a uniform norm on bounded functions on $\X$ taking their values in $\mathrm{End}(\V)$ : we note, for such a function $f$,
\[
\|f\|_\infty = \sup_{x\in \X} |f(x)|_{HS}
\]
In the same way, we can define hölder-continuous functions from $\X$ to $\mathrm{End}(\V)$ and note
\[
\mathcal{C}^{0,\gamma}_\xi(\X,\mathrm{End}(\V)) =\left\{ f\in \cal C^{0,\gamma}(\X,\mathrm{End}(\V_\xi)) \middle| \forall x\in \X\, \forall h\in \Hb \; f(hx)  =f(x) \rho(h)^\star \right\}
\]
(To simplify notations, we will simply note this space $\cal C^{0,\gamma}_\xi(\X)$ and sometimes $\cal C^{0,\gamma}_\xi$).

\medskip
We now have the following
\begin{lemma}\label{lemme:decomposition_paire_impaire}
Let $\Hb$ be a group acting by isometries on a compact metric space $(\X,d)$.

Then, the space $\cal C^{0,\gamma}(\X)$ injects into
$\prod_{\xi \in \widehat \Hb} \cal C^{0,\gamma}_\xi(\X, \mathrm{End}(\V))$.
Moreover, for any $\xi \in \widehat \Hb$, the projection onto $\cal C_\xi^{0,\gamma}(\X)$ is given by equation~\eqref{equation:type_projection} and the reciproqual application is given by equation~\eqref{equation:type_projection_reciproque}.
\end{lemma}

\subsection{Control of the resolvant of the perturbated operator} \subsectionmark{Control of the resolvant}

\subsubsection{Statement of the proposition}

\begin{miniabstract}
We are now ready to state proposition~\ref{proposition:Dolgopyat_markov}, aim of this section.
\end{miniabstract}

We keep the notations of the previous part.

Let $\sigma : \G\times \X\to \R$ be an $\Hb-$invariant cocycle. We note, for $t\in \R$, $P(it)$ (or sometimes $P_{\rho}(it)$ to insist on the measure $\rho$) the operator defined for any continuous function $f$ on $\X$ and any $x$ in $\X$, by
\[
P(it) f(x) = \int_\G e^{-it\sigma(g,x)} f(gx) \di\rho(g)
\]
If $\rho_\e= \frac 1 2 \delta_\e + \frac 1 2 \rho$, we simply note $P_\e(it)= P_{\rho_\e}(it)$ the operator associated to the lazy random walk.

Finally, we define
\[
\Msig = \sup_{g\in \G} \sup_{\substack{ x,y\in \X \\ x\not=y\\ \pi_\Abf(x) = \pi_\Abf(y)}} \frac{\left|\sigma(g,x)-\sigma(g,y) \right|}{N(g)^M d(x,y)}
\]
We saw in paragraph~\ref{subsubsection:pertubation_cocycles} that if $\Msig$ is finite and if $(g\mapsto N(g)^{M\gamma})$ is $\rho-$integrable then $P(it)$ preserves the space of $\gamma-$hölder-continuous functions on $\X$.

\medskip
The aim of this section is to study the properties of $P(it)$ and to connect them to the ones of $P$. A first example is when $e^{-it\sigma}$ is a coboundary : $e^{-it\sigma(g,x)} = \varphi(gx) \varphi(x)^{-1}$ where $\varphi$ is a $\gamma-$Holder continuous function taking it's values in $\U$, the set of complex numbers of modulus $1$. Indeed, in this case, $P(it)$ is conjugated to $P$ by the operator of multiplication by $\varphi$ and so, these two operators have the same spectral properties. In particular, the operator $P(it)$ has the eigenvalue $1$ and an associated eigenvector is $\varphi^{-1}$.

\medskip
We are going to see that we can have a partial reciprocal statement : the proposition proves that if $I_d-P(it)$ is not well invertible (the norm of it's inverse is large) then $e^{-it\sigma}$ is close to a coboundary.

\begin{proposition}\label{proposition:Dolgopyat_markov}
Let $\G$ be a second countable locally compact group, $N:\G\to [1,+\infty[$ a submultiplicative function on $\G$ and $\rho$ a borelian probability measure on $\G$.

Let $(\X,d)$ be a compact metric $\G-$space endowed with an action by isometries of a finite group $\Hb$ that commutes to the $\G-$action and such that the space $\X/\Hb$ is $(\rho,\gamma,M,N)-$contracted over a finite $\G-$set $\Abf$ on which the random walk defined by $\rho$ is irreducible and aperiodic. We fix a class of representatives of unitary irreducible representations of $\Hb$ up to isomorphism.

Let $\nu$ be the unique $P-$invariant borelian probability measure on $\X/\Hb$ (see proposition~\ref{proposition:unique_mesure_invariante_contraction}).

Then, for any $\gamma>0$ small enough, any $\alpha_1,\beta\in \R_+^\star$, there is $\alpha_2 \in \R_+$, such that for any $\Delta \in \R_+$ such that there is a point $x \in \X/\Hb$ that is $\Delta-\nu$-regular at scale $2^{-\alpha_2}$ (see definition~\ref{definition:mesures_regulieres}) we have that there are $L, C \in \R_+$ such that for any $t\in \R$ with $|t|\geqslant 2$, we have that if 
\[
\|(I_d-P(it))^{-1} \|_{\cal C^{0,\gamma}(\X)} \geqslant C |t|^L
\]
then there is $\xi \in \widehat \Hb$ and a function $f\in \cal C_\xi^{0,\gamma}(\X)$ with $\|f\|_\infty \leqslant 1$ and $m_\gamma(f) \leqslant C|t|$ such that for any point $x$ in $\X$ whose projection on $\X/\Hb$ is $\Delta-\nu$-regular at scale $|t|^{-\alpha_2}$, we have that
\[
|f(x)|_{HS} \geqslant \frac 1 2
\]
and
\[
\int_\G  \left| e^{-it\sigma(g,x)}f(gx) - f(x) \right|^2_{HS} \di \rho_{\e}^{\star n(\beta,t)} (g)\leqslant \frac {1} {|t|^{\alpha_1}}
\]
where we noted
\[
n(\beta,t) = \lfloor \beta \ln |t| \rfloor
\]
and $\rho_{\e}$ is the measure associated to the lazy random walk (see paragraph~\ref{subsubsection:marche_ralentie}).
\end{proposition}

\begin{remark}
Note that in the conclusion of the proposition, it really is the measure $\rho_\e$ that is used and not the measure $\rho$ itself. This won't change anything in our study study since their supports generate the same subgroup of $\G$
\end{remark}

\begin{remark}
We would want to take $\varphi(x)= \frac 1 {\sqrt{\dim \V}} \trace f(x)$ since this will give the following inequality :
\[
\int_\G \left| e^{-it\sigma(g,x)}\varphi(gx) - \varphi(x) \right|^2 \di \rho_{\e}^{\star n(\beta,t)} (g)\leqslant \frac {1} {|t|^{\alpha_1}}
\]
And so, if $\varphi$ didn't vanish on $\X$, the theorem would imply that $e^{-it\sigma}$ were close to the coboundary $\varphi(x) \varphi(gx)^{-1}$.
However, our control on $|f(x)|_{HS}$ doesn't give any control on $|\varphi(x)|$.
\end{remark}

\subsubsection{Proof of the proposition}

The proof relies on lemmas that are adapted from the ones of Dolgopyat for Ruelle operators.

\medskip
The first difficulty comes from the fact that the space is only locally contracted over $\Abf$. To solve it, we are going to use the isotypic decomposition that we saw in lemma~\ref{lemme:decomposition_paire_impaire} and the fact that $\G$ preserves this decomposition since it's action and the $\Hb-$one commutes. Moreover, we also assumed that $\sigma$ is $\Hb-$invariant and so, we can study the operator $P(it)$ in each $\cal C^{0,\gamma}_\xi(\X)$.

One also has to remark that there is an equivalent to proposition~\ref{proposition:essential_spectral_radius} in the space $\cal C^{0,\gamma}_\xi(\X)$.

\medskip
Proposition~\ref{proposition:essential_spectral_radius} suggests that we should renormalise the norm in $\cal C^{0,\gamma}(\X)$ to study $P(it)$. This is why, we make the following
\begin{notation*}
Under the assumption of proposition~\ref{proposition:Dolgopyat_markov}, we note $C_2$ the constant given by proposition~\ref{proposition:essential_spectral_radius} (we can assume without any loss of generality that $C_2\geqslant 1$).

\medskip
Let $t \in \R$ with $|t|\geqslant 2$. Then, for any $f\in \cal C^{0,\gamma}(\X)$, we note
\[
\|\,.\,\|_{(t)}= \max\left( \|f\|_\infty, \frac{m_\gamma(f)}{2C_2 |t|} \right)
\] 
\end{notation*}
In the same way, we define the norm $\|\,.\,\|_{(t)}$ for functions in $\cal C^{0,\gamma}_\xi(\X)$.

\medskip
Remark that for any $f\in \cal C^{0,\gamma}(\X)$,
\[
\|f\|_{(t)} \leqslant \|f\|_\gamma \leqslant \left(1+ 2C_2|t|\right) \|f\|_{(t)}
\]
So $(\cal C^{0,\gamma}(\X), \|\,.\,\|_\gamma)$ and $(\cal C^{0,\gamma}(\X), \|\,.\,\|_{(t)})$ are isomorphic as Banach spaces. Moreover, $P(it)$ is better controlled with $\|\,.\,\|_{(t)}$ as shown by next
\begin{lemma}\label{lemma:control_norm_P_sigma}
Under the assumptions of proposition~\ref{proposition:Dolgopyat_markov}, for any $t\in \R$ with $|t|\geqslant 2$ and any $n\in \N$,
\[
\|P(it)^n \|_{(t)} \leqslant 2C_2
\]
\end{lemma}

\begin{remark}
This lemma still holds in $\cal C^{0,\gamma}_\xi(\X)$.
\end{remark}

\begin{proof}
Let $f\in \cal C^{0,\gamma}$ such that $\|f\|_{(t)} \leqslant 1$. According to proposition~\ref{proposition:essential_spectral_radius}, for any $n\in \N$, we have that
\[
\frac{m_\gamma(P^n(it) f)}{2C_2 |t|} \leqslant \frac 1 {2|t|} \left(1+|t| + 2C_2 |t| e^{-\delta_2 n}  \right) \leqslant 1+C_2 e^{-2\delta n} \leqslant 2C_2
\]
Moreover, we still have that
\[
\|P(it)^nf\|_\infty \leqslant \|f\|_\infty \leqslant 1
\]
and so,
\[
\|P^n(it)\|_\sigma \leqslant 2C_2
\]
Which is what we intended to prove.
\end{proof}

For $\alpha_2,\Delta, \sigma$ and $\xi$ fixed, we are going to study the assumption

\medskip
\begin{center}
$\mathbf{H}(\alpha_1, \beta,\xi):$~\begin{minipage}{0.8\textwidth} \label{hypothese:sigma}For any $f\in \cal C^{0,\gamma}_\xi(\X)$ with $\|f\|_{(t)} \leqslant 1$, there are $x_0 \in \X$ which is $\Delta-\nu$-regular at scale $|t|^{-\alpha_2}$ and $n\in [0, \lfloor \beta \ln|t| \rfloor]$ such that
\[
|P_{\e}(it)^n f(x_0)|_{HS} \leqslant 1- \frac {1}{|t|^{\alpha_1}}
\]
\end{minipage}
\end{center} 

\begin{lemma}\label{lemme:premier_lemme_dolgopyat}
Under the assumptions of proposition~\ref{proposition:Dolgopyat_markov}, for any $\alpha_1,\beta\in\R_+^\star$ there is $\alpha_2 \in \R_+^\star$ such that for any $\Delta \in \R_+^\star$, there are $\alpha'_1,C$, such that for any $t\in \R$ with $|t|\geqslant 2$ and any $\xi\in \widehat \Hb$, we have that if the assumption $\mathbf{H}(\alpha_1,\beta,\xi)$ is true, then
\[
\|(I_d-P(it))^{-1}\|_{\cal C^{0,\gamma}(\X)} \leqslant C |t|^{\alpha'_1}
\]
\end{lemma}

\begin{proof}
Let $\xi \in \widehat \Hb$ and $f\in \cal C^{0,\gamma}_\xi(\X)$ with $\|f\|_{(t)} \leqslant 1$.

By assumption, there are $n\in \lib 0,n(\beta,|t|)\rib$ and a point $x_0 \in \X$ whose projection on $\X/\Hb$ is $\Delta-\nu$-regular at scale $|t|^{-\alpha_2}$ such that
\[
|P^{n}_\e(it)f(x_0)|_{HS} \leqslant 1 - \frac 1 {|t|^{\alpha_1}}
\]
We are going to prove in a first step that this control at some point $x_0$ can be extended to a control of the uniform norm of $P^m_{\e}(it) f$ for some $m$ and then that this implies the expected result.

First of all, using the triangular inequality, we have that for any $m,n\in \N$ with $m\geqslant n$ and any $x\in \X$,
\begin{align*}
|P^{m}_{\e}(it) f(x)|_{HS} &=\left|\int_\G e^{it\sigma(g,x)} P_{\e}(it)^{n} f(gx) \di\rho_\e^{\star m-n}(g) \right|_{HS}  \\
& \leqslant \int_\G |P^n_{\e}(it) f(gx)|_{HS} \di\rho_\e^{\star m-n}(g) = P_{\e}^{m-n} |P^n_{\e}(it) f|_{HS}(x)
\end{align*}
Moreover, since $\sigma$ is $\Hb-$invariant and the action of $\Hb$ and $\G$ commutes, we also have, by definition of $\cal C^{0,\gamma}_\xi(\X)$, that for any $m\in \N$, the function $|P^m_{\e}(it) f|_{HS}$ is $\Hb-$invariant (cf. lemma~\ref{lemme:decomposition_paire_impaire}). So, it passes to a function on $\X/\Hb$ and we get, using proposition~\ref{proposition:spectral_gap_P0}, that
\[
P_{\e}^{m-n} |P^n_{\e}(it) f(gx)|_{HS}(x) \leqslant \int_{\X/\Hb} |P^n_{\e}(it) f(y)|_{HS}  \di\nu(y)+ C_0 C_\Abf e^{-\kappa(m-n)}\|P^{n}_{\e}(it) f\|_{\gamma} 
\]
Moreover, using lemma~\ref{lemma:control_norm_P_sigma}, the assumption on $\|f\|_{(t)}$ and the fact that $C_2 \geqslant 1$, we can compute
\begin{align*}
\|P^n_{\e}(it) f\|_\gamma &= \|P^n_{\e}(it) f\|_\infty + m_\gamma( P^n_{\e}(it) f) \leqslant 1 + C_2 \left(  (1+|t|)\|f\|_\infty + e^{-\delta n} m_\gamma(f) \right) \\
& \leqslant 1+C_2( 2|t| + 2 e^{-\delta_2 n} C_2 |t|) \\
& \leqslant 5C_2^2|t|
\end{align*}
Then, as $n\leqslant \beta \ln |t|$, we also have that
\[
e^{\kappa n}  |t| \leqslant e^{\kappa \beta \ln |t|}  |t| \leqslant |t|^{1+ \beta \kappa}
\]
And so, we get that for any $m\in \N$ larger than $ \beta\ln |t|$ and any $x\in \X$,
\[
|P_{\e}(it)^m f(x)|_{HS} \leqslant \int_{\X/\Hb} |P^{n}_{\e}(it) f(y)|_{HS}\di\nu(y) + 5 C_0C_\Abf C_2^2e^{-\kappa m} b^{1+\beta \kappa}
\]

Moreover, if $\mathbf{Z}$ is a borelian subset of $\X/\Hb$ and $M_{\mathbf{Z}}=\sup_{x\in{\mathbf{Z}}} |P^{n}_{\e}(it) f(x)|_{HS}$, then,
\begin{align*}
\int_{\X/\Hb} |P^{n}_{\e}(it) f(y)|_{HS}\di\nu(y) &\leqslant M_{\mathbf{Z}} \nu({\mathbf{Z}}) + \nu({\mathbf{Z}}^c) \leqslant 1+ (M_{\mathbf{Z}}-1)\nu({\mathbf{Z}})
\end{align*}
Taking $\mathbf{Z}=B(\pi_\Hb(x_0),r)$ where we noted $r=(1/(10C_2^2 |t|^{\alpha_1+1}))^{1/\gamma}$, we get that
\begin{align*}
\sup_{x\in \mathbf{Z}} |P^{n}_{\e}(it)f(x)|_{HS} &\leqslant |P^{n}_{\e}(it)f(\pi_\Hb(x_0))|_{HS} + \|P^{n}_{\e}(it)f\|_{\gamma} d(x,\pi_\Hb( x_0))^\gamma \\& \leqslant 1-\frac 1 {|t|^{\alpha_1}} + 5C_2^2 |t| r^\gamma = 1-\frac 1 {2|t|^{\alpha_1}}
\end{align*}
Taking now $\alpha_2$ large enough so that $|t|^{-\alpha_2} \leqslant (1/(5C_2^2 |t|^{\alpha_1+1}))^{1/\gamma}$, we get that $|t|^{-\alpha_2} \leqslant r$ and so, as $x_0$ is $\Delta-\nu-$regular at scale $|t|^{-\alpha_2}$,
\begin{align*}
\nu(\mathbf{Z})\geqslant |t|^{-\alpha_2 \Delta}
\end{align*}
and so,
\[
\int_{\X/\Hb}|P^{n}_{\e}(it) f(y)|_{HS} \di\nu(y) \leqslant 1 - \frac 1 {2 |t|^{\alpha_1+ \alpha_2 \Delta}} 
\]
To sum-up, we found that for any $m\in \N$ larger than $\beta\ln|t|$,
\begin{equation}
\|P_{\e}(it)^m f\|_\infty \leqslant  1 - \frac 1 {2 |t|^{\alpha_1 + \alpha_2 \Delta } } + 5C_0 C_\Abf C_2^2e^{-\kappa m} |t|^{1+\beta \kappa}
\end{equation}
To simplify notations, let $\alpha_3 = \alpha_1 + \alpha_2 \Delta$.

Let $m=K n(\beta, |t|)$ for some $K$ that we will choose later, then
\[
\|P^m_{\e}(it) \|_\infty \leqslant 1 - \frac 1 {2|t|^{\alpha_3}} + 4C_0 C_2^2 C_\Abf |t|^{1+\beta \kappa-K\kappa \beta}
\]
and so, for $K$ large enough (recall that $|t|\geqslant 2$), we get that
\[
\|P^m_{\e}(it) f\|_\infty \leqslant 1- \frac 1 {4|t|^{\alpha_3}}  
\]

Moreover, for $l\in \N$ larger than $m$, using proposition~\ref{proposition:essential_spectral_radius}, we find that
\begin{align*}
\frac 1 {2C_2|t|} m_\gamma( P^l_{\e}(it)f) &\leqslant \|P^{m}_{\e}(it)  f\|_\infty + \frac 1 {2|t|} e^{-\delta (l-m)} m_\gamma(P^m_{\e}(it)  f) \\
&\leqslant 1- \frac 1 {4 |t|^{\alpha_3}}  + \frac {1} {2|t|} e^{-\delta l} |t|^{K\beta \delta} 4C_2^2|t| \\
&\leqslant  1- \frac 1 {4b^{\alpha_3}}  + 2C_2^2 e^{-\delta l} |t|^{K\beta\delta}
\end{align*}
So, taking $l=L m = KLn(\beta,|t|)$, where $L\in \N$ is large enough, we get that
\[
\frac 1 {2C_2|t|} m_\gamma( P^l_{\e}(it)f) \leqslant 1 - \frac{1}{8|t|^{\alpha_3}}
\]
But, as we also have that
\[
\|P^l_{\e}(it)f\|_\infty \leqslant \|P^m_{\e}(it)f\|_\infty \leqslant 1 - \frac{1}{4 |t|^{\alpha_3}}
\]
what we proved is that under the assumptions of the theorem, is $\mathbf{H}(\alpha_1,\beta,\xi)$ is true, then for any $f\in \cal C^{0,\gamma}_\xi(\X)$ with $\|f\|_{(t)} \leqslant 1$,
\[
\|P^l_{\e}(it) f\|_{(t)} \leqslant 1 - \frac{1}{8 |t|^{\alpha_3}}
\]
And so, in $C_\xi^{0,\gamma}(\X)$,
\[
\|(I_d-P^l_{\e}(it))^{-1} \|_{(t)} \leqslant 8 |t|^{\alpha_3}
\]
Moreover, as
\[
(I_d-P_{\e}(it))^{-1} = \sum_{k=0}^{l-1} P_{\e}(it)^k (I_d-P_{\e}(it)^l)^{-1} \text{ and } \frac 1 2 (I_d-P(it))^{-1} =2 (I_d-P_{\e}(it))^{-1},
\]
we can compute
\[
\|(I_d-P(it))^{-1} \|_{(\cal C^{0,\gamma}_\xi(\X), \|.\|_{(t)})} \leqslant 2\sum_{k=0}^{l-1} \|P^k_{\e}(it)\|_{(t)} \|(I_d-P_{\e}(it)^l)^{-1}\|_{(t)} \leqslant  2C_2 l 8b^{\alpha_3}
\]
Finally, we recall that for any $f\in \cal C^{0,\gamma}_\xi(\X)$,
\[
\|f\|_{(t)} \leqslant \|f\|_\gamma \leqslant (1+2C_2|t|) \|f\|_{(t)} 
\]
and this proves that, in $\cal C^{0,\gamma}_\xi(\X)$,
\[
\|(I_d-P(it))^{-1}\|_\gamma \leqslant (1 + 2C_2 |t|)16 C_2 |t|^{1+\alpha_3} l
\]
And as $l$ is bounded by the product of $\ln |t|$ and some constant, we get the expected result for $\alpha'>1+\alpha_3$.
\end{proof}

We saw in lemma~\ref{lemme:premier_lemme_dolgopyat} what happens if, under the assumptions of proposition~\ref{proposition:Dolgopyat_markov}, for any hölder-continuous function $f$ on $\X$ there is a point $x_0$ and an integer $n$ such that $|P^{n}_\e(it) f(x_0)|_{HS}$ is far from $1$. We are now going to study the other alternative : the case where there is a function $f$ such that for any $x\in \cal \X$, $|P^n_\e(it) f(x)|_{HS}$ stays close to $1$ for many values of $n$.

\medskip
To explain how this works, we begin with the extreme case where we have that $|f(x)|$ and $|P_\e(it) f(x)|$ are not only close to $1$ but we even have that for any $x\in \X$, $ |P_\e(it) f(x)| = |f(x)|=1$. Then, for any $x\in \X$, $P_\e(it) f(x)$ is an average of complex numbers of modulus $1$ and so, for any $x\in \X$ and $\rho_\e-$a.e. $g\in \G$,
\begin{equation} \label{equation:Dolgopyat_first}
e^{-it\sigma (g,x)} f(gx) = P_\e(it) f(x)
\end{equation}
In particular, since by definition $\rho_\e(\e)>0$, we get that for any $x\in \X$,
\[
f(x) = e^{-it\sigma(\e,x)} f( x)= P_\e(it) f(x)
\]
This proves that for any $x\in \X$ and $\rho_\e-$a.e. $g\in \G$,
\[
e^{-it\sigma(g,x)} = \frac{ f(x)}{f(gx)}
\]
and so, $e^{-it\sigma}$ is a coboundary.

\medskip
We quantify this argument in next
\begin{lemma}\label{lemma:controle_fonctions_proches_1}
Under the assumptions of proposition~\ref{proposition:Dolgopyat_markov}.

Let $\xi\in \widehat \Hb$, $f\in \cal C^{0,\gamma}_\xi(\X)$ with $\|f\|_{\infty} \leqslant 1$, $x\in\X$ and $L\in\N^\star$.

Let $s$ be such that $1 -s \leqslant|f(x)|_{HS}, |P^L_{\e}(it) f(x)|_{HS}$, then we have that
\[
\int_\G \left| e^{-it\sigma(g,x)} f(gx) - f(x)\right|^2_{HS} \di\rho_\e^{\star L}(g) \leqslant 2^{L+3} s
\]
\end{lemma}

\begin{proof}
Expending the following expression, we can compute
\begin{align*}
I(x) :&=\int_\G \left| e^{-it\sigma(g,x)} f(gx) - P^{L}_{\e}(it)f(x)\right|^2_{HS} \di\rho_\e^{\star L}(g) \\
&= P_{\e}^L |f|_{HS}^2(x) + |P^{L}_{\e}(it) f(x)|_{HS}^2 - 2 |P^{L}_{\e}(it) f(x)|^2_{HS} \\
& \leqslant 1- |P^{L}_{\e}(it) f(x)|^2_{HS} \leqslant 2s
\end{align*}
Thus,
\[
\rho_\e^{\ast L}(\e) \left| f(x) - P^L_{\e}(it) f(x)\right|^2_{HS} \leqslant \int_\G \left| e^{-it\sigma(g,x)} f(gx) - P^{L}_{\e}(it)f(x)\right|^2_{HS} \di\rho_\e^{\star L}(g) \leqslant 2s
\]
This proves, using the fact that $\rho_\e(\e) \geqslant 1/2$, that
\[
\left| f(x) - P^L_{\e}(it)f(x)\right|_{HS} \leqslant \sqrt{2^{L+1}s}
\]
And so, the triangular inequality proves that,
\begin{align*}
\left( \int_\G \left| e^{-it\sigma(g,x)} f(gx) - f(x)\right|^2_{HS} \di\rho_\e^{\star L}(g) \right)^{1/2} &\leqslant \sqrt{I(x)} + \left| f(x) - P^L_{\e}(it) f(x)\right|_{HS} \\&\leqslant \sqrt{ 2s} + \sqrt{2^{L+1} s} \leqslant \sqrt{2^{L+3} s} 
\end{align*}
And this finishes the proof of the lemma.
\end{proof}

\begin{lemma}\label{lemme:second_lemme_dolgopyat}
Under the assumptions of proposition~\ref{proposition:Dolgopyat_markov}.

For any $\alpha_1,\beta$ there is $\alpha_2$ such that for any $\Delta$ there are $\alpha'_1, \beta'$ such that for any $\sigma\in \cal Z^M$, and any $\xi\in \widehat\Hb$, is the assumption $\mathbf{H}(\alpha'_1, \beta',\xi)$ is false then there is $f \in \cal C^{0,\gamma}_\xi(\X)$ with $\|f\|_{(t)}\leqslant 1$ such that for any $x\in \X$ whose projection on $\X/\Hb$ is $\Delta-\nu$-regular at scale $|t|^{-\alpha_2}$, we have that
\[
|f(x)|_{HS} \geqslant 1 - \frac 1 {|t|^{\alpha_1}}
\]
and
\[
\int_\G \left| e^{i\sigma(g,x)} f(gx) -  f(x) \right|^2_{HS} \di\rho_\e^{n(\beta,|t|)}(g) \leqslant \frac 1 {|t|^{\alpha_1}}
\]
\end{lemma}

\begin{proof}
Fix $\alpha_1,\alpha_2, \beta,\Delta$ and take $\alpha'_1 \in \R_+^\star$ large (we will precise this later in the proof).

If the assumption $\mathbf{H}(\alpha'_1, \beta,\xi)$ is not satisfied, then there is $f\in \cal C^{0,\gamma}_\xi(\X)$ with $\|f\|_{(t)} \leqslant 1$ such that for any $n\in \lib 0,n(\beta,t)\rib$ and any $x\in \X$ that is $\Delta-\nu$-regular at scale $|t|^{-\alpha_2}$,
\[
|P_{\e}(it)^n f(x)|_{HS} \geqslant 1 - \frac 1 {|t|^{\alpha'_1}}
\]
Using the previous lemma, we get that
\[
\int_\G \left| e^{-it\sigma(g,x)} f(gx) - f(x) \right|^2_{HS} \di\rho_\e^{\star n(\beta,t)} \leqslant \frac{2^{n(\beta,t)+3}}{|t|^{\alpha_1'}} \leqslant \frac{ 8}{|t|^{\alpha_1' - \beta \ln2}} 
\]
And this proves the lemma if we take $\alpha_1'> \beta \ln2 + \alpha_1+3$ since we have by assumption that $|t|\geqslant 2$.
\end{proof}

Lemmas~\ref{lemme:premier_lemme_dolgopyat} and~\ref{lemme:second_lemme_dolgopyat} prove proposition~\ref{proposition:Dolgopyat_markov}.
\section{Diophantine properties in linear groups} \label{section:proprietes_diophantiennes_SL_d}

\begin{miniabstract}
In this section, we prove that the logarithms of the spectral radii of elements of strongly irreducible and proximal subgroups of $\mathrm{SL}_d(\R)$ have good diophantine properties.

We also prove a property of lower regularity of the stationary measure on the projective space by showing that it is lower regular at fixed points of proximal elements.

This will allow us, in theorem~\ref{theorem:controle_resolvante_SL_d}, to control the operator $P(it)$ defined in section~\ref{section:Dolgopyat_Markov} and that appears in the study of the renewal theorem in section~\ref{section:renouvellement}.
\end{miniabstract}

In $\mathrm{SL}_d(\R)$, the application mapping a matrix to it's spectral radius is not a group morphism (for $d\geqslant 2$).

In a Zariski-dense subgroup $\Gamma$ of $ \mathrm{SL}_d(\R) $, we can construct sequences of elements $(g_n)$ and $(h_n)$ for which we have a good control of the difference between the logarithm of the spectral radius of $g_n h_n$ and the sum of the ones of $g_n$ and $h_n$ (see~\cite{Qu05}). In particular, this proves that the logarithms of the spectral radii of proximal elements of a Zariski-dense subgroup of $ \mathrm{SL}_d(\R) $ generate a dense subgroup of $\R$.

In this section, we quantify this construction to prove a technical result that will allow us to check the assumption of proposition~\ref{proposition:Dolgopyat_markov} and to prove theorem~\ref{theorem:controle_resolvante_SL_d} that gives a control of the resolvent of the pertubated operator.

\medskip
More specifically, studying the renewal in $\R$ (see~\cite{Car83} and the beginning of section~\ref{section:Dolgopyat_Markov}), we see that the rate of convergence depends on a diophantine condition on the measure and we are about to prove that it's equivalent in $\mathrm{SL}_d(\R)$ is always satisfied for measures having an exponential moment and whose support generate a strongly irreducible and proximal subgroup. This will be the
\begin{proposition*}[\ref{proposition:rayons_spectraux_diophantiens}]
Let $\rho$ be a strongly irreducible and proximal borelian probability measure on $\G:=\mathrm{SL}_d(\R)$ having an exponential moment.

Then, there are $\alpha, \beta \in \R_+^\ast$ and $p\in \N^\ast$ such that
\[
\liminf_{b\to \pm\infty} |b|^\alpha \int_\G \left| e^{ib \lambda_1(g)} - 1 \right|^2 \di\rho^{\ast pn(\beta,b)}(g) >0
\]
where we noted $\lambda_1(g)$ the logarithm of the spectral radius of $g$ and
\[
n(\beta,b) = \lfloor \beta \ln |b| \rfloor
\]
\end{proposition*} 

We recall that an element $g$ of $\mathrm{SL}_d(\R)$ is said to be proximal if it has a locally attractive fixed point in $\prob(\R^d)$.

\medskip
We will prove the genericness of the lower regular points for the stationary measure on the projective space, other condition that we used in the study of the perturbated operator, in the
\begin{corollary*}[\ref{corollary:regularite_inferieure_espace_projectif}] Let $\rho$ be a strongly irreducible and proximal borelian probability measure on $\mathrm{SL}_d(\R)$ having an exponential moment.

Let $\nu$ be the unique borelian probability measure on $\prob(\R^d)$ (existence and uniqueness of $\nu$ are proved in proposition~\ref{proposition:unique_mesure_invariante_contraction}).

Then, for any $M \in \R_+^\ast$, there are $n_0\in \N$, $\Delta \in \R$ and $t\in \R_+^\ast$ such that for any $n\in \N$ with $n\geqslant n_0$,
\[
\rho^{\ast n} \left( \left\{ g\in \G \middle|g\text{ is proximal and }\nu\left(B\left(V_g^+, e^{-Mn}\right)\right) \geqslant e^{-\Delta Mn} \right\}\right) \geqslant 1 - e^{-t n}\]
where $V_g^+$ is the locally attractive fixed point of $g$ in $\prob(\R^d)$ and we endowed $\prob(\R^d)$ with the usual distance (see equation~\eqref{equation:distance_projectif}).
\end{corollary*}

These two results will allow us to prove the next
\begin{theorem_annexe}\label{theorem:controle_resolvante_SL_d}
Let $\rho$ be a borelian probability measure on $\mathrm{SL}_d(\R)$ having an exponential moment and whose support generates a strongly irreducible and proximal subgroup $\Gamma$.

Let $(\Abf, \nu_\Abf)$ be a finite $\Gamma-$set endowed with the uniform probability measure and assume that the random walk on $\Abf$ defined by $\rho$ is irreducible and aperiodic.

Let $\sigma :\Gamma \times \Sb^{d-1} \times \Abf \to \R$ be the cocycle defined for any $g\in \Gamma$, $x \in \Sb^{d-1}$  and $a\in \Abf$, by
\[
\sigma(g,(x,a)) = \ln \frac{\|gx\|}{\|x\|}
\]
and let, for any $t\in \R$, $P(it)$ be the operator defined on $\cal C^0(\Sb^{d-1} \times \Abf)$ by
\[
P(it) f(x,a) = \int_\G e^{-it \sigma(g,(x,a))}f(gx,ga) \di\rho(g)
\]
Then, for any $\gamma>0$ small enough and any $t_0 \in \R_+^\ast$, there are $C,L \in \R_+$ such that for any $t\in \R$ with $|t|\geqslant t_0$,
\[
\|(I_d-P(it))^{-1} \|_{\cal C^{0,\gamma}(\Sb^{d-1} \times \Abf)} \leqslant C |t|^L
\]
\end{theorem_annexe}

\begin{proof}
This is a direct use of lemma~\ref{lemma:passage_espace_proj_rayon_spectral} and of proposition~\ref{proposition:rayons_spectraux_diophantiens} that we can apply (with the measure $\rho_\e= \frac 1 2 \delta_{\e} + \frac 1 2 \rho$) thanks to lemma~\ref{lemma:action_lipschitz_espace_proj} and lemma~\ref{lemma:contraction_zariski_dense} (remark that we can assume without any loss of generality that $p=1$ since it just changes $P(it)$ into $P^p(it)$ and a control of $(I_d-P^p(it))^{-1}$ gives a control of $(I_d-P(it))^{-1}$ as we already saw in the proof of lemma~\ref{lemme:premier_lemme_dolgopyat}).
\end{proof}

\subsection{Notations and preliminaries}
\begin{miniabstract}
We first fix the notations that we will use in this section.
\end{miniabstract}

\subsubsection{Proximal elements of \texorpdfstring{$\mathrm{SL}_d(\R)$}{G}}

Let $(\V,\|\,.\,\|)$ be a finite dimensional $\R-$vector space endowed with an euclidian norm and an orthonormal basis $(e_1, \dots, e_d)$.

\medskip
We define a distance on $\prob(\V)$ setting, for any $X= \R x,Y= \R y \in \prob(\V)$,
\begin{equation} \label{equation:distance_projectif}
d(X,Y) = \frac{\|x\wedge y\|}{\|x\| \|y\|}
\end{equation}
where we extended the scalar product of $\V$ to $\wedge^2 \V$ asking the basis $(e_i \wedge e_j)_{1\leqslant i <j\leqslant l}$ to be orthonormal.

We also define a pairing between $\prob(\V)$ and $\prob(\V^\ast)$, setting, for $X= \R x \in \V$ and $Y = \R \varphi \in \V^\ast$,
\[
\delta(X,Y) := \frac{ |\varphi(x)|}{\|\varphi\| \|x\|} = \inf_{Y' \in Y^\bot} d(X,Y')
\]
where $Y^\bot = \{ Y'=\R y'\in Y| \varphi(y')=0\}$.

We refer to the chapter $9$ in~\cite{BQred} for a proof of the next lemma that shows that we really are in the context developed in section~\ref{section:Dolgopyat_Markov}.

\begin{lemma} \label{lemma:action_lipschitz_espace_proj}
For any $g\in \G$ and any $X,Y \in \prob(\R^d)$,
\[
d(gX,gY) \leqslant \|g\|^{2d} d(X,Y)
\]
Moreover, there is $C\in \R$ such that for any $X,Y \in \prob(\R^d)$ and any $g\in \G$,
\[
\left|\sigma(g,X) - \sigma(g,Y)\right| \leqslant C \|g\|^C {d(X,Y)}
\]
where we noted, for $X=\R x \in \prob(\R^d)$ and $g\in \G$,
\[
\sigma(g,X) = \ln\frac{\|gx\|}{\|x\|}
\]
So, with the notations of the previous section, we have that $\sigma \in \cal Z^{C}_N(\prob(\R^d))$ where we noted $N(g) = \|g\|$.
\end{lemma}

An element $g$ of $\mathrm{SL}(\V)$ is said to be \emph{proximal} if it has a locally attractive fixed point in $\prob(\V)$. Equivalently, a proximal element is one having a unique eigenvalue of maximal modulus and whose eigenspace is a line. In this case, this eigenvalue is real. We note that $g$ is proximal if and only if $^t g$ is proximal in $\V^\ast$.

If $g$ is proximal, we note $V_g^+\in \prob(\V)$ the space associated to it's eigenvalue of maximal modulus and $V_g^<\in \prob(\V^\ast)$ the class defined by the $g-$invariant supplementary subspace of $V_g^+$ (or equivalently the locally attractive fixed point of $^t g$ in $\V^\ast$). In the sequel, we will always note $v_g^+$ a representative of $V_g^+$ in $\V$ and $\varphi_g^<$ a representative of $V_g^<$ in $\V^\ast$ and we will use these representatives in a way such that our formulas will not depend on their choices.

For any $g\in \G$, we note $\lambda_1(g), \dots , \lambda_d(g)$ the logarithms of the eigenvalues of $g$ sorted in decreasing order and counted with multiplicities. So, $g$ is proximal if and only if $\lambda_1(g)>\lambda_2(g)$. Moreover, if $g$ is proximal, we have by definition that $gv_g^+ = \varepsilon_1(g) e^{\lambda_1(g)} v_g^+$ for some $\varepsilon_1(g) \in \{\pm1\}$.

For an element $g \in \mathrm{SL}(\V)$, we choose a Cartan decomposition $g=k_g a_g l_g$. This means that $k_g, l_g \in \mathcal{O}(\V)$ and $a_g$ is a diagonal matrix :
\[
a_g = \left(\begin{array}{ccc} \kappa_1(g) & &0 \\ & \ddots & \\ 0& & \kappa_d(g)\end{array} \right)
\]
where the $\kappa_i(g)$ are the \emph{singular values of $g$}, and satisfy $\kappa_1(g) \geqslant \dots \geqslant\kappa_d(g)$ and are given by
\[
\kappa_i(g) =\frac{ \|\wedge^i g \|}{\|\wedge^{i-1} g\|}
\]
where we noted $\wedge^i g$ the endomorphism defined by the action of $g$ in $\bigwedge^i \V$ endowed with the scalar product induced by the one of $\V$ : this means that
\[
\wedge^i g (v_1 \wedge \dots \wedge v_i) = (gv_1) \wedge \dots \wedge (gv_i)
\]

Moreover, we note
\[
\kappa_{1,2}(g) = \frac{ \kappa_{2}(g)}{\kappa_1(g)} = \frac{ \|\wedge^2 g\|}{\|g\|^2}
\]

Finally, for an element $g\in \G$ and a choice of a Cartan decomposition $g=k_g a_g l_g$, we note
\[
x_g^M =  k_g e_1, \quad X_g^M = \R x_g^M,\quad y_g^m = \,^t l_g e_1 ^\ast  \text{ and } Y_g^m = \R y_g^m
\]

\begin{lemma} \label{lemma:controles_r_prox}
Let $\V= \R^d$ endowed with an euclidian norm, $g$ an element of $ \mathrm{SL}(\V)$, $X=\R x \in \prob(\V)$ and $Y= \R \varphi \in \prob(\V^\ast)$.

Then,
\begin{enumerate}
\item
\[
\delta(X,Y_{g}^m) \leqslant \frac{ \|gx\|}{\|g\| \|x\|} \leqslant \delta(X,Y_{g}^m) + \kappa_{1,2}(g)
\]
\item
\[
\delta(X_{ g}^M, Y) \leqslant \frac{ \|^t g \varphi\|}{\|g\| \|\varphi\|} \leqslant \delta(X_{ g}^M,Y) + \kappa_{1,2}(g)
\]
\item
\[
d(gX,X_{g}^M)\delta(X,Y_g^m) \leqslant \kappa_{1,2}(g)
\]
\end{enumerate}
\end{lemma}

\begin{proof}
As the norm is supposed to be euclidian, we can assume without any loss of generality that $g$ is the diagonal matrix $(\kappa_1(g), \dots, \kappa_d(g))$. We write $x=x_1+x_2$ with $x_1 \in \mathrm{Vect}(e_1)$ and $x_2 \in \mathrm{Vect}(e_{2}, \dots, e_d)$.

Then,
\[
gx_1 = \kappa_1(g) x_1\text{ and }\|gx_2\| \leqslant \kappa_{2}(g)\|x_2\|
\]
And so, using the fact that $\kappa_1(g) = \|g\|$, we get that
\[
\frac{\|x_1\|}{\|x\|} \leqslant \frac{ \|gx\| }{\|g\| \|x\|}\leqslant \frac{\|x_1\|}{\|x\|}+ \kappa_{1,2}(g) \frac{ \|x_2\|}{\|x\|}
\]
Finally,
\[
\frac{ \|x_1\|}{\|x\|} = d(X, \mathrm{Vect}(e_{2}, \dots, e_d)) = \delta(X, Y_g^m)
\]
and this proves the first inequalities.

The second ones can be proved in the same way if we work in the dual space.

Finally, the last one comes from the fact that
\[
d(gX, X_g^M) \delta(X,Y_{ g}^m) = \frac{\|g x_2\|}{\|gx\|} \frac{\|x_1\|}{\|x\|} \leqslant \frac{ \kappa_{2}(g)}{\kappa_1(g)} = \kappa_{1,2}(g) \qedhere
\]
\end{proof}

In the sequel, we will have to control the Cartan decomposition of products of elements of $\G$. To do so, we will use the following
\begin{lemma} \label{lemma:product_SVD}
For any $p \in \N$, $p\geqslant 2$, any $\varepsilon \in \left]0,1/4\right]$, any $g_1, \dots, g_p\in \G$ with $\kappa_{1,2}(g_i) \leqslant \varepsilon^3$, $\delta(X_{g_{i+1}}^M,Y_{g_i}^m) \geqslant 2 \varepsilon$ and $\delta(X_{g_i}^M, Y_{g_{i+1}}^m) \geqslant 2\varepsilon$, we have that
\[
\kappa_1(g_p \dots g_1) \geqslant\varepsilon^{p-1} \kappa_1(g_p) \dots \kappa_1(g_1) , \quad
\kappa_{1,2}(g_p \dots g_1)\leqslant \frac{\kappa_{1,2}(g_1) \dots \kappa_{1,2}(g_p)}{\varepsilon^{2(p-1)}} 
\]
and,
\[
d(X_{g_p \dots g_1}^M,X_{g_p}^M) \leqslant \frac{ \kappa_{1,2}(g_p)} \varepsilon, \quad d(Y_{g_p \dots g_1}^m, Y_{g_1}^m) \leqslant \frac{ \kappa_{1,2}(g_1)} \varepsilon 
\]
\end{lemma}

\begin{proof}
According to lemma~\ref{lemma:controles_r_prox}, we have that for any $(g_i)\in \G^p$ and any $x\in \R^d\setminus\{0\}$, noting $X= \R x$,
\begin{align*}
\| g_p \dots g_1 x \| &\geqslant \|g_p \| \|g_{p-1} \dots g_1x\| \delta(g_{p-1} \dots g_1 X, Y_{g_p}^m) \\
&\geqslant \|g_p\| \dots \|g_1\| \|x\| \delta(g_{p-1} \dots g_1 X, Y_{g_p}^m) \dots \delta(X,Y_{g_1}^m)
\end{align*}
Moreover, taking $x$ orthogonal to $Y_{g_1}^m $, we have that for any $l\in \lib 1, p \rib$,
\[
d(g_{l-1} \dots g_1 X, Y_{g_{l}}^m) \geqslant \frac{ l+1} l \varepsilon
\]
Indeed, this is true for $l=1$ by assumption and by induction, we have that for any $l\in \lib 1, p-1 \rib$,
\[
d(g_l \dots g_1 X, X_{g_l}^M) \leqslant \frac{\kappa_{1,2}(g_l)}{d(g_{l-1} \dots g_1 X, Y_{g_l}^m) } \leqslant \frac{l}{l+1} \varepsilon^2 \leqslant \frac{l}{l+1} \varepsilon
\]
and so,
\[
\delta(g_l \dots g_1 X, Y_{g_{l+1}}^m) \geqslant \varepsilon\left(2 - \frac{ l }{l+1}\right) = \frac{l+2}{l+1}\varepsilon
\]
This proves that
\[
\frac{\|g_p \dots g_1 x\|}{\|x\|} \geqslant \frac{p} 2 \|g_p\| \dots \| g_1\| \varepsilon^{p-1}
\]
and therefore,
\[
\kappa_1(g_p \dots g_1) \geqslant \frac{ p}2 \varepsilon^{p-1} \kappa_1(g_p) \dots \kappa_1(g_1) 
\]
Moreover, using the sub-multiplicativity of the function $\kappa_1 \kappa_2$ and that $\kappa_{1,2}(g_i) \leqslant \varepsilon^3$, we find that
\begin{align*}
\kappa_{1,2}(g_p \dots g_1)& = \frac{ \kappa_2(g_p \dots g_1) \kappa_1(g_p \dots g_1)}{\kappa_1(g_p \dots g_1)^2} \leqslant  \frac{4 }{ p^2\varepsilon^{2(p-1)}} \kappa_{1,2}(g_1) \dots \kappa_{1,2}(g_p)\\& \leqslant \frac 4 {p^2}\varepsilon^{p-1} \kappa_{1,2}(g_p)
\end{align*}
Finally, still using lemma~\ref{lemma:controles_r_prox}, we have that
\[
\delta(X, Y_{g_p \dots g_1}^m) + \kappa_{1,2}(g_p \dots g_1) \geqslant \frac p 2 \varepsilon^{p-1} 
\]
and so,
\[
\delta(X,Y_{g_p \dots g_1}^m) \geqslant \frac p 2 \varepsilon^{p-1} - \frac 4 {p^2} \varepsilon^{p+2} = \frac p2 \varepsilon^{p-1} \left( 1 - \frac{ 8}{p^3} \varepsilon^3 \right)
\]
This proves that
\[
d(g_p \dots g_1 X, X_{g_p \dots g_1}^M) \leqslant \frac{ \kappa_{1,2}(g_p \dots g_1)}{\delta(X,Y_{g_p \dots g_1}^m) } \leqslant \frac 2 p \frac{\kappa_{1,2}(g_p) }{  1 - 8 \varepsilon^3/p^3}
\]
And so,
\begin{align*}
d(X_{g_p \dots g_1}^M, X_{g_p}^M) & \leqslant \frac {p } {p+1}\frac {\kappa_{1,2}(g_p)} \varepsilon + \frac 2 p\frac{\kappa_{1,2}(g_p)}{ (1- 8 \varepsilon^3 / p^3)}  \\ & \leqslant \frac{ \kappa_{1,2}(g)}\varepsilon \left( \frac p {p+1} + \frac {2\varepsilon } {p(1- 8 \varepsilon^3 / p^3)} \right)
\end{align*}
This proves the third inequality since for any $\varepsilon \in \left]0,1/4\right]$ and $p\geqslant 2$, $\frac {2\varepsilon}{p(1-8 \varepsilon^3/p^3)}  \leqslant \frac 1 {p+1}$.

To get the control of $d(Y_{g_p \dots g_1}^m, Y_{g_1}^m)$, we do the same computations in the dual space.
\end{proof}

Next lemma will allow us to, knowing the Cartan decomposition of an element $g$ of $\G$, prove that it is proximal and have a control on $V_g^+$ and $V_g^<$.

\begin{lemma} \label{lemma:quantification_proximale}
For any $\varepsilon\in \left]0,1/4\right]$ and any element $g$ in $\G$, if $\kappa_{1,2}(g)\leqslant \varepsilon^3$ and $ \delta(X_{ g }^M, Y_{ g}^m) \geqslant 2\varepsilon$ then $g$ is proximal and
\[
d(V_g^+,X_{g}^M) \leqslant \frac{\kappa_{1,2}(g)}{\varepsilon} , \quad d(V_g^<, Y_g^m) \leqslant  \frac{\kappa_{1,2}(g)}{\varepsilon}
\]
Moreover,
\[
e^{\lambda_1(g)}\geqslant \kappa_1(g) \delta(X_g^M, Y_g^m) \text{ et }\|g_{|_{V_g^<}}\| \leqslant \frac{ 2\kappa_{2}(g)}{\varepsilon}
\]
\end{lemma}

\begin{proof}
The first three inequalities come from lemma~13.14 in~\cite{BQred}.

To prove that the norm of $g$ restricted to $V_g^<$ is controlled by $\kappa_2(g)$, we remark that, according to lemma~\ref{lemma:controles_r_prox}, for any $x\in \R^d\setminus\{0\}$, noting $X= \R x$, we have that
\[
\frac{\|gx\|}{\|x\|} \leqslant \kappa_1(g) \delta(X,Y_g^m) + \kappa_2(g)
\]
But, for any $X \in V_g^<$, we have that
\[
\delta(X, Y_g^m) \leqslant d(V_g^<, Y_g^m) \leqslant \frac{ \kappa_{1,2}(g)}{\varepsilon}
\]
This proves that for any $x \in V_g^< \setminus\{0\}$,
\[
\frac{\|gx\|}{\|x\|} \leqslant \kappa_2(g) \left( 1 + \frac 1 {\varepsilon} \right) \leqslant \frac{2 \kappa_{2}(g)}\varepsilon
\]
And thus we get the last expected inequality.
\end{proof}

From now on, we note, for any $g\in \mathrm{SL}_d(\R)$ and $X= \R x\in \prob(\R^d)$,
\begin{equation} \label{equation:sigma}
\sigma(g,X)= \ln \frac{ \|gx\|}{\|x\|}
\end{equation}

\begin{lemma}\label{lemma:contraction_proximale}
For any $\varepsilon \in \left]0,1/4\right]$, and any element $g$ in $\G$, if $\kappa_{1,2}(g)\leqslant  \varepsilon^4$ and $d(X_g^M, Y_g^m) \geqslant 2\varepsilon$, then we have that for any $X\in \prob( \R^d ) $ with $\delta(X,V_g^<) \geqslant 2\varepsilon$,
\[
\left|\sigma(g,X)- \lambda_1(g)- \ln\frac{\delta(X,V_g^<)}{\delta(V_g^+, V_g^<)} \right| \leqslant 2 \frac{ \kappa_{1,2}(g)} {\varepsilon^3}
\]
Moreover, for any $X,Y \in \prob(\R^d)$ with $\delta(X,V_g^<) , \delta(Y,V_g^<) \geqslant 2\varepsilon$, we have that
\[
d(gX,gY)  \leqslant \frac{ \kappa_{1,2}(g)}{4\varepsilon^4}
\]
\end{lemma}

\begin{proof}
Note $v_g^+, \varphi_g^<$ such that $V_g^+ = \R v_g^+$ and $V_g^< = \R \varphi_g^<$.

Then, according to the previous lemma,
\[
\delta(V_g^+, V_g^<) \geqslant 2\varepsilon - 2\frac {\kappa_{1,2}(g)}{\varepsilon} \geqslant 2 \varepsilon (1-\varepsilon^2) \geqslant \frac 3 2 \varepsilon
\]

For any $x\in \R^d$, we can write
\[
x= \frac{\varphi_g^<(x)}{\varphi_g^<(v_g^+)} v_g^+ + x- \frac{\varphi_g^<(x)}{\varphi_g^<(v_g^+)} v_g^+
\]
and so, as $g v_g^+ =\varepsilon_1(g) e^{\lambda_1(g)} v_g^+$, we get
\[
g x = \varepsilon_1(g) e^{\lambda_1(g)} \frac{\varphi_g^<(x)}{\varphi_g^<(v_g^+)} v_g^+ + g\left( x- \frac{\varphi_g^<(x)}{\varphi_g^<(v_g^+)} v_g^+ \right)
\]
But, $x- \frac{\varphi_g^<(x)}{\varphi_g^<(v_g^+)} v_g^+ \in V_g^<$ and so, according to lemma~\ref{lemma:quantification_proximale},
\[
\left\| g\left( x- \frac{\varphi_g^<(x)}{\varphi_g^<(v_g^+)} v_g^+ \right) \right\| \leqslant \frac{\kappa_2(g)}{\varepsilon} \left\|x- \frac{\varphi_g^<(x)}{\varphi_g^<(v_g^+) } v_g^+ \right\|\leqslant  \frac{\kappa_2(g)}{\varepsilon} \frac{2\|x\| }{\delta(V_g^+, V_g^<)}
\]
Thus, if $x\not=0$,
\[
e^{\lambda_1(g)} \frac{\delta(X,V_g^<)}{\delta(V_g^+, V_g^<)} \left( 1 - \|u\| \right)
 \leqslant \frac{\|gx\|}{\|x\|} \leqslant e^{\lambda_1(g)} \frac{\delta(X,V_g^<)}{\delta(V_g^+, V_g^<)} \left( 1 + \|u\| \right)
\]
with
\[
u= \frac{ \delta(V_g^+, V_g^<)}{ e^{\lambda_1(g)} \delta(X,V_g^<)\|x\|} g\left( x- \frac{\varphi_g^<(x)}{\varphi_g^<(v_g^+)} v_g^+ \right)  \text{ and } \|u\|\leqslant  \frac{1}{e^{\lambda_1(g)}\delta(X,V_g^<)}\frac{4\kappa_2(g)}{\varepsilon } 
\]
But, still using lemma~\ref{lemma:quantification_proximale}, we have that $e^{\lambda_1(g)} \geqslant 2 \|g\|\varepsilon$ and so,
\[
\frac{4\kappa_2(g)}{e^{\lambda_1(g)}\delta(X,V_g^<)\varepsilon } \leqslant \frac{2 \kappa_{1,2}(g)}{\varepsilon^2 \delta(X,V_g^<)}
\]
So, for $X$ with $\delta(X,V_g^<) \geqslant 2\varepsilon$,
we have that
\[
\|u \|\leqslant \frac{ \kappa_{1,2}(g)}{\varepsilon^3} \leqslant \varepsilon
\]
and
\[
\ln(1-\|u\|) \leqslant \sigma(g,X) - \lambda_1(g) - \ln \frac {\delta(X, V_g^<)}{ \delta(V_g^+, V_g^<)}  \leqslant \ln(1+\|u\|)
\]
and we get the expected result if we use that
\[
\|u\| \leqslant \min\left( \frac{\kappa_{1,2}(g)} {\varepsilon ^3}, \frac 1 2 \right)
\]
Finally, for any $X,Y \in \prob(\R^d)$ with $\delta(X,V_g^<) , \delta(Y,V_g^<) \geqslant 2\varepsilon$,
\begin{align*}
d(gX,gY) &= \frac{ \|\wedge ^2 g (x\wedge y) \|} {\|gx\|\|gy\|} \leqslant \frac{ \kappa_1(g) \kappa_2(g) \|x\| \|y\|} {\|gx\| \|gy\|} \leqslant \frac{ \kappa_1(g) \kappa_2(g)}{e^{-2\lambda_1(g)}} \frac{ 4 \delta(V_g^+, V_g^<)^2}{\delta(X,V_g^<) \delta(Y,V_g^<)} \\
& \leqslant \frac{ \kappa_{1,2}(g)}{4\varepsilon^4} \qedhere
\end{align*}
\end{proof}

\subsubsection{Genericness of proximal elements}

First, recall that if $\rho$ is a borelian probability measure on $\mathrm{SL}_d(\R)$ having a moment of order $1$\footnote{$\int_\G \ln\|g\| \di\rho(g)$ is finite.}, then, there are $\lambda_1,\dots, \lambda_d\in \R$, called Lyapunov exponents of $\rho$, such that $\lambda_1 + \dots + \lambda_d=0$ and for any $i\in [1,d]$,
\[
\frac 1 n \ln\| \wedge^i g_n \dots g_1 \| \xrightarrow\, \lambda_1 + \dots + \lambda_i \; \rho^{\otimes \N}-\text{a.e.}
\]
Moreover, if the support of $\rho$ generates a strongly irreducible and proximal subgroup, then $\lambda_1 > \lambda_2$ (see~\cite{GuRa85}).

In the sequel, we will have to produce elements in the support of $\rho^{\ast n}$ having good properties. To do so, we will use the following result which is proved in the chapter~12 of~\cite{BQred}.

\begin{lemma} \label{lemma:genericite_loxodromie}
Let $\rho$ be a borelian probability measure on $\mathrm{SL}_d(\R)$ having an exponential moment and whose support generate a strongly irreducible and proximal subgroup. Then, for any $\varepsilon \in \R_+^\ast$, there are $t\in \R_+^\ast$ and $n_0 \in \N$ such that for any $n\in \N$ with $n\geqslant n_0$, we have that for any $x\in \prob(\R^d)$ and any $y\in \prob((\R^d)^\ast)$,
\[
\rho^{\ast n} \left( \left\{ g\in \G \middle| \forall i \in \lib 1,d\rib,\; \left| \frac 1 n \kappa_i(g) - \lambda_i \right| \leqslant \varepsilon \right\}\right) \geqslant 1 - e^{-tn}
\]
\[
\rho^{\ast n}\left(\left\{ g\in \G \middle| \delta(x,y_g^m) \geqslant 2e^{-\varepsilon n} \right\}\right) \geqslant 1 - e^{-tn }
\]
\[
\rho^{\ast n}\left(\left\{ g\in \G \middle| d(gx, x_g^M) \leqslant e^{-(\lambda_1 - \lambda_2-\varepsilon) n} \right\}\right) \geqslant 1 - e^{-tn }
\]
\[
\rho^{\ast n}\left(\left\{ g\in \G \middle| \delta(x_g^M, y) \geqslant 2e^{-\varepsilon n} \right\}\right) \geqslant 1 - e^{-tn }
\]
\[
\rho^{\ast n}\left(\left\{ g\in \G \middle| \delta(gx,y) \geqslant 2 e^{-\varepsilon n} \right\}\right) \geqslant 1 - e^{-tn }
\]
\[
\rho^{\ast n} \left( \left\{ g\in \G\middle| \delta(x_g^M, y_g^m) \geqslant 2e^{-\varepsilon n} \right\}\right) \geqslant 1 - e^{-tn}
\]
\end{lemma}

Moreover, we have, with lemma~\ref{lemma:action_lipschitz_espace_proj}, the
\begin{lemma}\label{lemma:contraction_zariski_dense} \label{lemma:contraction_indeed}
Let $\rho$ be a borelian probability measure on $\G$ having an exponential moment and whose support generates a strongly irreducible and proximal subgroup of $\G$. Note, for $g\in \G$, $N(g) := \|g\|$.

Then, there is $\gamma \in \R_+^\ast$ such that $\prob(\R^d)$ is $(\rho,\gamma,2d, N)-$contracted.
\end{lemma}

\begin{proof}
We refer to~\cite{BL85} for a proof of this result.
\end{proof}

\subsubsection{A regularity lemma for convolution powers of measures}

In this paragraph, we prove a technical lemma about regularity of powers of convolution of measures on $\mathrm{SL}_d(\R)$. With the language we used in the previous section, this means that for a measure on $\mathrm{SL}_d(\R)$ having an exponential moment, there is a parameter $\Delta$ such that for any $n$ large enough, any point of $\supp\rho^{\ast n}$ is $\Delta-$regular at scale $e^{-t_2 n}$ except on a set of exponentially small measure.

\begin{lemma} \label{lemma:convolution_mesures_SLd}
Let $\rho$ be a borelian probability measure on $\mathrm{SL}_d(\R)$ having an exponential moment.

Then, for any $t_1, t_2 \in \R_+^\ast$, there are $n_0\in \N$ and $t_3 \in \R_+^\ast$ such that for any $n\in \N$ with $n\geqslant n_0$, we have that
\[
\rho^{\ast n} \left( \left\{ g\in \G \middle| \rho^{\ast n} \left(B(g, e^{-t_2 n})\right) \geqslant e^{-t_3 n} \right\}\right) \geqslant 1-e^{-t_1 n}
\]
\end{lemma}

\begin{proof}
Let $\varepsilon \in \R_+^\ast$ be such that $\int_\G \|g\|^\varepsilon \di\rho(g)$ is finite and fix a $n\in \N$. Using Markov's inequality we have that for any $M\in \R_+$,
\[
\rho^{\ast n}\left(\left\{ g\in \G \middle| \|g\| \geqslant e^{Mn} \right\}\right) \leqslant e^{-\varepsilon Mn} \int_\G \|g\|^\varepsilon \di\rho^{\ast n}(g) \leqslant \left(e^{-\varepsilon M}\int_\G \|g\|^\varepsilon \di\rho(g) \right)^n
\]
So, noting $\widetilde\Omega_n = \left\{ g\in \G \middle| \|g\| \leqslant e^{Mn} \right\}$, we have that
\begin{align*}
\rho^{\ast n}\left(\widetilde\Omega_n^c\right) \leqslant  \left(e^{-\varepsilon M}  \int_\G \|g\|^\varepsilon \di\rho(g) \right)^n
\end{align*}
Moreover, there is a constant $C(d)$ depending only on the dimension $d$ and $g_1, \dots g_L \in \widetilde{\Omega}_n$ such that
\[
\widetilde{\Omega}_n \subset \bigcup_{i=1}^L B(g_i,e^{-t_2 n}/2)
\]
and $L \leqslant C(d) e^{(M+t_2)d^2 n}$.
Now, we note, for $K\in \R_+^\ast$,
\[
\G_n = \left\{ g \in \{g_1, \dots, g_L\}\middle|\rho^{\ast n}(B(g,e^{-t_2n}/2)) \geqslant e^{-Kn} \right\} \text{ and }\Omega_n = \bigcup_{g \in G_n } B(g,e^{-t_2 n}/2)
\]
Then, for $h\in \Omega_n$, there is $g\in \G_n$ such that $d(g,h) \leqslant e^{-t_2 n}/2$ and so,
\[
B(h,e^{-t_2 n}) \supset B(g,e^{-t_2 n}/2)
\]
and so,
\[
\rho^{\ast n}(B(h,e^{-t_2 n})) \geqslant \rho^{\ast n} ( B(g,e^{-t_2 n}/2))\geqslant e^{-K n}
\]
Finally, as $\rho$ is a probability measure we have that
\[
1= \rho^{\ast n} \left( \widetilde\Omega_n^c \right) + \rho^{\ast n} \left( \Omega_n\right) + \rho^{\ast n} \left(\widetilde \Omega_n \setminus \Omega_n \right)
\]
But, by definition,
\[
\rho^{\ast n} \left( \widetilde\Omega_n \setminus \Omega_n \right) \leqslant L e^{-Kn}
\]
And so, for any $n\in \N$,
\begin{align*}
\rho^{\ast n}( \Omega_n) &\geqslant 1- \left(e^{-\varepsilon M} \int_\G \|g\|^\varepsilon \di\rho(g) \right)^n -Le^{-K n} \\
&\geqslant  1- \left(e^{-\varepsilon M} \int_\G \|g\|^\varepsilon \di\rho(g) \right)^n - C(d)e^{-K n}e^{(M+t_2)d^2 n}
\end{align*}
And this proves the lemma since we can choose $K$ and $M$ as large as we want.
\end{proof}

\subsection{Regular points in the projective space}

\begin{miniabstract}
In this paragraph, we study the lower regularity of the stationary measure on the projective space at fixed points of generic proximal elements.
\end{miniabstract}

In our study of the perturbation of Markov operators on compact spaces (section~\ref{section:Dolgopyat_Markov}), we used an assumption of lower regularity on the points we considered and it is time to prove that this assumption holds for the walk on the projective space.

\medskip
We remind that given a compact metric space $(\X,d)$ endowed with a borelian probability measure $\nu$, we say that a point $x\in \X$ is $\Delta-\nu$-regular at scale $r$ where $r\in \R_+^\ast$ and $\Delta \in \R_+$, if
\[
\nu(B(x,r)) \geqslant r^\Delta
\]
We are going to prove that for a strongly irreducible and proximal subgroup of $\mathrm{SL}_d(\R)$, the fixed points of proximal elements are regular.

\medskip
First of all, we have the following
\begin{lemma}
Let $\rho$ be a borelian probability measure on $\G$ whose support generates a strongly irreducible and proximal subsemigroup $T_\rho$.

Then, there is a unique borelian $P-$invariant probability measure $\nu$ on $\prob(\R^d)$.

Moreover, for any proximal $g \in T_\rho$, we have that $V_g^+ \in \supp\nu$.
\end{lemma}

\begin{proof}
Existence and unicity of $\nu$ come from~\cite{GuRa85}.

To prove the end of the lemma, note that for any proximal $g\in T_\rho$, there is $X\in \supp\nu$ such that $X\not\in V_g^<$. Indeed, if not, we would have some proximal $g\in T_\rho$ such that $\supp\nu \subset V_g^<$ but this is impossible since we assumed that $T_\rho$ is strongly irreducible.

Moreover, for any proximal $g\in T_\rho$ and any $X\not\in V_g^<$, we have that
\[
g^n X \xrightarrow\, V_g^+
\]
and, as $\supp\nu$ is closed and $T_\rho-$invariant, this proves that $V_g^+ \in \supp\nu$.
\end{proof}

We are now ready to prove the regularity of the stationary measure in the

\begin{proposition} \label{proposition:regularite_inferieure_espace_projectif}
Let $\rho$ be a strongly irreducible and proximal borelian probability measure on $\mathrm{SL}_d(\R)$, having an exponential moment.

Let $\nu$ be the unique stationary borelian probability measure on $\prob(\R^d)$.

Then, for any $m\in \N^\ast$ and any $\varepsilon \in \R_+^\ast$ small enough we have that for any $g\in \G$ with $\kappa_{1,2}(g) \leqslant \varepsilon^4$ and $\delta(X_g^M, Y_g^m) \geqslant 2\varepsilon$, $g$ is proximal and
\[
\nu(B(V_g^+, (\kappa_{1,2}(g)/ \varepsilon^4)^m) \geqslant \frac 1 2 \left( \rho^{\ast n }\left( B(g,r)\right) \right) ^m
\]
where we noted
\[
r=\kappa_1(g)^{-(1+d)m}
\]
\end{proposition}

\begin{proof}
Let $\varepsilon \in \left]0,1/4\right]$ and $g\in \G$ such that $\kappa_{1,2}(g) \leqslant \varepsilon^3$ and $\delta(X_g^M, Y_g^m) \geqslant 2 \varepsilon$.

According to lemma~\ref{lemma:quantification_proximale}, $g$ is proximal.
Moreover, using the $\rho-$stationarity of the measure $\nu$ we have that for any $m\in \N$,
\begin{align*}
\nu(B(V_g^+, \kappa_{1,2}(g)^m )) &= \int_{\X} \un_{B(V_g^+,\kappa_{1,2}(g)^m)} (X) \di\nu(X) \\ &= \int_\X \int_\G \un_{B(V_g^+, \kappa_{1,2}(g)^m)} (hX) \di\rho^{\ast mn}(h) \di\nu(X) \\
& = \int_\X \rho^{\ast mn}\left(\left\{ h \middle| hX \in B(V_g^+,\kappa_{1,2}(g)^m) \right\}\right) \di\nu(X)\\
& \geqslant \int_\X \un_{\delta(X,Y_g^m) \geqslant 2 \varepsilon} \rho^{\ast mn}\left(\left\{ h \middle| hX\in B(V_g^+,\kappa_{1,2}(g)^m) \right\}\right) \di\nu(X)
\end{align*}
But, we saw in lemma~\ref{lemma:contraction_proximale} that if $X \in \prob(\R^d)$ is such that $\delta(X,Y_g^m) \geqslant 2 \varepsilon$, then
\[
d(g^m X,V_g^+) =d(g^m X,g^m V_g^+) \leqslant \frac{  \kappa_{1,2}(g)}{4 \varepsilon^4} d(g^{m-1} X, V_g^+) \leqslant \left( \frac{  \kappa_{1,2}(g)}{4 \varepsilon^4}\right)^m
\]
where we used that $gV_g^+ = V_g^+$.

Moreover, if $r\in \left]0,1\right]$ then for any $h_1, \dots, h_m \in B(g, r)$ we have that
\[
\|g^m - h_1 \dots h_m\| \leqslant \|g-h_1\|\|g\|^{m-1} + \|h_1 \| \|g^{m-1} - h_2 \dots h_m\| \leqslant m(2\|g\|)^{m-1} r
\]
and so,
\[
d(h_1 \dots h_m X, V_g^+) \leqslant \|g^m - h_1 \dots h_m\| + d(g^m X,V_g^+) \leqslant m(2\|g\|)^{m-1} r + \left( \frac{  \kappa_{1,2}(g)}{4 \varepsilon^4}\right)^m
\]
Therefore, taking $r= \kappa_1(g)^{-(1+d)m}$ and using that $\kappa_1(g) = \|g\|$, $\kappa_2(g) \geqslant \kappa_1(g)^{1-d}$ (since $\kappa_1(g) \dots \kappa_d(g) = 1$ and $\kappa_1(g) \geqslant \kappa_2(g) \geqslant \dots \geqslant \kappa_d(g)$), $\kappa_{1,2} (g) = \kappa_2(g) / \kappa_1(g)$, we get that
\[
\frac 1 m \left( \frac{  \kappa_{1,2}(g)}{8 \|g\| \varepsilon^4}\right)^m \geqslant \kappa_1(g)^{-(1+d)m} \frac 1 m \left(\frac 1{8\varepsilon^4} \right)^m
\]
and we can assume that $\varepsilon$ is small enough so that for any $m\in \N^\ast$, $\frac 1 m (8 \varepsilon^4)^{-m} \geqslant 1$.

In this way, we get that for any $h_1, \dots, h_m \in B(g,r)$,
\[
d(h_m \dots h_1 X,V_g^+) \leqslant 2 \left( \frac{  \kappa_{1,2}(g)}{4 \varepsilon^4}\right)^m \leqslant \left( \frac{  \kappa_{1,2}(g)}{\varepsilon^4}\right)^m
\]
Thus, for any $X\in \prob(\R^d)$ such that $\delta(X,Y_g^m) \geqslant 2 \varepsilon$,
\[
\rho^{\ast mn}\left(\left\{ h \middle| hX \in B(V_g^+,(\kappa_{1,2}(g)/\varepsilon^4)^m) \right\}\right) \geqslant \left(\rho^{\ast n}\left(B\left(g, r \right)\right)\right)^m
\]
This proves that for any $g\in \G$ with $\delta(X_g^M, Y_g^m) \geqslant 2\varepsilon$ and $\kappa_{1,2}(g) \leqslant \varepsilon^3$, we have that
\begin{align*}
\nu(B(V_g^+, (\kappa_{1,2}(g)/ \varepsilon^4)^m)) &\geqslant \left(\rho^{\ast n} B(g,r) \right)^m \nu\left(\left\{ X\in \prob(\R^d) \middle| \delta(X,Y_g^m) \geqslant 2 \varepsilon \right\} \right) \\
& \geqslant \left(\rho^{\ast n} B(g,r) \right)^m \left(1 - (2 \varepsilon)^c\right)
\end{align*}
where we used the upper regularity of the measure $\nu$ (see chapter 12 in~\cite{BQred}) to have that for some constant $c$ and $\varepsilon$ small enough,
\[
\nu(X| \delta(X,Y_g^m) \geqslant 2\varepsilon)\geqslant 1-(2\varepsilon)^c. \qedhere
\]
\end{proof}

\begin{corollary}  \label{corollary:regularite_inferieure_espace_projectif}
Let $\rho$ be a strongly irreducible and proximal borelian probability measure on $\mathrm{SL}_d(\R)$ having an exponential.

Let $\nu$ be the unique $P-$invariant borelian probability measure on $\prob(\R^d)$.

Then, for any $M\in \N^\ast$, there are $n_0\in \N$, $\Delta \in \R$ and $t\in \R_+^\ast$ such that for any $n\in \N$ with $n\geqslant n_0$,
\[
\rho^{\ast n} \left( \left\{ g\in \G \middle|g\text{ is proximal and }\nu\left(B\left(V_g^+, e^{-Mn}\right)\right) \geqslant e^{-\Delta Mn} \right\}\right) \geqslant 1 - e^{-t n}\]
\end{corollary}

\begin{proof}
Let $m\in \N$, $0<\varepsilon < (\lambda_1-\lambda_2)/4$ and $\lambda= \lambda_{1} - \lambda_2 - 3\varepsilon$. According to proposition~\ref{proposition:regularite_inferieure_espace_projectif}, we have that for any $n$ large enough and any $g\in \G$ such that $\kappa_{1,2}(g) \leqslant e^{- \lambda n}$ and $\delta(X_g^M, Y_g^m) \geqslant 2e^{-\varepsilon n}$,
\[
\nu(B(g, e^{-m \varepsilon n})) \geqslant \frac 1 2 \left( \rho^{\ast n} (B(g, e^{-(\lambda_1+ \varepsilon)(1+d) mn})) \right)^m
\]
In particular, taking $m \geqslant M/\varepsilon$, we have that
\[
\nu(B(g,e^{-Mn})) \geqslant \frac 1 2 \left( \rho^{\ast n} (B(g, e^{-(\lambda_1+ \varepsilon)(1+d) mn})) \right)^m
\]
To conclude, we use the fact that such elements $g$ with $\delta(X_g^M, Y_g^m)\geqslant 2 e^{-\varepsilon n}$ and $\kappa_{1,2}(g) \leqslant e^{-\lambda n}$ are generic according to lemma~\ref{lemma:genericite_loxodromie} and so are the elements for which we have a good lower bound on the measure of $B(g,e^{-(\lambda_1+ \varepsilon)(1+d) mn})$ according to lemma~\ref{lemma:convolution_mesures_SLd}.
\end{proof}

Now, we use this regularity property to pass from a condition in proposition~\ref{proposition:Dolgopyat_markov} where the action of $\G$ on $\X$ plays a role to a condition on $\G$.
\begin{lemma} \label{lemma:passage_espace_proj_rayon_spectral}
Let $\rho$ be a borelian probability measure on $\G$ having an exponential moment and whose support generates a strongly irreducible and proximal subgroup.

Let $\X = \Sb^{d-1} \times \Abf$ and $\sigma : \G\times \prob(\R^d) \to \R$ the cocycle defined for any $g\in \G$ and $X=\R x \in \prob(\R^d)$ by
\[
\sigma(g,X) = \ln \frac{ \|gx\|}{\|x\|}
\]
that we extend to $\X$ noting for any $(x,a) \in \X$, $\sigma(g,(x,a)) = \sigma(g,\R x)$.

Let $\nu$ be the unique $P_\rho-$invariant borelian probability measure on $\X/\Hb$ where we noted $\Hb = \{I_d,\vartheta\}$ where $\vartheta$ is the involution map on $\X$ that is the antipodal application on $\Sb^{d-1}$ and the identity on $\Abf$ (existence and unicity of $\nu$ are given in proposition~\ref{proposition:unique_mesure_invariante_contraction}).

\medskip
For any $\alpha, t_0,t_2,t_3, \beta, \Delta \in \R_+^\ast$, there are $\alpha_2,C,L \in \R_+^\ast$ such that for any $t\in \R$ with $|t| \geqslant t_0$, if
\[
\| (I_d-P(it))^{-1} \|_{\cal C^{0,\gamma}(\Sb^{d-1}\times \Abf)}\geqslant C|t|^L
\]
Then,
\[
\int_\G \un_{\left\{\substack{g\text{ is proximal and }\\V_g^+\text{ is }\Delta-\nu\text{-regular at scale }|t|^{-\alpha_2}}\right\}} \left| e^{-2i|\Abf|t\lambda_1(g)} - 1 \right|^2 \di\rho^{\ast n(\beta,t)}(g) \leqslant \frac 1 {|t|^\alpha}
\]
where we noted $n(\beta,t) = \lfloor \beta \ln|t| \rfloor$.
\end{lemma}

\begin{proof}
We can assume without any loss of generality that $t_0=2$ (otherwise we set $\sigma' =t_0 \sigma/2$ and $t'=2t/t_0$).

Note that as $\Hb$ is isomorphic to $\Z/2\Z$, it's irreducible unitary representations are $1-$dimensional and the decomposition of lemma~\ref{lemme:decomposition_paire_impaire} is just the decomposition between even and odd parts.

Applying proposition~\ref{proposition:Dolgopyat_markov}, for any $\alpha_1, \beta,\Delta \in \R_+^\ast$, there are $\alpha_2,C,L$ such that if
\[
\| (I_d-P(it))^{-1} \|_{\cal C^{0,\gamma}(\Sb^{d-1}\times \Abf)} \geqslant C|t|^L
\]
then there are $\xi\in \widehat \Hb$ and a function $f\in \cal C_\xi^{0,\gamma}(\X)$ with $\|f\|_\infty \leqslant C$ and $m_\gamma(f)\leqslant C|t|$ such that for any point $x$ in $\X$ whose projection on $\X/\Hb$ is $\Delta-\nu$-regular at scale $|t|^{-\alpha_2}$, we have
\[
|f(x)| \geqslant \frac 1 2
\]
and
\[
\int_\G  \left| e^{-it\sigma(g,x)}f(gx) - f(x) \right|^2 \di \rho_{\e}^{\star n(\beta,t)} (g)\leqslant \frac {1} {|t|^{\alpha_1}}
\]
where we set $n(\beta,t) = \lfloor \beta \ln |t| \rfloor
$ and $\rho_{\e}$ is the measure associated to the lazy random walk (see paragraph~\ref{subsubsection:marche_ralentie}).

To prove the proposition, we are going to make the point $x$ in the integral depend on the point $g$ chosen with the measure $\rho_\e^{\ast n(\beta, t)}$ to take $x= V_g^+$ since this in this way we will have that $gx=x$ and $\sigma(g,x) = \lambda_1(g)$.

To do so, we will first get a control for any point $x$ and any $g$ except on a set of exponentially small measure, choose the point $x$ we want and then integrate the result. The cost of this operation is passing from $\alpha$ to $\alpha_1$.

\medskip
We note $n=n(\beta, t)$ and for $t_2,t_3 \in \R_+^\ast$ whose value will be fixed later, we note
\[
G_n = \left\{\begin{array}{c|l}\multirow{2}*{$g\in \G$} & g\text{ is proximal, }V_g^+ \text{ is }\Delta-\nu\text{-regular at scale }|t|^{-\alpha_2} \text{ and }\\& \rho^{\ast n} (B(g,e^{-t_2 n})) \geqslant e^{-t_3 n} \end{array}\right\}
\]
Then, for any $g\in G_n$ and any $x$ that is $\Delta-\nu$-regular at scale $|t|^{-\alpha_2}$,
\[
 \int_\G \un_{h\in B(g,e^{-t_2 n})} \left| e^{-it\sigma(h,x)} f(hx) - f(x) \right|^2 \di\rho^{\ast n}(h) \leqslant \frac{1}{|t|^{\alpha_1}}
\]
Using the triangular inequality, we get that
\begin{align*}
I_b(x) :&=\sqrt{\rho^{\ast n}(B(g,e^{-t_2 n}))}\left|e^{-it\sigma(g,x)} f(gx) - f(x)\right| \\&\leqslant \left( \int_{ B(g,e^{-t_2 n})} \left| e^{-it\sigma(h,x)} f(hx) - f(x) \right|^2 \di\rho^{\ast n}(h) \right)^{1/2} \\
&\retrait+ \left(\int_{B(g,e^{-t_2 n})} \left| e^{-it\sigma(g,x)} f(gx) - e^{-it\sigma(h,x)}f(hx) \right|^2\di\rho^{\ast n}(h)  \right)^{1/2} \\
& \leqslant \frac{1}{|t|^{\alpha_1/2}} + \sqrt{ C' |t| e^{-t_2 \gamma n}\rho^{\ast n}(B(g,e^{-t_2 n}))}
\end{align*}
where we use that for any $h\in B(g,e^{-t_2 n})$, we have that
\begin{align*}
\left|e^{-it\sigma(g,x)} f(gx) - e^{-it\sigma(h,x)} f(hx)\right| &\leqslant \left| e^{-it\sigma(h,x)} - e^{-it\sigma(g,x)}\right| |f(gx)|+  \left|f(hx) - f(gx)\right| \\
& \leqslant 2|t|^\gamma C \left|\sigma(g,x) - \sigma(h,x) \right|^{\gamma} + m_\gamma(f) d(hx,gx)^\gamma \\
& \leqslant C'|t| e^{-t_2 \gamma n}
\end{align*}
For some constant $C'$ depending on $\sigma$ (for more details, we refer to the computations in the proof of proposition~\ref{proposition:essential_spectral_radius}).

Thus, we get that for any $g\in G_n$ and any $x$ that is $\Delta-\nu$-regular at scale $|t|^{-\alpha_2}$,
\[
\left| e^{-it\sigma(g,x)}f(gx) -f(x) \right| \leqslant \frac{e^{t_3 n/2}}{|t|^{\alpha_1/2}} + \sqrt{ C'|t| e^{-t_2 \gamma n}}
\]
Note now, for any $x=(u,a)\in \Sb^{d-1}\times \Abf$,
\[
\varphi(x)= \prod_{a\in \Abf} f(u,a) f(-u,a)
\]
This function is $\Hb-$invariant and moreover, we have that for any $x$ whose projection on the projective space is $\Delta-\nu$-regular at scale $|t|^{-\alpha_2}$,
\begin{align*}
\left|e^{-2it|\Abf|\sigma(g,x)} \varphi(gx) - \varphi(x) \right| &\leqslant \sum_{a\in \Abf} \sum_{h\in \Hb} \left| e^{-it\sigma(g,x)} \varphi(g hx) - \varphi(hx) \right|
\\& \leqslant 2 |\Abf| \left( \frac{e^{t_3 n/2}}{|t|^{\alpha_1/2}} + \sqrt{ C'|t| e^{-t_2 \gamma n}}\right)
\end{align*}
But, we can now use that, by definition of $G_n$, $V_g^+$ is $\Delta-\nu$-regular at scale $|t|^{-\alpha_2}$ and so, we get that for any $g\in G_n$,
\begin{align*}
\left| e^{-2it|\Abf|\lambda_1(g)} - 1 \right| |\varphi(V_g^+)| &=\left|e^{-2it|\Abf| \lambda_1(g)} \varphi_1(V_g^+) - \varphi_1(V_g^+) \right| \\&\leqslant 2 |\Abf| \left( \frac{e^{t_3 n/2}}{|t|^{\alpha_1/2}} + \sqrt{ C'|t| e^{-t_2 \gamma n}}\right)
\end{align*}
where we identified $\varphi$ with a function on the projective space.

This proves, using that $|\varphi(V_g^+)| \geqslant 2^{-2|\Abf|}$, that for any $g\in G_n$,
\[
\left| e^{-2it|\Abf| \lambda_1(g)} - 1 \right| \leqslant 2^{1+2|\Abf|} |\Abf|\left( \frac{e^{t_3 n/2}}{|t|^{\alpha_1/2}} + \sqrt{ C'|t| e^{-t_2 \gamma n}}\right)
\]
Thus, we get that
\[
\int_\G \un_{G_n}(g) \left|e^{-2i|\Abf|t \lambda_1(g)} -1 \right|^2 \di\rho^{\ast n(\beta,t)}_\e(g) \leqslant  2^{1+2|\Abf|} |\Abf|\left( \frac{e^{t_3 n/2}}{|t|^{\alpha_1/2}} + \sqrt{ C'|t| e^{-t_2 \gamma n}}\right)
\]
Moreover, we note $C_\Abf =2^{1+2|\Abf|} |\Abf|$ and
\[
G_n^1 := \left\{g\in \G \middle| g\text{ is proximal, }V_g^+ \text{ is }\Delta-\nu\text{-regular at scale }|t|^{-\alpha_2} \right\}
\]
We now have that
\[
\int_{G_n^1} \left|e^{-2i|\Abf|t \lambda_1(g)} -1 \right|^2 \di\rho^{\ast n(\beta,t)}_\e(g) \leqslant \int_{G_n} \left|e^{-2i|\Abf|t \lambda_1(g)} -1 \right|^2 \di\rho^{\ast n(\beta,t)}_\e(g) + 2 \rho^{\ast n}(G_n^1 \setminus G_n)
\]
And so, using lemma~\ref{lemma:convolution_mesures_SLd}, we get that for any $t_1,t_2$, there is $t_3$ such that for any $n$ large enough,
\[
\int_{G_n^1} \left|e^{-2i|\Abf|t \lambda_1(g)} -1 \right|^2 \di\rho^{\ast n(\beta,t)}_\e(g) \leqslant C_\Abf \left( \frac{e^{t_3 n/2}}{|t|^{\alpha_1/2}} + \sqrt{ C'|t| e^{-t_2 \gamma n}}\right) +2 e^{-t_1 n}
\]
So taking $t_1$ and $t_2$ large enough (this imposes a value for $t_3$), then taking $\alpha_1$ large enough and using the fact that $|t|\geqslant 2$, we get that
\[
\int_{G_n^1} \left|e^{-2i|\Abf|t \lambda_1(g)} -1 \right|^2 \di\rho^{\ast n(\beta,t)}_\e(g) \leqslant \frac 1 {|t|^\alpha}
\]
and this is what we intended to prove.
\end{proof}

\subsection{Diophantine properties of the lengths of translations}

\begin{miniabstract}
In this paragraph, we prove that for strongly irreducible and proximal measures on $\mathrm{SL}_d(\R)$, the logarithms of the spectral radii of generic elements satisfy a diophantine condition of the same kind of the one used by Carlsson in~\cite{Car83} in the study of the walk on $\R$.
\end{miniabstract}

Let $\rho$ be a borelian probability measure on $\G=\mathrm{SL}_d(\R)$. For $g \in \G$, we note $\lambda_1(g)$ the logarithm of the spectral radius of $g$. As in the study of the renewal theorem on $\R$, the study of the renewal theorem on $\R^d$ requires a diophantine assumption. However, we are going to prove that it is always satisfied for measures whose support generate a strongly irreducible and proximal subgroup.

First of all, we are going to get an estimation of the difference between $\lambda_1(gh) $ and $\lambda_1(g)  + \lambda_1(h)$ when $g$ and $h$ are proximal elements in generic position.

\begin{lemma} \label{lemma:erreur_rayon_spectral}
There are $c_1, c_2 \in \R_+^\ast$ such that for any $\varepsilon \in \left]0,c_1\right]$ and any $g,h\in \G$ with $\kappa_{1,2}(g), \kappa_{1,2}(h) \leqslant \varepsilon^4$, $\delta(X_g^M, Y_g^m) \geqslant 2 \varepsilon$, $ \delta(X_h^M, Y_h^m) \geqslant 2 \varepsilon$, $\delta(X_g^M,Y_h^m) \geqslant 2 \varepsilon$ and $\delta(X_h^M, Y_g^m) \geqslant 2\varepsilon$, we have that
\[
\left|\lambda_1(g) + \lambda_1(h) - \lambda_1(gh) - \ln \frac{\delta(V_h^+, V_h^<) \delta(V_g^+, V_g^<)}{\delta(V_g^+, V_h^<) \delta(V_h^+, V_g^<)} \right| \leqslant c_2 \left( \frac{\kappa_{1,2}(g)}{\varepsilon^3} + \frac{ \kappa_{1,2}(h)}{\varepsilon^3} \right)
\]
\end{lemma}

\begin{remark}
According to lemma~\ref{lemma:genericite_loxodromie}, there are many pairs of elements $g$ and $h$ that satisfy the assumption of the lemma in the support of $\rho^{\ast n}$.
\end{remark}

\begin{proof}
We take at first $c_1 = 1/4$ and $c_2= 1$.

First, according to lemma~\ref{lemma:product_SVD}, we have that
\[
\kappa_{1,2}(gh) \leqslant \frac{ \kappa_{1,2}(g)\kappa_{1,2}(h)}{\varepsilon^2},\quad d(X_{gh}^M,X_g^M) \leqslant \frac{ \kappa_{1,2}(g)}{\varepsilon} \text{ and }d(Y_{gh}^m,Y_h^m) \leqslant \frac{ \kappa_{1,2}(h)}{\varepsilon}
\]
and so,
\[
\kappa_{1,2}(gh) \leqslant \varepsilon^4 \text{ and } \delta(X_{gh}^M, Y_{gh}^m) \geqslant 2 \varepsilon (1-\varepsilon) \geqslant \frac 3 2 \varepsilon
\]
We note $\varepsilon' = \frac 3 4 \varepsilon$ and so we have that $d(X_{gh}^M, Y_{gh}^m) \geqslant 2 \varepsilon'$ and $\kappa_{1,2}(gh) \leqslant \frac 1 4\left(\frac 4 3\right)^3 \varepsilon'^3 \leqslant \varepsilon'^3$.

Thus, using lemma~\ref{lemma:quantification_proximale}, $gh$ is proximal and
\[
d(V_{gh}^+, V_g^+) \leqslant d(V_{gh}^+, X_{gh}^M) + d(X_{gh}^M, X_g^M) + d(X_g^M, V_g^+) \leqslant  \frac{ \kappa_{1,2}(gh)}{\varepsilon'} +  \frac{2\kappa_{1,2}(g)} \varepsilon 
\]
In the same way we get
\begin{align*}
d(V_{gh}^<, V_h^<) &\leqslant d(V_{gh}^<,Y_{gh}^m) + d(Y_{gh}^m, Y_h^m) + d(Y_h^m, V_h^<) \leqslant \frac{ \kappa_{1,2}(gh)}{\varepsilon'} + 2\frac{\kappa_{1,2}(h)} \varepsilon \\
& \leqslant \frac 4 3 \varepsilon^3 + 2 \varepsilon^2 \leqslant 3 \varepsilon^2
\end{align*}
and so,
\[
\delta(V_h^+, V_{gh}^<) \geqslant 2 \varepsilon - 3 \varepsilon^2 = 2 \varepsilon (1 - 3 \varepsilon/2) \geqslant \frac \varepsilon 2
\]
So we note $\varepsilon'' = \frac \varepsilon 4$ and then, $\delta(V_{h}^+, V_{gh}^<) \geqslant 2 \varepsilon''$ and $\kappa_{1,2}(gh) \leqslant 64 c_1 \varepsilon''^3$ so, assuming that $c_1 \leqslant \frac 1{64}$, according to lemma~\ref{lemma:contraction_proximale}, we have that
\[
\left|\sigma(gh, V_h^+)  - \lambda_1(gh) - \ln \frac{\delta(V_h^+,V_{gh}^<)}{\delta(V_{gh}^+,V_{gh}^<)} \right| \leqslant 2  \frac{ \kappa_{1,2}(gh)}{ \varepsilon''^3} \leqslant 2 \frac{ \kappa_{1,2}(h) \varepsilon}{\varepsilon''^3}
\]
Moreover, using the cocycle relation and the fact that, by definition of $V_h^+$, $\sigma(h,V_h^+) = \lambda_1(h)$, we get that
\[
\sigma(gh,V_h^+) = \sigma(g,hV_h^+)+ \sigma(h,V_h^+) = \sigma(g,V_h^+) + \lambda_1(h)
\]
and finally, we also have, still with lemma~\ref{lemma:contraction_proximale}, that
\[
\left|\sigma(g,V_h^+) - \lambda_1(g) -  \ln \frac{\delta(V_h^+,V_{g}^<) }{\delta(V_g^+, V_g^<) } \right| \leqslant  2 \frac{ \kappa_{1,2}(g)}{\varepsilon^3}
\]
This leads to
\[
\left|\lambda_1(g) + \lambda_1(h) - \lambda_1(gh) - \ln \frac{\delta(V_h^+, V_{gh}^<) \delta(V_g^+, V_g^<)}{\delta(V_{gh}^+, V_{gh}^<) \delta(V_h^+, V_g^<)} \right| \leqslant 2 \frac{ \kappa_{1,2}(g)}{\varepsilon^3} + 2^7 \frac{ \kappa_{1,2}(h) }{\varepsilon^2}
\]
To conclude, remark that
\[
\left|\delta(V_h^+, V_{gh}^<) - \delta(V_h^+, V_h^<) \right| \leqslant d(V_h^<, V_{gh}^<) \leqslant 3 \varepsilon^2
\text{ and }
\delta(V_h^+, V_{gh}^<) \geqslant\frac  \varepsilon 2
\]
so for $\varepsilon$ small enough, we have that for some constant $c_3$,
\[
\left| \ln \frac{\delta(V_h^+, V_{gh}^<)}{\delta(V_h^+, V_h^<)} \right| \leqslant  \frac{c_3}\varepsilon d(V_h^<, V_{gh}^<) \leqslant \frac{c_3}\varepsilon \left( \frac{ \kappa_{1,2}(gh)}{\varepsilon'} + 2\frac{\kappa_{1,2}(h)} \varepsilon \right)
\]
We can control $|\ln \delta(V_{gh}^+, V_{gh}^<) - \ln \delta(V_g^+, V_h^<)|$ with the same kind of ideas and get the expected result.
\end{proof}

In the sequel, we will need a result about the continuity of the Cartan decomposition. This will be the following
\begin{lemma} \label{lemma:SVD_puissance}
Let $g\in \G$.

Then, for any $h\in \G$,
\[
d\left(X_h^M, X_g^M\right) \leqslant  \left(2\|g-h\| + \kappa_{1,2}(g) \right), \text{ and } d \left(Y_h^m ,Y_g^m\right) \leqslant \left(2\|g-h\| + \kappa_{1,2}(g) \right)
\]
Moreover, there are constants $c_1,c_2 \in \R_+^\ast$ such that for any $p\in \N^\ast$, any $\varepsilon \in \left]0,c_1\right]$, any $g$ in $\G$ with $\kappa_{1,2}(g) \leqslant  \varepsilon^3$ and $\delta(X_g^M, Y_g^m) \geqslant 2\varepsilon$, any $r\in \R_+^\ast$ with $r\leqslant \varepsilon^2 \left( \frac{ \kappa_{1,2}(g)}{\varepsilon^2}\right)^p$ and any element $f\in B(g^p, r)$,
\[
\delta(X_f^M, Y_f^m) \geqslant \varepsilon,\quad d(X_f^M, X_g^M) \leqslant 2\frac {\kappa_{1,2}(g)}\varepsilon, \quad d(Y_f^m, Y_g^m) \leqslant 2\frac {\kappa_{1,2}(g)} \varepsilon
\]
\[
\quad \kappa_1(f) \geqslant \frac 1 2 \varepsilon^{p-1} \kappa_1(g)^p , \quad \kappa_{1,2}(f) \leqslant \frac{ 16}{ \varepsilon^{2(p-1)}} \kappa_{1,2}(g)^p 
\]
\[
d(V_f^+, V_{g}^+)  \leqslant c_2 \frac{ \kappa_{1,2}(g)^p}{\varepsilon^{2p-1}} \text{ et }d(V_f^<, V_{g}^<) \leqslant c_2 \frac{ \kappa_{1,2}(g)^p}{\varepsilon^{2p-1}}
\]
\end{lemma}

\begin{proof}
We are going to evaluate $g ^t l_h e_1$ in two ways.

First,
\begin{align*}
\kappa_1(h) x_h^M &= h ^t l_h e_1 = g ^tl_h e_1 + (h-g) ^tl_h e_1  \\
&= \kappa_1(g) \langle ^t l_h e_1, ^t l_g e_1\rangle x_g^M + u+  (h-g) ^tl_h e_1
\end{align*}
for some $u$ such that $\|u\| \leqslant \kappa_2(g)$.

So,
\begin{align*}
\kappa_1(h)\|x_h^M -\langle ^t l_h e_1, ^t l_g e_1\rangle x_g^M\| &\leqslant |\kappa_1(g) - \kappa_1(h)| + \|u\| + \|g-h\| \\
&\leqslant \kappa_2(g) + 2\|g-h\|
\end{align*}
and as $\|x_h^M\|=1= \|x_g^M\|$, we can deduce that
\[
d(X_g^M, X_h^M) = \|x_g^M \wedge x_h^M\| = \left\| \left(x_h^M - \langle ^tl_h e_1, ^tl_g e_1 \rangle x_g^M \right) \wedge x_g^M \right\| \leqslant \kappa_{1,2}(g) + 2 \frac{ \|g-h\|}{\|g\|}
\]
This proves the first part of the lemma since for any $g\in \G$, $\|g\|=\kappa_1(g)\geqslant 1$.

To get the control of $d(Y_g^m ,Y_h^m)$, we do the same computations in the dual space.

\medskip
To prove the end of the lemma, we note that according to lemma~\ref{lemma:product_SVD}, we have that
\[
\kappa_1(g^p) \geqslant \varepsilon^{p-1} \kappa_1(g)^p, \quad \kappa_{1,2}(g^p) \leqslant \frac 1 { \varepsilon^{2(p-1)}} \kappa_{1,2}(g)^p
\]
and
\[
d(X_{g^p}^M, X_g^M) \leqslant \frac{ \kappa_{1,2}(g)}\varepsilon, \quad d(Y_{g^p}^m, Y_g^m) \leqslant \frac{ \kappa_{1,2}(g)} \varepsilon
\]

Thus, for $f\in B(g^p, r)$,
\begin{align*}
\delta(X_f^M, Y_f^m) &\geqslant 2 \varepsilon - d(X_{g^p}^M, X_g^M) - d(Y_{g^p}^m, Y_g^m)- d(X_f^M, X_{g^p}^M) - d(Y_f^m, Y_{g^p}^m) \\
&\geqslant 2\varepsilon -2 \frac{ \kappa_{1,2}(g)}{\varepsilon} -2 \left(2r + \kappa_{1,2}(g^p) \right) \\
& \geqslant \varepsilon(2 - 6\varepsilon )
\end{align*}
and so, for maybe some smaller $c_1$, we have that
\[
\delta(X_f^M, Y_f^m) \geqslant  \varepsilon
\]
Moreover,
\[
\kappa_1(f) \geqslant \kappa_1(g^p) - r  \geqslant \varepsilon^{p-1} \kappa_1(g)^p - \varepsilon^{p+2} \geqslant  \varepsilon^{p-1} \kappa_1(g)^p \left( 1- \frac{ \varepsilon^3 }{\kappa_1(g)^p}\right)
\]
And using that $\kappa_1(g)\geqslant 1$ and $\varepsilon \leqslant c_1$, we get that if $c_1$ is small enough,
\[
\kappa_1(f) \geqslant  \varepsilon^{p-1} \kappa_1(g)^p 
\]
Moreover, using the equality $\kappa_1(g) \kappa_2(g) = \kappa_1(\wedge ^2 g)$, we have that
\[
\kappa_1(f) \kappa_2(f) \leqslant \kappa_1(g^p) \kappa_2(g^p) + \| \wedge^2 g^p - \wedge^2 f\| \leqslant \kappa_1(g^p) \kappa_2(g^p) + (\kappa_1(g^p) + \kappa_1(f))\|f-g^p\|
\]
so,
\[
\kappa_{1,2}(f) \leqslant \frac{ \kappa_1(g^p)\kappa_2(g^p)}{\kappa_1(f)^2} + \frac{ \kappa_1(g^p) + \kappa_1(f)}{\kappa_1(f)^2}\|g^p-f\| \leqslant \frac{16}{ \varepsilon^{2(p-1)}} \kappa_{1,2}(g)^p \leqslant 16 \varepsilon^{p+2}
\]
Therefore, is $c_1$ is small enough, we get that $\kappa_{1,2}(f) \leqslant (\varepsilon/2)^3$, and so, using to lemma~\ref{lemma:quantification_proximale} we find that for any $f\in B(g^p,r)$, $f$ is proximal and
\[
d(V_f^+, X_f^M) \leqslant 2 \frac{ \kappa_{1,2}(f)} \varepsilon \text{ et }d(V_f^<, Y_f^m) \leqslant  2 \frac{ \kappa_{1,2}(f)} \varepsilon
\]
Thus, using that $V_{g^p}^+= V_g^+$, we find that
\[
d(V_f^+, V_g^+) \leqslant d(V_f^+,X_f^M) + d(X_f^M, X_{g^p}^M) + d(X_{g^p}^M, V_g^+) \leqslant c_2 \frac{ \kappa_{1,2}(g)^p}{\varepsilon^{2p-1}}
\]
for some universal constant $c_2$.

Working in the same way in the dual space, we obtain the inequalities for $d(V_f^<, V_g^<)$ and finish the proof of the lemma.
\end{proof}

We are now ready to evaluate the difference between the logarithms of the spectral radii of well chosen elements of $\G$. We would like to take two proximal elements $g$ and $h$ and study the elements $g^p gh$ on side and $g^p$ and $gh$ on the other side (as in~\cite{Qu05}). However, as we don't want to study only purely atomic measures on $\mathrm{SL}_d(\R)$, we have to take, not $g^p$ but an element $f$ in a small neighbourhood. This is what we do in next

\begin{lemma} \label{lemma:erreur_produit_fgh}
There are constants $c_1, c_2,c_3$ such that for any $p\in \N^\ast$, any $\varepsilon \in \left]0,c_1\right]$, any $g \in \G$ with $\kappa_{1,2}(g) \leqslant  \varepsilon^5$ and $\delta(X_g^M, Y_g^m) \geqslant 2 \varepsilon$ we have that for any $h \in \G$ with $\kappa_{1,2}(h) \leqslant \varepsilon^3 $, $\delta(X_h^M, Y_h^m) \geqslant 2\varepsilon$, $\delta(X_h^M, Y_g^m) \geqslant 2\varepsilon$, $\delta(X_g^M, Y_h^m) \geqslant2 \varepsilon$ and any $f\in \G$ such that $\|g^p-f\| \leqslant \varepsilon^2 \left( \frac {\kappa_{1,2}(g)}\varepsilon\right)^{p} $, we have that
\[
\left|\lambda_1(fgh) - \lambda_1(f) - \lambda_1(gh) - \ln \frac{\delta(V_g^+ , V_g^<) \delta( g V_h^+, V_h^<)}{\delta(V_g^+, V_h^<) \delta (g V_h^+, V_g^<)} \right| \leqslant c_2\left( \frac{\kappa_{1,2}(g)^p}{\varepsilon^{2p}} + \frac{ \kappa_{1,2}(h)} {\varepsilon} \right)
\]
Moreover, we note $\pi_g$ the projection onto $V_g^<$ parallel to $V_g^+$ and we have that if $X_h^M \not=V_g^+$, $d(g\pi_g X_h^M, Y_h^m) \geqslant 2\varepsilon$, $d(X_g^M, X_h^M) \geqslant 2\varepsilon$ and $\kappa_{1,2}(h) \kappa_{1}(g) \leqslant \frac 1 2 \varepsilon^3$, then
\[
\frac{ \varepsilon^3}{c_3 \kappa_1(g)^d} \leqslant\left|\ln \frac{\delta(V_g^+ , V_g^<) \delta( g V_h^+, V_h^<)}{\delta(V_g^+, V_h^<) \delta (g V_h^+, V_g^<)}\right| \leqslant c_3 \frac{ \kappa_{1,2}(g)} {\varepsilon^5}
\]
\end{lemma}

\begin{proof}
We want to apply lemma~\ref{lemma:erreur_rayon_spectral} to$f$ and $gh$. To do so, we are going to prove at first that $gh$ is proximal.

According to lemma~\ref{lemma:product_SVD}, we have that
\[
\kappa_{1,2}(gh) \leqslant \frac{\kappa_{1,2}(g) \kappa_{1,2}(h)}{\varepsilon^2} \leqslant \varepsilon^4
\]
and
\[
d(X_{gh}^M, X_g^M) \leqslant \frac{  \kappa_{1,2}(g)}\varepsilon \text{ and }d(Y_{gh}^m,Y_h^m) \leqslant \frac{  \kappa_{1,2}(h)}\varepsilon 
\]
This proves that
\[
\delta(X_{gh}^M , Y_{gh}^m) \geqslant \delta(X_g^M, Y_h^m) - d(X_{gh}^M, X_g^M) - d(Y_h^m, Y_{gh}^m) \geqslant 2\varepsilon (1- \varepsilon) \geqslant \frac 3 2 \varepsilon
\]
So, for $c_1$ small enough, $gh$ satisfy the assumptions of lemma~\ref{lemma:quantification_proximale} and~\ref{lemma:contraction_proximale} with $\varepsilon'= \varepsilon/2$. So, $gh$ is proximal and
\[
d(V_{gh}^+, g V_h^+) = d(gh V_{gh}^+, gh V_h^+) \leqslant \frac{ \kappa_{1,2}(gh)}{4 \varepsilon'^4} \leqslant c_2 \frac{ \kappa_{1,2}(h)}{\varepsilon}
\]
and
\[
d(V_{gh}^<, V_h^<) \leqslant d(V_{gh}^<,Y_{gh}^m) + d(Y_{gh}^m, Y_h^m) + d(V_{h}^<, Y_h^m) \leqslant 2\frac{ \kappa_{1,2}(h)}{\varepsilon} + \frac{ \kappa_{1,2}(gh)}\varepsilon 
\]

Moreover, according to lemma~\ref{lemma:SVD_puissance}, for $f \in \G$ such that $\| g^p - f\| \leqslant \varepsilon^2 \left( \frac {\kappa_{1,2}(g)}\varepsilon\right)^{p} $, we have that
\[
\delta(X_f^M, Y_f^m) \geqslant \varepsilon,\quad d(X_f^M, X_g^M) \leqslant 2\frac {\kappa_{1,2}(g)}\varepsilon, \quad d(Y_f^m, Y_g^m) \leqslant 2\frac {\kappa_{1,2}(g)} \varepsilon
\]
Moreover,
\[
 \kappa_1(f) \geqslant \frac 12 \varepsilon^{p-1} \kappa_1(g)^p , \quad \kappa_{1,2}(f) \leqslant \frac{ 16}{\varepsilon^{2(p-1)}} \kappa_{1,2}(g)^p 
\]
and
\[
d(V_f^+, V_{g}^+)  \leqslant c_2 \frac{ \kappa_{1,2}(g)^p}{\varepsilon^{2p-1}} \text{ et }d(V_f^<, V_{g}^<) \leqslant c_2 \frac{ \kappa_{1,2}(g)^p}{\varepsilon^{2p-1}}
\]
Finally,
\begin{align*}
\delta(X_f^M, Y_{gh}^m) &\geqslant \delta(X_g^M, Y_h^m) - d(X_f^M, X_g^M) - d(Y_h^m, Y_{gh}^m) \\
& \geqslant \varepsilon\left( 2  - 2 \frac{\kappa_{1,2}(g)}{\varepsilon} - \frac{\kappa_{1,2}(h)}{\varepsilon} \right) \\
& \geqslant \varepsilon\left( 2 - 3 \varepsilon^2\right)
\end{align*}
So, is $c_1$ is small enough, we have that $\delta(X_f^M, Y_{gh}^m) \geqslant \varepsilon $ and also, $\delta(X_{gh}^m , Y_g^m) \geqslant \varepsilon$. Thus, according to lemma~\ref{lemma:erreur_rayon_spectral},
\[
\left|\lambda_1(f) + \lambda_1(gh) - \lambda_1(fgh) - \ln \frac{\delta(V_{gh}^+, V_{gh}^<) \delta(V_f^+, V_f^<)}{\delta(V_f^+, V_{gh}^<) \delta(V_{gh}^+, V_f^<)} \right| \leqslant c_2 \left( \frac{\kappa_{1,2}(f)}{\varepsilon^2} + \frac{ \kappa_{1,2}(gh)}{\varepsilon^2} \right) 
\]
This finishes the proof of the first part of the lemma since we also have seen controls on $d(V_{gh}^+, gV_h^+), d(V_f^+, V_g^+), d(V_f^<,V_g^<)$ and $ d(V_{gh}^<, V_h^<)$.

To prove the second part, we use the equality
\begin{flalign}
\frac{\delta(V_g^+ , V_g^<) \delta( g V_h^+, V_h^<)}{\delta(V_g^+, V_h^<) \delta (g V_h^+, V_g^<)} &= \left|\frac{\varphi_g^<(v_g^+)\varphi_h^<(gv_h^+)}{\varphi_h^<(v_g^+)\varphi_g^<(gv_h^+)} \right| \notag\\&=\left| 1 + \frac{\varphi_g^<(v_g^+) \varphi_h^<(gv_h^+) - \varphi_h^<(v_g^+) \varphi_g^<(g v_h^+)} { \varphi_h^<(v_g^+) \varphi_g^<(gv_h^+)} \right| \label{equation:erreur_produit_fgh_1}
\end{flalign}
And, as
\[
v_h^+= \frac{\varphi_g^<(v_h^+)}{\varphi_g^<(v_g^+)} v_g^+ + v_h^+ - \frac{\varphi_g^<(v_h^+)}{\varphi_g^<(v_g^+)} v_g^+ = \frac{\varphi_g^<(v_h^+)}{\varphi_g^<(v_g^+)} v_g^+ + \pi_g(v_h^+)
\]
where we noted $\pi_g$ the projection on $V_g^<$ parallel to $V_g^+$, we have that
\[
\varphi_g^<(g v_h^+) = \varepsilon_1(g) e^{\lambda_1(g)} \varphi_g^<(v_h^+)
\]
and
\[
\varphi_h^<(g v_h^+) = \varepsilon_1(g)  e^{\lambda_1(g)} \frac{ \varphi_g^<(v_h^+)}{\varphi_g^<(v_g^+) } \varphi_h^<(v_g^+) + \varphi_h^< \left( g \pi_g(v_h^+)\right)
\]
This proves that
\[
\varphi_g^<( v_g^+) \varphi_h^<(g v_h^+) - \varphi_h^<(v_g^+) \varphi_g^<(g v_h^+) = \varphi_g^<(v_g^+) \varphi_h^<(g\pi_g v_h^+)
\]
And so, replacing in equation~\eqref{equation:erreur_produit_fgh_1}, we get that
\[
\frac{\delta(V_g^+ , V_g^<) \delta( g V_h^+, V_h^<)}{\delta(V_g^+, V_h^<) \delta (g V_h^+, V_g^<)} = \left|1 +\varepsilon_1(g) e^{-\lambda_1(g)}\frac{\varphi_g^<(v_g^+) \varphi_h^<(g\pi_g v_h^+)}{ \varphi_h^<(v_g^+) \varphi_g^<(v_h^+)} \right| 
\]
Moreover, we use that $\pi_g v_h^+ \in V_g^<$, to get, with lemma~\ref{lemma:quantification_proximale} that
\[
\| g \pi_g v_h^+\| \leqslant \frac{\kappa_2(g)} \varepsilon \|\pi_g\| \| v_h^+\| \leqslant \frac{\kappa_2(g)} \varepsilon \frac{ 2}{\delta(V_g^+, V_g^<)}\|v_h^+\| \leqslant \frac{ 2 \kappa_2(g)}{\varepsilon^2} \|v_h^+\|
\]
and so, using that $e^{\lambda_1(g)} \geqslant 2 \kappa_1(g) \varepsilon$, we find
\[
\left| e^{-\lambda_1(g)}\frac{\varphi_g^<(v_g^+) \varphi_h^<(g\pi_g v_h^+)}{ \varphi_h^<(v_g^+) \varphi_g^<(v_h^+)} \right| \leqslant \frac{ \kappa_{1,2}(g) }{\varepsilon^3}\frac{\|\varphi_g^<\| \|v_g^+\| \|\varphi_h^<\|  \|v_h^+\|}{ |\varphi_h^<(v_g^+) \varphi_g^<(v_h^+)|} \leqslant \frac{ \kappa_{1,2}(g)}{5\varepsilon^4}
\]

We can also compute
\begin{flalign}
\left| \frac{\varphi_g^<(v_g^+) \varphi_h^<(g\pi_g v_h^+)}{ \varphi_h^<(v_g^+) \varphi_g^<(v_h^+)} \right| & =  \delta(V_g^+, V_g^<) \delta(g \pi_g V_h^+, V_h^<) \frac{ \|\varphi_g^<\| \|v_g^+\| \|\varphi_h^<\|\|g\pi_g v_h^+\|}{|\varphi_h^<(v_g^+) \varphi_g^<(v_h^+)|} \notag \\
& \geqslant \delta(V_g^+, V_g^<) \delta(g\pi_g V_h^+, V_h^<)  \frac{ \|g\pi_g v_h^+\|}{\|v_h^+\|} \notag \\
& \geqslant \delta(V_g^+, V_g^<) \delta(g\pi_g V_h^+, V_h^<) \frac 1 {\|g^{-1}\|} \frac{  \|\pi_g v_h^+\|}{\|v_h^+\|} \label{equation:erreur_produit_fgh_2}
\end{flalign}
And finally, as $v_g^+ \wedge v_h^+ = v_g^+ \wedge \pi_g v_h^+$, we get that
\[
d(V_g^+, V_h^+) = \frac{\|v_g^+ \wedge v_h^+\|}{\|v_g^+\| \|v_h^+\|} = \frac{ \|v_g^+ \wedge \pi_g v_h^+\|}{\|v_g^+\| \|v_h^+\|} \leqslant \frac{ \|\pi_g v_h^+\|}{\|v_h^+\|}
\]
and so, using inequality~\eqref{equation:erreur_produit_fgh_2}, this proves that
\[
\left| \frac{\varphi_g^<(v_g^+) \varphi_h^<(g\pi_g v_h^+)}{ \varphi_h^<(v_g^+) \varphi_g^<(v_h^+)} \right| \geqslant \frac{ 1 }{\|g^{-1}\|} \delta(V_g^+, V_g^<) \delta(g\pi_g V_h^+, V_h^<) d(V_g^+, V_h^+)
\]
To conclude, note that $\kappa_d(g) \dots \kappa_1(g) = \det(g)=1$ and $\kappa_d(g) = \|g^{-1}\|^{-1}$ and so $\|g^{-1} \|^{-1} \geqslant \kappa_1(g)^{1-d}$.

Moreover,
\begin{align*}
\delta(g\pi_g V_h^+, V_h^<) &\geqslant \delta(g \pi_g x_h^M, y_h^m) - d(g\pi_g V_h^+, g\pi_g x_h^M) - d(V_h^<, y_h^m) \\&\geqslant 2\varepsilon - \|g\| \|\pi_g\| d(V_h^+, x_h^M) - d(V_h^<, y_h^m) \\&\geqslant 2\varepsilon - 2 \frac{\kappa_{1}(g)}{\varepsilon} \frac{ \kappa_{1,2}(h)}{\varepsilon} \geqslant 2\varepsilon \left( 1- \frac{ \kappa_1(g) \kappa_{1,2}(h)}{\varepsilon^3} \right) \geqslant \varepsilon
\end{align*}
And
\[
d(V_g^+, V_h^+) \geqslant d(X_g^M, X_h^M) - d(V_g^+, X_g^M) - d(V_h^+, X_h^M) \geqslant 2\varepsilon - \frac{ \kappa_{1,2}(g)}{\varepsilon} - \frac{ \kappa_{1,2}(h)}{\varepsilon} \geqslant \varepsilon
\]
and so, using the fact that $e^{-\lambda_1(g)} \geqslant \kappa_1(g)^{-1}$, we find that
\[
\left| e^{-\lambda_1(g)}\frac{\varphi_g^<(v_g^+) \varphi_h^<(g\pi_g v_h^+)}{ \varphi_h^<(v_g^+) \varphi_g^<(v_h^+)} \right| \geqslant \frac{\varepsilon^3 }{\kappa_1(g)^{d}}
\]
And this is what we intended to prove.
\end{proof}

Lemma~\ref{lemma:erreur_produit_fgh} proved that if we make good assumptions of proximality and transversality on the elements $g$ and $h$ in $\G$, then we have a good control on $\lambda_1(fgh) - \lambda_1(f) - \lambda_1(gh)$ for $f$ in a small neighbourhood of $g^p$. Using lemma~\ref{lemma:genericite_loxodromie} we will get that those elements $g$ and $h$ are generic and this leads to the following
\begin{lemma} \label{lemma:controle_difference_rayons_spectraux}
Let $\rho$ be a borelian probability measure on $\mathrm{SL}_d(\R)$ having an exponential moment and whose support generates a strongly irreducible and proximal subgroup.

Let $\nu$ be the unique borelian stationary measure on $\prob(\R^d)$.

Then, there are $n_0,p \in \N$ et $c_1,c_2 \in \R_+^\ast$ such that for any $\alpha_2 \in \R_+^\ast$, there are $\Delta,c_3 \in \R_+^\ast$ such that for any $n\in \N$ with $n\geqslant n_0$,
\[
\rho^{\ast pn} \otimes \rho^{\ast pn} \left( \left\{\begin{array}{c|c} \multirow{3}*{$g,h$} & g,h \text{ and }gh\text{ are proximal, }  V_g^+,V_h^+ \text{ and }V_{gh}^+\text{ are } \\
&\Delta-\nu\text{-regular at scale }{e^{-\alpha_2 p n} }\text{ and } \\ & e^{-c_1 n} \leqslant |\lambda_1(gh) - \lambda_1(g) - \lambda_1(h)| \leqslant e^{-c_2 n}  \end{array}\right\}\right) \geqslant e^{-c_3 n}
\]
\end{lemma}

\begin{proof}
We note, for $n\in \N$ and $\eta,t_2  \in \R_+^\ast$,
\[
G_n = \left\{ \begin{array}{c|l} \multirow{2}{*}{$g\in \G $}& \rho^{\ast n} (B(g,e^{-t_2 n})) \geqslant e^{-t_3 n} ,\;\delta(X_g^M, Y_g^m) \geqslant 2e^{-\eta n} \\ & \forall i\in \{1,2\} \left|\frac 1 n \kappa_i(g) - \lambda_i \right| \leqslant \eta \end{array}\right\}
\]
Then, for any $g\in G_n$,
\[
\kappa_{1,2}(g) = \frac{ \kappa_2(g)}{\kappa_1(g)} \leqslant e^{-(\lambda_1 - \lambda_2 - 2 \eta)n}
\]
And so, noting $\varepsilon = e^{-\eta n}$, we have that if $7\eta< \lambda_1 - \lambda_2$, for $n$ large enough, $g$ satisfy the assumptions of lemma~\ref{lemma:erreur_produit_fgh}. Moreover, for $g\in G_n$ and $p \in \N^\ast$, we note
\[
H_n^p(g) = \left\{ \begin{array}{c|l} \multirow{2}{*}{$ h\in G_{pn}$} & \delta(X_h^M, Y_g^m) \geqslant 2e^{-\eta n},\; \delta(X_g^M, Y_h^m) \geqslant 2e^{-\eta n} \\ & d(g\pi_g X_h^M, Y_h^m) \geqslant 2e^{-\eta n}, \;d(X_g^M, X_h^M) \geqslant 2 e^{-\eta n} \end{array}\right\}
\]
where $\pi_g$ is the projection to $V_g^<$ parallel to $V_g^+$.

If $p$ is such that $(p-1)(\lambda_1 - \lambda_2 + \eta)>\lambda_1 + \eta$, then for any $g\in G_n$ and any $h\in H_n^p(g)$, the pair $(g,h)$ satisfy the assumptions of lemma~\ref{lemma:erreur_produit_fgh} and so we have that for any $f\in B(g^p, e^{ - p(\lambda_1 - \lambda_2 - \eta) n})$,
\begin{align*}
\left|\lambda_1(fgh) - \lambda_1(f) - \lambda_1(gh)-\ln\frac{d(V_g^+, V_g^<) d(gV_h^+, V_h^<)}{d(V_g^+, V_h^<) d(gV_h^+, V_g^<)}\right| & \leqslant 2c_2 e^{-p(\lambda_1 - \lambda_2 - 4 \eta)n} 
\end{align*}
and
\[
\frac{e^{-(d(\lambda_1+\eta) + 5 \eta)n}}{c_3}  \leqslant \left|\ln\frac{d(V_g^+, V_g^<) d(gV_h^+, V_h^<)}{d(V_g^+, V_h^<) d(gV_h^+, V_g^<)} \right|\leqslant c_3 e^{-(\lambda_1 - \lambda_2 - 7 \eta) n}
\]

Moreover, according to lemma~\ref{lemma:product_SVD}, we also have that
\[
\kappa_{1,2}(gh)\leqslant \frac{\kappa_{1,2}(g) \kappa_{1,2}(h)}{\varepsilon^2} \leqslant e^{-(p+1)(\lambda_1 - \lambda_2 - 4 \eta) n}
\]
and
\begin{align*}
\delta(X_{gh}^M, Y_{gh}^m) &\geqslant \delta(X_g^M, Y_h^m) - d(X_{gh}^M, X_g^M) - d(Y_{gh}^M, Y_h^M) \\
&\geqslant 2 e^{-\eta n} - 2e^{-(\lambda_1 -\lambda_2- 3\eta)n}
\end{align*}
So, according to proposition~\ref{proposition:regularite_inferieure_espace_projectif}, choosing $m$ such that $m(\lambda_1 - \lambda_2 - 4 \eta)\geqslant\alpha_2$, we get that
\[
\nu(B(V_{gh}^+, e^{- p \alpha_2 n}) \geqslant \nu(B(V_{gh}^+, (\kappa_{1,2}(gh)/\varepsilon^4)^m) \geqslant \frac 1 2 \left(\rho^{\ast pn} (B(gh, r)) \right)^m
\]
With $r= e^{-(\lambda_1 + \eta)(d+1) pn}$.

Moreover, if $g_1 \in B(g,r_1)$ and $h_1 \in B(h,r_2)$, then
\[
\|gh - g_1 h_1 \| \leqslant r_1 \|h\| + r_2 \|g_1\|
\]
and so,
\[
\rho^{\ast pn} (B(gh,r)) \geqslant \rho^{\ast n} (B(g,e^{-2(\lambda_1 + \eta)(d+1)n})) \rho^{\ast (p-1)n}(B(h,e^{-2(\lambda_1 + \eta)(d+1)n}))
\]
Thus, taking $t_2 = 2(\lambda_1 + \eta)(d+1)$, this proves that $V_{gh}^+$ is $\Delta-\nu$-regular at scale $e^{-\alpha_2p n}$ for some $\Delta \in \R_+$.

Doing the same for $V_f^+$ and $V_{fgh}^+$ (using lemma~\ref{lemma:erreur_rayon_spectral} and lemma~\ref{lemma:product_SVD}), we get that $V_f^+$ and $V_{fgh}^+$ are also $\Delta-\nu$-regular at scale $e^{-\alpha_2 pn}$ (maybe for some different $\Delta$ but there is one that works for $V_{gh}^+$, $V_f^+$ and $V_{fgh}^+$).

\medskip
We just proved that for $p$ such that $p(\lambda_1 - \lambda_2 - 4 \eta) > d\lambda_1 + 5 \eta$,
\begin{align*}
\rho^{\ast pn} \otimes \rho^{\ast pn} &\left( \left\{\begin{array}{c|c} \multirow{3}*{$g,h$} & g,h \text{ and }gh\text{ are proximal, }  V_g^+,V_h^+ \text{ and }V_{gh}^+\text{ are } \\
&\Delta-\nu\text{-regular at scale }{e^{-\alpha_2 p n} }\text{ and } \\ & e^{-c_1 n} \leqslant |\lambda_1(gh) - \lambda_1(g) - \lambda_1(h)| \leqslant e^{-c_2 n}  \end{array}\right\}\right) \\
& \geqslant \int_\G \un_{G_n}(g) \rho^{\ast (p-1)n} (H_n^{p-1})(g) \rho^{\ast pn}(B(g^p, e^{ - p(\lambda_1 - \lambda_2 +\eta) n})) \di\rho^{\ast n}(g)
\end{align*}
But, according to lemma~\ref{lemma:genericite_loxodromie}, we have that any $h \in H_n^{p-1}(g)$ except a set of exponentially small measure (to get the lower bound on $\delta(g\pi_g X_h^M, Y_h^m)$, we do as in the proof of inequality 13.38 in lemma 13.13 in~\cite{BQred} to get that $X_h^M$ and $Y_h^M$ are essentially independent and then guarantee that $X_h^M\not \in V_g^+$ and $d(V_g^<,Y_h^m) \geqslant 2e^{-\eta n}$). So, we only need to find a lower bound on
\[
\int_\G \un_{G_n}(g)\rho^{\ast pn}(B(g^p, e^{ - p(\lambda_1 - \lambda_2 +\eta) n})) \di\rho^{\ast n}(g)
\]
But, doing the same as previously, we find that, maybe for some other constant $t_2$, there is a constant $c_3$ such that for any $g\in G_n$,
\[
\rho^{\ast pn}(B(g^p, e^{ - p(\lambda_1 - \lambda_2+\eta) n})) \geqslant e^{-c_3 n}
\]
And finally, using lemma~\ref{lemma:convolution_mesures_SLd} and lemma~\ref{lemma:genericite_loxodromie}, we have that every $g$ belong to $G_n$ except a set of exponentially small measure and this finishes the proof of the lemma.
\end{proof}

We can now state an ``integrated version'' of lemma~\ref{lemma:controle_difference_rayons_spectraux}.
\begin{lemma} \label{lemma:controle_diff_rayons_spectraux_integre}
Let $\rho$ be a strongly irreducible and proximal borelian probability measure on $\mathrm{SL}_d(\R)$ having an exponential moment.

There are $\alpha, \beta \in \R_+^\ast$ and $p\in \N^\ast$ such that for any $\alpha_2 \in \R_+^\ast$, there is $\Delta \in \R_+^\ast$ such that
\[
\liminf_{b\to \pm\infty} |b|^\alpha \int_{G^{\Delta,2}(b^{-\alpha_2})} \left| e^{i b (\lambda_1(gh) - \lambda_1(g) - \lambda_1(h))} - 1 \right|^2 \di\rho^{\ast p n(\beta,b)}(g) \di\rho^{\ast p n(\beta,b)}(h) >0
\]
Where we noted, $n(\beta,b) = \lfloor \beta \ln |b| \rfloor$ and, for $r\in \R_+^\ast$,
\[
G^{\Delta,2} (r) := \left\{ \begin{array}{c|l} \multirow{2}*{$ (g,h)\in \G^2 $}& g,h \text{ and }gh\text{ are proximal and }\\&V_g^+, V_h^+, V_{gh}^+ \text{ are }\Delta-\nu\text{-regular at scale }r\end{array}\right\}
\]
\end{lemma}

\begin{proof}
We note, with the same notations as in the proof of the previous lemma, for any $n\in \N$,
\[
G_n^2 := \left\{\begin{array}{c|c} \multirow{3}*{$g,h$} & g,h \text{ and }gh\text{ are proximal, }  V_g^+,V_h^+ \text{ and }V_{gh}^+\text{ are } \\
&\Delta-\nu\text{-regular at scale }{e^{-\alpha_2 p n} }\text{ and } \\ & e^{-c_1 n} \leqslant |\lambda_1(gh) - \lambda_1(g) - \lambda_1(h)| \leqslant e^{-c_2 n}  \end{array}\right\}
\]
We choose $\alpha, \beta \in \R_+^\ast$ and we will fix their value later. We note $p$ the parameter given by the previous lemma. For $b\in \R$, we note $n = \lfloor \beta \ln |b| \rfloor$.

Then, for $b$ large enough and any $(g,h) \in G_n^2$,
\[
|b|^{1-c_1 \beta} \leqslant |b| |\lambda_1(g h) - \lambda_1(g) - \lambda_1(g)| \leqslant |b|^{1-c_2 \beta}
\]
So, if $\beta >1/c_2$, we use the inequalities
\[
0< \inf_{x\in [-1,1]}\frac{\left| e^{ix} - 1\right|}{|x|} \leqslant \sup_{x\in [-1,1]}\frac{\left| e^{ix} - 1\right|}{|x|} <+\infty
\]
to get that for large enough $b$ and any $(g,h) \in G_n$,
\[
\left|e^{ib( \lambda_1(g h) - \lambda_1(g) - \lambda_1(g))} - 1 \right| \asymp |b| |\lambda_1(g h) - \lambda_1(g) - \lambda_1(g)| \geqslant |b|^{1-c_1\beta }
\]
So, for $b$ large enough and uniformly for any $(g,h) \in G_n^2$,
\[
\left|e^{ib( \lambda_1(g h) - \lambda_1(g) - \lambda_1(g))} - 1 \right|\geqslant |b|^{1-c_1 \beta}
\]
Moreover, according to lemma~\ref{lemma:controle_difference_rayons_spectraux},
\[
\rho^{\ast p n(\beta, b)} \otimes \rho^{\ast pn(\beta,b)} (G_n^2) \geqslant e^{-c_3 n(\beta,b)} \geqslant |b|^{-c_3 \beta}
\]
So, if $\alpha$ is such that $\alpha - c_3 \beta +2(1 - c_1 \beta) >0$, then
\[
\liminf_{b\to \pm\infty} |b|^\alpha \int_{G^{\Delta,2}(b^{-\alpha_2})}\left| e^{i b (\lambda_1(gh) - \lambda_1(g) - \lambda_1(h))} - 1 \right|^2 \di\rho^{\ast p n(\beta,b)}(g) \di\rho^{\ast p n(\beta,b)}(h) >0
\]
and this is what we intended to prove.
\end{proof}

We can finally prove the control on the logarithm of the spectral radii of proximal elements of $\mathrm{SL}_d(\R)$.
\begin{proposition} \label{proposition:rayons_spectraux_diophantiens}
Let $\rho$ be a strongly irreducible and proximal borelian probability measure on $\mathrm{SL}_d(\R)$ having an exponential moment.

Then, there are $\alpha, \beta \in \R_+^\ast$ and $p\in \N^\ast$ such that for any $\alpha_2 \in \R_+^\ast$, there is $\Delta \in \R_+$ such that
\[
\liminf_{b\to \pm\infty} |b|^\alpha \int_\G \un_{\left\{\substack{g\text{ is proximal and}\\V_g^+\text{ is }\Delta-\nu\text{-regular at scale }b^{-\alpha_2}}\right\}} \left| e^{ib \lambda_1(g)} - 1 \right|^2 \di\rho^{\ast pn(\beta,b)}(g) >0
\]
where we noted $\lambda_1(g)$ the spectral radius of $g$ and $n(\beta,b) = \lfloor \beta \ln |b| \rfloor$.
\end{proposition} 

\begin{proof}
Note, for $r\in \R_+^\ast$,
\[
G^\Delta (r) := \left\{  g\in \G \middle| g\text{ is proximal and }V_g^+\text{ is }\Delta-\nu\text{-regular at scale }r\right\}
\]
Suppose that there are no such $\alpha,\beta,p$. Then, for any $\alpha, \beta,p$,
\[
\liminf_{b\to \pm\infty} |b|^\alpha \int_\G \un_{G^\Delta(b^{-\alpha_2})}(g) \left| e^{ib \lambda_1(g)} - 1 \right|^2 \di\rho^{\ast pn(\beta,b)}(g) =0
\]
In particular, for any $\alpha, \beta\in \R_+^\ast$ and any $p\in \N$,
\begin{align*}
\liminf_{b\to \pm\infty} |b|^\alpha &\int_{\G^2} \un_{G^\Delta(b^{-\alpha_2})}(gh) \left| e^{ib \lambda_1(gh)} - 1 \right|^2 \di\rho^{\ast pn(\beta,b)}(g) \di\rho^{\ast pn(\beta,b)}(h) \\
& \retrait=  \liminf_{b\to \pm\infty} |b|^\alpha \int_{\G} \left| e^{ib \lambda_1(g)} - 1 \right|^2 \di\rho^{\ast 2pn(\beta,b)}(g) =0
\end{align*}
But, using the triangular inequality and keeping the notation $G_2^\Delta$ of the previous lemma, we obtain that
\begin{align*}
I_b(\beta):&=\left(\int_{G^{\Delta,2}(b^{-\alpha_2})} \left| e^{ib (\lambda_1(gh) - \lambda_1(g) - \lambda_1(h))} - 1 \right|^2 \di\rho^{\ast pn(\beta,b)}(g) \di\rho^{\ast pn(\beta,b)}(h) \right)^{1/2}\\ & \leqslant \left(\int_\G \un_{G^\Delta(b^{-\alpha_2})}(g)\left| e^{ib\lambda_1(g)}-1 \right|^2 \di\rho^{\ast 2pn(\beta,b)}(g) \right)^{1/2} \\ & \retrait+ 2\left(\int_\G \un_{G^\Delta(b^{-\alpha_2})}(g) \left| e^{ib\lambda_1(g)}-1 \right|^2\di\rho^{\ast pn(\beta,b)}(g)  \right)^{1/2}
\end{align*}
So, for any $\alpha, \beta$,
\[
\liminf_{b\to \pm\infty} |b|^\alpha I_b (\beta)= 0
\]
and this contradicts lemma~\ref{lemma:controle_diff_rayons_spectraux_integre}.
\end{proof}
\section{The renewal theorem} \label{section:renouvellement} 

\begin{miniabstract}
In this section, we link, as we said in the introduction, the rate of convergence for the renewal theorem to the control of $(I_d-P(it))^{-1}$ in some Banach space.

The aim is to prove theorem~\ref{theorem:renouvellement_general} that we need to establish theorem~\ref{theorem:renouvellement_Rd}.
\end{miniabstract}

\subsection{Preliminaries}
Given a second countable locally compact group $\G$ acting continuously on a metric space $\X$ and a cocycle $\sigma : \G\times \X\to \R$ (see definition~\ref{definition:cocycle}), we can define an action of $\G$ on $\X\times \R$ by setting
\[
g.(x,t) = (gx, t+\sigma(g,x))
\]

If $\rho$ is a borelian probability measure on $\G$, this defines a Markov chain on $\X\times \R$ whose associated operator is the one defined for any continuous function $f$ on $\X\times \R$ and any $(x,t) \in \X\times \R$ by
\[
Pf(x,t) = \int_\G f(g.(x,t)) \di\rho(g) = \int_\G f(gx, t+\sigma(g,x)) \di\rho(g)
\]
This operator commutes to translations on $\R$ and this implies that for any $f\in \mathrm{L}^\infty(\X\times \R)$ and any $g\in \mathrm{L}^1(\R)$,
\[
(Pf) \ast g = P(f\ast g)
\]
where we noted, for $f\in \mathrm{L}^\infty(\X\times \R)$, $g\in \cal L^1(\R)$ and $(x,t) \in \X\times \R$,
\[
f\ast g(x,t) = \int_\R f(x,u) g(t-u) \di u
\]

\medskip
We call renewal kernel the operator $G= \sum_{n=0}^{+\infty} P^n$ when it is defined.

\medskip
Kesten studied in~\cite{Ke74} the properties of $G$ in a very general setting and Guivarc'h and Le Page used his result in~\cite{GuLe15} to get the renewal theorem in $\R^d$ associated to a borelian probability measure on $\mathrm{SL}_d(\R)$ : this is theorem~\ref{theorem:guivarch_lepage} that we stated in the introduction.

\medskip
In this chapter, we would like to study the rate of convergence in Kesten's theorem but we will not do it in his general setting but directly in the case of a group contracting some compact metric space.

So, in the sequel, we fix a second countable locally compact group $\G$ and a borelian probability measure $\rho$ on $\G$.

Let $(\X,d)$ be a compact metric $\G-$space endowed with an action of a finite group $\Hb$ that commutes to the $\G-$action and such that $\X/\Hb$ is $(\rho,\gamma,M,N)-$contracted over a finite $\G-$set $\Abf$ on which the random walk defined by $\rho$ is irreducible and aperiodic (see section~\ref{section:Dolgopyat_Markov}). To simplify notations, we simply note $\pi_\Abf$ the two projections of $\X$ to $\Abf$ and of $\X/\Hb$ to $\Abf$.

We recall that for us, $\G= \mathrm{SL}_d(\R)$, $\rho$ is a probability measure on $\mathrm{SL}_d(\Z)$, $\X = \mathbb{S}^{d-1} \times \Abf$, $\X/\Hb = \prob(\R^d) \times \Abf$ and $\Abf$ is finite $\Gamma_\rho-$set.

\medskip
For technical reasons that we will justify in section~\ref{section:renouvellement_holder}, we introduce the function $\omega : (\X\times \R)^2 \to \R_+$ defined by
\[
\omega((x,t),(x',t')) =\left\{\begin{array}{cl} e^{-\frac{|t|+|t'|}2} \sqrt{d(x,x')^2 + \left(e^{(t-t')/2} - e^{(t'-t)/2}\right)^2}& \text{if }\pi_\Abf  (x) = \pi_\Abf (x')\\ 1 & \text{if not}  \end{array} \right.
\]
and we set, for $\gamma \in \left]0,1\right]$,
\[
\cal C^{0,\gamma}_\omega = \left\{ f\in \cal C^0(\X\times \R) \middle| \|f\|_{\gamma,\omega}:= \sup_{\substack{(x,t),(x',t') \in \X\times \R \\ (x,t)\not=(x',t')}} \frac{|f(x,t) - f(x',t')|}{\omega((x,t),(x',t'))^\gamma} \text{ is finite}\right\}
\]
In the same way, we note,
\[
\omega_0((x,t),(x',t'))  =   \left\{\begin{array}{cl} \frac{\sqrt{|t-t'|^2+  d(x,x')^2} }{(1+|t'|)(1+|t|)} & \text{si }\pi_\Abf  (x) = \pi_\Abf (x')\\ 1 & \text{si non}  \end{array} \right.
\]
and we define $\cal C^{0,\gamma}_{\omega_0} (\X\times \R)$ like $\cal C^{0,\gamma}_\omega(\X\times \R)$.

We will see in section~\ref{section:renouvellement_holder} that for any $f$ in $\cal C^{0,\gamma}_\omega$, there are functions $p^+(f), p^-(f)$ on $\cal \Abf$ such that for any $x\in \X$,
\[
p^-(f)(\pi_\Abf(x))= \lim_{t\to -\infty} f(x,t) \text{ and }p^+(f)(\pi_\Abf(x))= \lim_{t\to +\infty} f(x,t) 
\]

We recall that for a cocycle $\sigma$ on $\X$ and $t\in \R$, we note $P(it) $ the operator defined for any continuous function $f$ on $\X$ and any $x\in \X$ by
\[
P(it)f(x) = \int_\G e^{-it\sigma(g,x)} f(gx) \di\rho(g)
\]
We refer to section~\ref{section:Dolgopyat_Markov} and more specifically to paragraph~\ref{subsubsection:pertubation_cocycles} for more details.

The main result of this section is the following
\begin{theorem}\label{theorem:renouvellement_general}
Let $\G$ be a second countable locally compact group, $N:\G \to [1,+\infty[$ a submultiplicative function on $\G$ and $\rho$ a borelian probability measure on $\G$.

Let $\X$ be a compact metric $\G-$space endowed with an action of a finite group $\Hb$ that commutes to the $\G$ action and such that $\X/\Hb$ is $(\rho,\gamma_0,M,N)-$contracted over a finite $\G-$set $\Abf$ on which the random walk defined by $\rho$ is irreducible and aperiodic.

Let $\sigma\in \cal Z^M(\X/\Hb)$ and $\sigma_\rho = \int_\G \int_{\X/\Hb} \sigma(g,x) \di\nu(x) \di\rho(g)$ where $\nu$ is the unique $P-$invariant probability measure on $\X/\Hb$ given by proposition~\ref{proposition:unique_mesure_invariante_contraction}. We assume that $\sigma_\rho>0$.

We also assume that there is $\gamma_0 \in \left]0,1\right]$ such that for any $\gamma \in \left]0,\gamma_0\right]$ and $t_0 \in \R_+^\ast$, there are $C_0,L$ such that for any $t\in \R$ with $|t|\geqslant t_0$,
\[
\|(I_d-P(it))^{-1} \|_{\cal C^{0,\gamma}(\X)} \leqslant C_0|t|^L
\]
We note $\Pi_0$ the operator define for any $f\in \cal C^\gamma_\omega(\X\times \R)$ such that $p^+(f)=0$ par
\[
\Pi_0 f(x,t) = \int_t^{+\infty}N_0f(x,u) \di u
\]
where $N_0$ is the projector on the space of $P-$invariant function in $\cal C^{0}(\X)$ and we make the abuse of notations $N_0 f(x,u) = N_0 f_u(x)$ with $f_u = f(.,u)$.

Then, for any $\gamma>0$ small enough, there are $\alpha,C \in \R_+^\ast$ such that for any $f\in \cal C^\gamma_\omega(\X\times \R)$ with $p^+(f)=0=\sum_{a\in\Abf} p^-(f)(a)$ and for any $x\in \X$,
\[
\lim_{t\to -\infty} \left(G-\frac 1 {\sigma_\rho}\Pi_0\right)f(x,t) = \sum_{n\in\N} P^n p^-(f)(\pi_\Abf(x))
\]
Moreover, $(G-\frac 1 {\sigma_\rho}\Pi_0)f \in \cal C^\alpha_{\omega_0} (\X\times \R)$ and
\[
\left\|\left(G-\frac 1 {\sigma_\rho}\Pi_0\right)f\right\|_{\alpha,\omega_0} \leqslant C \|f\|_{\gamma,\omega}
\]
\end{theorem}

\begin{remark}
The function $Gf$ is well defined under the assumptions of the theorem since the convergence in the series is uniform on any compact subset of $\X\times \R$.
\end{remark}

\begin{proof}
To prove the theorem, we use the decomposition given in lemma~\ref{lemma:decomposition_fonctions_bord}, the corollary~\ref{corollaire:renouvellement_fonctions_nulles_bord} and lemma~\ref{lemma:renouvellement_fonctions_bord}
\end{proof}

\subsection{Non-unitary perturbations by cocycles}

\begin{miniabstract}
In this paragraph, we study the inverse of the operator $I_d-P(z)$ and we prove proposition~\ref{proposition:controle_inverse_perturbe} that shows that a control of the growth of the norm on the imaginary axis gives a control of the norm of the operator and it's derivatives on a neighbourhood with a nice shape.
\end{miniabstract}

Let $\G$ be a second countable locally compact group acting on a compact metric space $(\X,d)$ and $\rho$ a borelian probability measure on $\G$.

For a cocycle $\sigma:\G\times \X\to \R$ and $g\in \G$, we recall that we noted
\[
\sigma_{\mathrm{sup}}(g) = \sup_{x\in \X} |\sigma(g,x)| \text{ and }\sigma_{\mathrm{Lip}}(g) = \sup_{\substack{x,y\in \X\\ \pi_\Abf(x) = \pi_\Abf(y) \\ x\not=y}} \frac{|\sigma(g,x) - \sigma(g,y)|}{d(x,y)}
\]
And, for $M\in \R_+$ and $N:\G \to [1, +\infty[$ a submultiplicative function on $\G$,
\[
\cal Z^{M}_N(\X) = \left\{\sigma\text{ is a cocycle on }\X
\middle|  \sup_{g\in \G}\frac{\sigma_{\mathrm{Lip}}(g)}{N(g)^M } \text{ and }\sup_{g\in \G}\frac{e^{\sigma_{\mathrm{sup}}(g)}}{N(g)^M } \text{ are finite}\right\}
\]
Finally, for $\sigma \in \cal Z^{M}_N(\X)$, we set
\[
\Msig = \sup_{g\in \G} \frac{\sigma_{\mathrm{Lip}}(g)}{N(g)^M}  \text{ and }\Isig = \sup_{g\in \G}\frac{e^{\sigma_{\mathrm{sup}}(g)}}{N(g)^M } 
\]

We note $\C_\gamma := \{z\in \C||\Re(z)|<\gamma\}$.
For $z \in \C_{\gamma}$ and $\sigma\in \cal Z^M_N(\X/\Hb)$, we define the operator $P(z)$ on $\cal C^0(\X)$ by
\[
P(z) f(x) = \int_\G e^{-z\sigma(g,x)} f(gx) \di\rho(g)
\]
This is a continuous operator since for any continuous function $f$ on $\X$, any $z\in \C_\gamma$ and any $x\in \X$,
\[
|P(z) f(x)| \leqslant \|f\|_\infty \int_\G e^{-\Re(z) \sigma(g,x) }\di\rho(g) \leqslant \|f\|_\infty \int_\G \Isig^\gamma N(g)^{\gamma M} \di\rho(g)
\]
and $\int_\G N(g)^{\gamma M} \di\rho(g)$ is finite since the action is contracting.

In the sequel, we will note, for $\eta\in \R_+^\ast$, $\C_\eta:= \{z\in \C | |\Re(z)|<\eta\}$.

\medskip
We group the main results of this paragraph in the next
\begin{proposition} \label{proposition:controle_inverse_perturbe}
Under the assumptions and notations of theorem~\ref{theorem:renouvellement_general}.

For any $\gamma >0$ small enough, there are $\eta, C,L,t\in \R_+^\ast$ such that $(z\mapsto P(z))$ is an analytic function from $\C_\eta$ to the space of continuous operators on $\cal C^{0,\gamma}(\X)$. Moreover, for any $z\in[0,\eta[ \oplus i\R$ and any $n\in \N$,
\[
\|P(z)^n\|_\gamma\leqslant C(1+|z|)e^{-t \Re(z) n}
\]

Finally, noting
\[
U(z) = (I_d-P(z))^{-1} - \frac 1 {\sigma_\rho z} N_0,
\]
we have that $(z\mapsto U(z))$ (which is definite a priori on $i\R \setminus\{0\}$) can be extended to an analytic function onto the space of continuous operators on $\cal C^{0,\gamma}(\X)$ and defined on
\[
\cal D_{\eta, C,L}: = \left\{ z \in \C \middle| \frac{-1}{C(1+|\Im z|)^{L+1}} < \Re(z) < \eta \right\}
\]
and for any $n\in \N$ and any $z\in\cal D_{\eta, C,L} $
\[
\|U^{(n)}(z) \|_\gamma\leqslant n!  C^{n+1} (1+|\Im z|)^{(L+1)(n+1)}
\]
\end{proposition}

\begin{remark}
This proposition generalises the situation in $\R$ when $P(z)$ is the Fourier-Laplace transform of the measure $\rho$. In this case, the same estimations can be obtained under the ``non-lattice of type $p$'' assumption used by Carlsson in~\cite{Car83}. 
\end{remark}

Before we prove each of the assertions of the proposition, we draw the zone $\cal D_{\eta,C,L}$.

\begin{figure}[H]
\centering
\begin{tikzpicture}
\fill[pattern = north west lines] (1,-2) -- (1,2) -- plot [domain=2:-2,variable=\t,samples=200] ({-0.8/(1+(abs(\t)^2)},\t) -- cycle ;
\draw [very thin,->] (-2,0) -- (2,0) ;
\draw [very thin,->] (0,-2.2) -- (0,2.2) ;
\draw (-0.2,0) node[below]{$O$} ;
\draw (1.2,0.2) node[right]{$\eta$} ;
\draw (-0.8,0.25) node[left]{$-1/C$} ;
\draw (3,0) node[right]{ $\left\{ z \in \C \middle| \frac{-1}{C(1+|\Im z|)^{L+1}} < \Re(z) < \eta \right\}$} ;
\draw [ultra thick] (1,-2) -- (1,2) ;
\draw [ultra thick] plot [domain=-2:2,variable=\t,samples=200] ({-0.8/(1+(abs(\t)^2)},\t); 
\end{tikzpicture}
\caption{Shape of the zone~$\cal D_{\eta,C,L}$}
\end{figure}
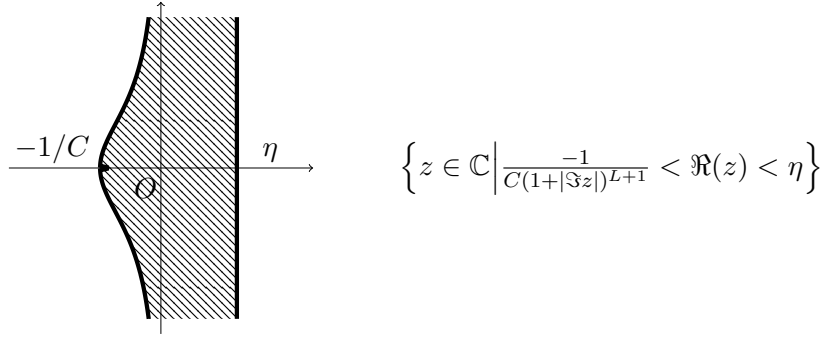

\begin{lemma} \label{lemma:drift}
Under the assumptions of proposition~\ref{proposition:controle_inverse_perturbe}, if $\sigma_\rho>0$ then there are $\eta,t,C \in \R_+^\ast$ such that for any $s\in[0,\eta]$ and any $n\in \N$,
\[
\sup_{x\in \X} \int_\G e^{-s\sigma(g,x)} \di\rho^{\ast n}(g) \leqslant Ce^{-ts n}
\]
\end{lemma}

\begin{proof}
First of all, according to Jensen's inequality, for any $0\leqslant s\leqslant \eta$,
\[
\int_\G e^{-s\sigma(g,x)} \di\rho^{\ast n}(g)\leqslant \left( \int_\G e^{-\eta\sigma(g,x)} \di\rho^{\ast n}(g) \right)^{s/\eta} = \left(P(\eta)^n 1(x) \right)^{s/\eta}
\]
Moreover, as $\sigma_\rho>0$, lemma 10.17 in~\cite{BQred} and the fact that (with their notations but reminding that their operator $P_\vartheta$ is what we called $P(-i\vartheta)$) and their equality 10.32,
\[
\lambda'(0) = \int_\X P'(0) 1 \di\nu(x) = - \sigma_\rho<0
\]
prove that there are $\eta,t,C\in \R_+^\ast$ such that \[
\sup_x P(\eta)^n 1(x) \leqslant Ce^{-tn}
\]
and this is what we intended to prove.
\end{proof}

\begin{lemma}\label{lemme:controle_P_z}
Under the assumptions of proposition~\ref{proposition:controle_inverse_perturbe}, for any $\gamma>0$ small enough there is $\eta\in \R_+^\ast$ such that the function $(z\mapsto P(z))$ is analytic from $\cal C_{\eta}$ to the space of continuous operators on $\cal C^{0,\gamma}(\X)$ and there are $t\in \R_+^\ast$ and $C\in \R$ such that for any $z\in \C_\eta$ with $\Re(z) \geqslant 0$, any function $f\in \cal C^{0,\gamma}(\X)$ and any $n\in \N$,
\[
\| P(z)^n f\|_\gamma \leqslant C\left(e^{-t n} m_\gamma(f) + (1+|z|) \|f\|_\infty\right)
\]
\end{lemma}

\begin{remark}
This proof is very close to the one of proposition~\ref{proposition:essential_spectral_radius}. The difficulty here is that the perturbation no longer have modulus $1$.
\end{remark}

\begin{proof}
To see that $(P(z))$ is an analytic family of operators, we refer to lemma 10.16 de~\cite{BQred}.

Let $\gamma>0$ small enough and $\eta \in \R_+^\ast$.

Compute, for $z\in \C_{\eta}$, $f\in \cal C^{0,\gamma}(\X)$, $x,y \in \X$ with $\pi_\Abf \circ\pi_\Hb(x) = \pi_\Abf\circ \pi_\Hb(y)$ and $n\in \N$,
\begin{align}
 | P(z)^n f(x) &- P(z)^n f(y) |   \notag\\ &\leqslant \int_\G \left| e^{- z \sigma(g,x)} f(gx) - e^{-z \sigma(g,y)} f(gy) \right| \di\rho^{\ast n}(g) \notag \\
& \leqslant \int_\G e^{-\Re(z)\sigma (g,x)} | f(gx)- f(gy)| + \|f\|_\infty  \left|e^{-z\sigma(g,x)} - e^{-z \sigma(g,y)} \right| \di\rho^{\ast n}(g) \label{equation:controle_P_z_1}
\end{align}
But, using the definition of $\Msig$ and $\Isig$ (cf. equation~\eqref{equation:Msig_I_sig}), we have that for any $\varepsilon \in \R_+^\ast$
\begin{align*}
\left| e^{-z\sigma(g,x)} - e^{-z \sigma(g,y)} \right| & \leqslant 2^{1-\gamma} |z| e^{\Re(z) \sigma_{\mathrm{sup}}(g)} (\sigma_{\mathrm{sup}}(g))^{1-\gamma}  |\sigma(g,x)- \sigma(g,y)|^\gamma \\
&\leqslant 2^{1-\gamma}|z| \Isig^{\eta} N(g)^{M (\gamma+|\Re(z)|)} \left(\ln (\Isig N(g)^M) \right)^{1-\gamma} \Msig^\gamma  d(x,y)^\gamma \\
& \leqslant C_\varepsilon  |z| N(g)^{M(\gamma + \eta+ \varepsilon)} \Isig \Msig d(x,y)^\gamma
\end{align*}
where we noted $C_\varepsilon$ such that for any $x\in [1,+\infty[$, $x\leqslant C_\varepsilon e^{\varepsilon x}/2^{1-\gamma}$.

And so,
\begin{equation} \label{equation:controle_P_z_2}
\int_\G \left| e^{-z\sigma(g,x)} - e^{-z \sigma(g,y)} \right| \di\rho^{\ast n}(g) \leqslant C_\varepsilon |z|\Isig \Msig d(x,y)^\gamma \int_\G  N(g)^{M(\gamma+\eta + \varepsilon)} \di\rho^{\ast n}(g)
\end{equation}
Moreover,
\begin{equation} \label{equation:controle_P_z_3}
\int_\G e^{-\Re(z)\sigma (g,x)} | f(gx)- f(gy)| \di\rho^{\ast n}(g) \leqslant m_\gamma(f) \Isig^\eta \int_\G N(g)^{M\eta} d(gx,gy)^\gamma \di\rho^{\ast n}(g)
\end{equation}
Let $d_0\in \R_+^\ast$ be such that if $d(x,y) \leqslant d_0$ then $d(x,y) = d(\pi_\Hb x,\pi_\Hb  y)$.
Then, for any $\varepsilon' \in ]0,1]$ and any $x,y$ with $0<d(x,y)\leqslant \varepsilon' d_0$, we have that
\begin{align}
I_n(x,y):&=\int_\G N(g)^{M\eta}  d(gx,gy)^\gamma \di\rho^{\ast n}(g) \notag \\
& \leqslant \int_\G N(g)^{M\eta} \un_{d(gx,gy) \leqslant d_0}d(gx,gy)^\gamma \di\rho^{\ast n}(g)\notag \\& \retrait+ \int_\G N(g)^{M\eta} \un_{d(gx,gy) \geqslant d_0} d(gx,gy)^\gamma \di\rho^{\ast n}(g) \notag \\
& \leqslant \int_\G N(g)^{M\eta}d(g\pi_\Hb x,g\pi_\Hb y)^\gamma \di\rho^{\ast n}(g)  \notag\\& \retrait+ \int_\G N(g)^{M\eta} \un_{MN(g)^M \geqslant 1/\varepsilon} MN(g)^{M\gamma} d(x,y)^\gamma \di\rho^{\ast n}(g)\notag \\
& \leqslant d(x,y)^\gamma \left(\int_\G N(g)^{M\eta}\frac{d(g\pi_\Hb x,g\pi_\Hb y)^\gamma}{d(\pi_\Hb x,\pi_\Hb y)^\gamma} \di\rho^{\ast n}(g) \right. \notag \\& \retrait\retrait\left.+ M\int_\G N(g)^{M(\eta+\gamma)} \un_{MN(g)^M \geqslant 1/\varepsilon}  \di\rho^{\ast n}(g) \right) \label{equation:controle_P_z_5}
\end{align}
But, using Cauchy-Schwartz's inequality and the contraction of $\X/\Hb$, we have that
\begin{align}
J_n(x,y) :& =\int_\G N(g)^{M\eta}\frac{d(g\pi_\Hb x,g\pi_\Hb y)^\gamma}{d(\pi_\Hb x,\pi_\Hb y)^\gamma} \di\rho^{\ast n}(g) \notag \\&\leqslant \left(\int_\G N(g)^{2M\eta} \di\rho^{\ast n}(g) \int_\G \frac{ d(g\pi_\Hb x,g\pi_\Hb y)^{2\gamma}}{ d(x,y)^{2\gamma}}\di\rho^{\ast n}(g) \right)^{1/2} \notag\\
& \leqslant \sqrt{C_{2\gamma}} e^{-\delta_{2\gamma}n/2}  \left(\int_\G N(g)^{2\eta M} \di\rho(g) \right)^{n/2} \label{equation:controle_P_z_4}
\end{align}
We can now choose $n$ such that $\sqrt{C_{2\gamma}} e^{-\delta_{2\gamma}n/4} \leqslant 1/{4 \Isig}$ and then, for $\eta$ such that
\[
e^{-\delta_{2\gamma} n /4} \int_\G N(g)^{M\eta} \di\rho(g) \leqslant 1,
\]
we find, with equation~\eqref{equation:controle_P_z_4}, that
\[
J_n(x,y) = \int_\G N(g)^{M\eta}\frac{d(g\pi_\Hb x,g\pi_\Hb y)^\gamma}{d(\pi_\Hb x,\pi_\Hb y)^\gamma} \di\rho^{\ast n}(g) \leqslant \frac 1 {4\Isig}
\]
Moreover, for this fixed $n$ we can choose $\varepsilon'\in \R_+^\ast$ such that
\[
M\int_\G N(g)^{M(\eta + \gamma)} \un_{MN(g)^M \geqslant 1/\varepsilon'} \di\rho^{\ast n}(g) \leqslant 1/{4\Isig}
\]
and this proves, with equations~\eqref{equation:controle_P_z_1}, \eqref{equation:controle_P_z_2}, \eqref{equation:controle_P_z_3} and \eqref{equation:controle_P_z_5}, that for any $f\in \cal C^{0,\gamma}(\X)$, any $x,y \in \X$ with $\pi_\Abf \circ \pi_\Hb(x) = \pi_\Abf \circ \pi_\Hb(y)$ and $d(x,y)\leqslant \varepsilon'd_0$,
\[
\frac{|P(z)^nf(x) - P(z)^n f(y) | }{d(x,y)^\gamma}\leqslant \frac 1 2 m_\gamma(f) + 2\|f\|_\infty |z| \Isig \Msig \int_\G N(g)^{M(\gamma + \eta + \varepsilon)} \di\rho^{\ast n}(g)
\]
Moreover, if $d(x,y) \geqslant \varepsilon' d_0$, then
\[
\left|P^n(z)f(x) - P^n (z)f(y)\right| \leqslant 2 \frac{d(x,y)^\gamma}{(\varepsilon' d_0)^\gamma} \|f\|_\infty \Isig \int_\G N(g)^{M\eta} \di\rho^{\ast n}(g)
\]
So, what we proved is that for any $\gamma,\eta>0$ small enough there are $n\in \N^\ast$ and a constant $C$ (depending on $n, \sigma$ and $\rho$) such that for any $z\in \C_{\eta}$ and any $f\in \cal C^{0,\gamma}(\X)$,
\[
m_\gamma(P(z)^n f) \leqslant \frac 1 2 m_\gamma(f) + Ce^{Cn}(1+|z|) \|f\|_\infty
\]
But, we also have, according to lemma~\ref{lemma:drift}, that for $\Re(z) \geqslant 0$,
\[
\|P(z)^n f\|_\infty \leqslant \|f\|_\infty \sup_{x\in \X}\int_\G e^{-\Re(z) \sigma(g,x)} \di\rho^{\ast n}(g) \leqslant C\|f\|_\infty e^{-\Re(z) tn}
\]
So, we obtain the expected inequalities by iterating this relations and we refer to~\cite{ITM50} for a proof that we can choose a constant $C$ that doesn't depend on $n$ nor on $z$.
\end{proof}

An analytic family of operators $(P(z))$ is said to be meromorphic at $z_0$ if there is $N \in \N$ such that the family $((z-z_0)^N P(z))$ is analytic on a neighbourhood of $z_0$.

We are now going to use a version of the analytic Fredholm theorem that holds for quasi-compacts operators and that we state in next
\begin{theorem} \label{theorem:Fredholm_analytique_quasi_compact}
Let $(\cal B,\|\,.\,\|_{\cal B})$ be a Banach space.

Let $\|\,.\,\|$ be a norm on $\cal B$ such that the unit ball in $\cal B$ for $\|\,.\,\|_{\cal B}$ is relatively compact for $\|\,.\,\|$.

Let $\cal U $ be a connected open subset of $\C$.

Let $(P(z))_{z\in \cal U}$ be an analytic family of operators defined on $\cal U$ such that there are $r\in [0,1[$ and a real valued function $(z\mapsto R(z))$ such that for any $f\in \cal B$ and any $z\in \cal U$,
\[
\|P(z) f\|_{\cal B} \leqslant r \|f\|_{\cal B} + R(z) \|f\| 
\]

Then, we have the following alternative :
\begin{itemize}
\item The operator $I_d-P(z)$ is invertible for no $z\in \cal U$.
\item The function $(z\mapsto (I_d-P(z))^{-1})$ is meromorphic on $\cal U$.
\end{itemize}
\end{theorem}

\begin{proof}
The proof is the same as in the classic case when we remark that, according to Ionescu-Tulcea and Marinescu's theorem (that we recalled in theorem~\ref{theorem:ITM}) we have a control on the essential spectral radius of $P(z)$ that is uniform in $z$.
\end{proof}

\begin{lemma} \label{lemma:U_holo}
Under the assumptions of proposition~\ref{proposition:controle_inverse_perturbe}, the family $((I_d-P(z))^{-1})_{z \in \C_{\eta}}$ is a meromorphic family of operators that is analytic on  $]0,\eta[ \oplus i\R$.

Moreover, we can choose $\eta$ such that there are $C,t\in \R_+^\ast$ such that for any $z\in\left]0,\eta\right[ \oplus i\R$ and any $n\in \N$,
\[
\|P(z)^n\|_\gamma\leqslant C(1+|z|)e^{-t \Re(z) n}
\]

Finally, if $\sigma_\rho \not=0$ and if $\eta'$ is small enough, we can write, for $z \in B(0,\eta') \setminus\{0\}$,
\[
(I_d-P(z))^{-1} = \frac 1 {\sigma_\rho z} N_0 + U(z)
\]
where $N_0$ is the operator of projection on to the space of $P-$invariant functions and $(U(z))$ is an analytic family of continuous operators on $\cal C^{0,\gamma}(\X)$ and defined on $B(0,\eta')$. 
\end{lemma}

\begin{proof}
For any $z\in [0,\eta]\oplus i\R$, any $n\in \N$ and any $f\in \cal C^0(\X)$, we have, according to lemma~\ref{lemma:drift}, that
\begin{align*}
\|P(z)^n f \|_\infty &\leqslant  \|f\|_\infty \sup_{x\in \X} \int_\G e^{-\Re(z) \sigma(g,x)} \di\rho^{\ast n}(g) \leqslant C\|f\|_\infty e^{-t \Re(z) n}
\end{align*}
And so, for any $f\in \cal C^{0,\gamma}(\X)$ and any $n\in \N$, (we can assume without any loss of generality that the constants $t$ given in lemma~\ref{lemme:controle_P_z} and lemma~\ref{lemma:drift} are equal and the same thing for the constants $C$), according to lemma~\ref{lemme:controle_P_z} applied to $P^n(z)(P^n(z) f)$ and then to $P^n(z) f$, we have that
\begin{align*}
\|P(z)^{2n} f\|_\gamma &\leqslant C \left( e^{- tn} \|P^n(z) f\|_\gamma + (1+|z|) \|P^n(z) f\|_\infty \right) \\
&\leqslant C \left( e^{-tn}C( m_\gamma(f) + (1+|z|)e^{-tn}\|f\|_\infty) + Ce^{-t\Re(z) n}(1+|z|)\|f\|_\infty \right) \\
& \leqslant C^2(1+|z|) \|f\|_\gamma \left(e^{-tn} + e^{-t \Re(z) n} \right)
\end{align*}
And this is what we intended to prove.

Moreover, lemma~\ref{lemme:controle_P_z} shows that the essential spectral radius of $P(z)$ is uniformly bounded by $e^{-t}$ on $\C_\eta$ and so we can apply theorem~\ref{theorem:Fredholm_analytique_quasi_compact} to show that the family $(I_d-P(z))^{-1}$ is meromorphic on $\C_\eta$ since we just proved that $(I_d- P(z))^{-1}$ is well defined on $]0,\eta] \oplus i\R$.

Moreover, this proves that $(I_d-P(z))^{-1} $ has no pole in $]0,\eta[ \oplus i\R$.

\medskip
Finally, we refer to lemma~3.2 de~\cite{BL85} or to lemma~10.17 in~\cite{BQred} for the expansion of $(I_d-P(z))^{-1}$ on the neighbourhood of $0$. Indeed, adapting their argument, we find that there are two analytic families of continuous operators $(N(z)), (U_1(z))$ defined on the neighbourhood of $0$ and an analytic function $\lambda$ such that for any non zero $z$ in the considered neighbourhood of $0$,
\[
(I_d-P(z))^{-1} = \frac 1 {1 - \lambda(z)} N(z) + U_1(z)
\]
This finishes the proof of the lemma since $\lambda$ and $N$ are analytic, $\lambda(0)=1$, $\lambda'(0) =- \sigma_\rho\not=0$ and $N(0)$ is the projector on $\ker(I_d-P)$.
\end{proof}

\begin{lemma}
Under the assumptions and notations of proposition~\ref{proposition:controle_inverse_perturbe}

We note
\[
U(z) = (I_d-P(z))^{-1} - \frac 1 {\sigma_\rho z} N_0
\]
Then, for any $\gamma,\eta>0$ small enough, there are $C,L$ such that for any $z\in \C$ with
\[
\frac {-1}  {C (1+|\Im z|)^L} < \Re(z) <\eta
\]
and any $n \in \N$, we have that
\[
\| U^{(n)}(z) \|_\gamma\leqslant C^{n+1} n! (1+| z|)^{(L+1)(n+1)}
\]
\end{lemma}

\begin{proof}
First of all, we note that, according to the previous lemma, $U(z)$ is well defined on a neighbourhood $B(0,\eta')$ of $0$.

Moreover, for any hölder-continuous function $f$, any $z\in \C_\eta$ and any $x\in \X$, we have that
\[
P'(z) f(x) = \int_\G \sigma(g,x) e^{-z \sigma(g,x)} f(gx) \di\rho(g)
\]
And so, doing the same kind of computations than in the proof of lemma~\ref{lemme:controle_P_z}, we find that for some constant $C_1\in \R_+$ and any $z \in \C_\eta$,
\[
\|P'(z)\|_{\gamma} \leqslant C_1(1+|z|)
\]
So, for any $f\in \cal C^{0,\gamma}(\X)$,
\begin{align*}
\|(I_d -P(z))f\|_\gamma &\geqslant \|(I_d-P(\Im z))f\|_\gamma -  |\Re(z)| \sup_{z'} \|P'(z)\|_\gamma \|f\|_\gamma\\& \geqslant \|(I_d-P(\Im z))f\|_\gamma  - C_1|\Re(z)|(1+|z|)\|f\|_\gamma
\end{align*}
Thus,
\[
\inf_{f\in \cal C^{0,\gamma}(\X)\setminus\{0\}} \frac{\|(I_d -P(z))f\|_\gamma}{\|f\|_\gamma} \geqslant \inf_{f\in \cal C^{0,\gamma}(\X)\setminus\{0\}} \frac{\|(I_d -P(\Im z))f\|_\gamma}{\|f\|_\gamma} - C_1|\Re(z)|(1+|z|)
\]
But, by assumption, for $z \in \C$ with $|z| \geqslant \eta'$,
\[
 \inf_{f\in \cal C^{0,\gamma}(\X)\setminus\{0\}} \frac{\|(I_d -P(\Im z))f\|_\gamma}{\|f\|_\gamma} = \frac{ 1 }{\|(I_d-P(\Im z))^{-1} \|_\gamma} \geqslant \frac 1 {C_0 (1+| z|)^L}
\]
So, for any $z \in \C_\eta$ with $|z| \geqslant \eta'$ and
\[
|\Re(z)| \leqslant \frac{ C_0}{C_1 (1+|z|)^{L+1}},
\]
we have that
\[
\inf_{f\in \cal C^{0,\gamma}(\X)\setminus\{0\}} \frac{\|(I_d -P(z))f\|_\gamma}{\|f\|_\gamma} \geqslant \frac 1 {2C_0 (1+|z|)^L}
\]
This proves, using that $P(z)$ has an essential spectral radius strictly smaller than $1$, that $I_d-P(z)$ is invertible and that
\[
\|(I_d-P(z))^{-1} \|_\gamma \leqslant 2C_0 (1+|z|)^L
\]
We just proved that there is a constant $C$ such that for any $z\in \C_\eta$ with
\[
|\Re(z)| \leqslant \frac 1 {C (1+|z|)^{L+1}}
\]
we have that $U$ is analytic at $z$ and
\[
\|U(z)\|_\gamma \leqslant C (1+|z|)^L
\]
Moreover, according to lemma~\ref{lemma:U_holo}, if $z\in \C_\eta$ with $\Re(z) \geqslant 1/C(1+|z|)^{L+1}$, then, we have that for any $n\in \N$,
\[
\|P(z)^n \|_\gamma\leqslant C(1+|z|)e^{-t{\Re(z)}n}
\]
So,
\[
\|(I_d-P(z))^{-1} \|_\gamma \leqslant \frac{C(1+|z|)}{1-e^{-t_0{\Re(z)}}} \leqslant  \frac{C(1+|z|)}{t \Re(z)} \leqslant \frac{ C^2 (1+|z|)^{L+2}}{t}
\]
And this proves that the function $U$ is analytic on $]0,\eta[\oplus i\R$.

We just proved that for $\eta$ small enough and any $z\in \C$ with
\[
\frac  {-1} {C(1+|z|)^{L+1}} < \Re(z) < \eta
\]
we have that
\[
\|U(z)\|_\gamma \leqslant C' (1+|z|)^{L+2}
\]
for some constant $C'$.

To conclude, we do the same kind of computations than Gelfand and Shilov in the proof of theorem~15 in~\cite{GC64} to get the control of the derivatives of $U$ on the domain $D_{\eta,C'',L+1}$.
\end{proof}

\subsection{The renewal theorem for regular functions} \subsectionmark{Regular functions}

\begin{miniabstract}
In this paragraph, we prove a result of representation of the renewal kernel and we deduce the rate of convergence in the renewal theorem for regular functions.
\end{miniabstract}

Let $\gamma \in \R_+^\ast$. For $f\in \cal C^0(\X\times \R)$, we note
\[
m_{\gamma, \cal E} (f) =  \sup_{t\in \R} \sup_{\substack{x,x' \in \X\\x\not=x'\\ \pi_\Abf(x) = \pi_\Abf(x')}} e^{\gamma|t|} \frac{|f(x,t) - f(x',t)|}{d(x,x')^\gamma}
\]
and
\[
\|f\|_{\gamma,\infty} = \sup_{x\in \X} \sup_{t\in \R} e^{\gamma|t|} |f(x,t)|
\]
Moreover, we note
\[
\|f\|_{\gamma,\cal E} = \|f\|_{\gamma,\infty}  + m_{\gamma, \cal E} (f)
\]
And finally, we set
\begin{equation} \label{equation:E_gamma_zero}
\cal E^{\gamma,k}(\X\times \R) = \left\{ f\in \cal C^k(\R, \cal C^0(\X)) \middle| \forall m\in\lib 0,k \rib \|f^{(m)}\|_{\gamma,\cal E}  \text{ is finite} \right\}
\end{equation} 
where we noted
\[
f^{(k)}(x,t) = \frac{ \partial ^k f}{\partial t^k}(x,t)
\]
And, for $f\in \cal E^{\gamma,k}(\X\times \R)$,
\[
\|f\|_{\gamma,k} = \max_{m\in\lib 0,k \rib} \|f^{(m)}\|_{\gamma,\cal E}
\]

If $f\in \cal E^{\gamma,0}(\X\times \R)$, then, we note for $x\in \X$ and $\xi \in \R$,
\[
\widehat f(x,\xi) = \int_\R e^{-i\xi t} f(x,t) \di t
\]
It is clear that, for fixed $\xi$ the function $\widehat f(.,\xi)$ is hölder-continuous on $\X$.

Moreover,
\[
\frac{ \partial^l \widehat f}{\partial \xi^l}(x,\xi) = \int_\R (-it)^l e^{-i\xi t} f(x,t) \di t
\]
And integrating by parts, we find that for any $\xi\in\R$ and any $m\in \N$,
\[
\xi^m \widehat f(x,\xi) =(-i)^m \int_\R e^{-i\xi t} f^{(m)}(x,t) \di t 
\]
So, if $f\in \cal E^{\gamma,k}(\X\times \R)$, then
\[
(1+|\xi|^k) |\widehat{f}(x,\xi)| \leqslant \int_\R |f(x,t)| \di t + \int_\R |f^{(k)}(x,t)| \di t \leqslant 2\|f\|_{\gamma,k} \int_\R e^{-\gamma|t|} \di t
\]
In the same way, if $\pi_\Abf \circ \pi_\Hb (x) = \pi_\Abf \circ \pi_\Hb(x')$, then
\[
(1+|\xi|^k)|\widehat f(x,\xi) - \widehat f(x',\xi)| \leqslant 2\|f\|_{\gamma,k}d(x,x')^\gamma \int_\R e^{-\gamma|t|} \di t
\]
So, as we can do the same with the functions $\frac{ \partial^l \widehat f}{\partial \xi^l}$, we just prove the following
\begin{lemma} \label{lemma:fourier_transform_vanish} Let $k,l \in \N$. There is a constant $C$ such that for any $f\in \cal E^{\gamma,k}(\X\times \R)$ and any $\xi \in \R$, we have that
\[
\left\| \frac{ \partial^l \widehat f}{\partial \xi^l}(\,.\,,\xi)\right\|_\gamma  \leqslant C \frac{\|f\|_{\gamma,k}}{1+|\xi|^k}
\]
\end{lemma}

\begin{remark}
This lemma shows that we recover the classical properties of the Fourier transform that exchanges regularity of functions and decay at infinity.
\end{remark}

We shall now prove that convolution with functions of $\cal E^{\gamma,k}$ regularises functions of $\cal E^{\gamma,0}$. As we will not need this result in full generality, we just prove it for some particular function in $\cal E^{\gamma,k}$.
\begin{lemma} \label{lemma:regularizing}
Let, for $k \in \N$, $\varphi_k$ be the function defined for $t\in \R$ by
\[
\varphi_k(t) = t^k e^{-t} \un_{\R_+}(t)
\]
Then for any $\gamma \in \left]0,1\right[$ and any $k\in \N$, there is a constant $C_k$ such that for any $f\in \cal E^{\gamma,0}(\X\times \R)$,
\[
\varphi_{k+1} \ast f \in \cal E^{\gamma,k}(\X\times \R)
\]
and
\[
\|\varphi_{k+1} \ast f\|_{\gamma,k} \leqslant C_k \|f\|_{\gamma,0}
\]
\end{lemma}

\begin{proof}
Thus usual properties of the convolution product prove that, as $\varphi_{k+1} \in \cal C^{k}(\R)$, the function $f\ast \varphi_{k+1}$ is $k-$times differentiable and it $k-$th derivative is continuous (since $f$ is continuous) and for any $x\in \X$, any $t\in \R$ and any $m\in\lib 0, k\rib$
\[
\left(\varphi_{k+1} \ast f \right)^{(m)}(x,t) = \varphi_{k+1}^{(m)} \ast f(x,t)
\]
But,
\[
\varphi_{k+1}^{(m)}(t) = \sum_{l=0}^{m} \binom m l  (-1)^{m-l}\frac{(k+1)!}{(k+1-l)!} t^{k+1-l} e^{-t} \un_{\R_+}(t)
\]
So,
\[
\left(\varphi_{k+1} \ast f \right)^{(m)}(x,t) = \sum_{l=0}^{m} \binom m l  (-1)^{m-l}\frac{(k+1)!}{(k+1-l)!} \int_{\R_+} u^{k+1-l} f(x,t-u) e^{-u} \di u
\]
and so, if $x,x' \in \X$ are such that $\pi_\Abf \circ \pi _\Hb(x) = \pi_\Abf \circ \pi_\Hb(x')$,
\begin{align*}
I_{k,m,x,x',t}&:=\left|\left(\varphi_{k+1} \ast f \right)^{(m)}(x,t) - \left(\varphi_{k+1} \ast f \right)^{(m)}(x',t) \right| \\
& \leqslant \sum_{l=0}^{m} \binom m l  \frac{(k+1)!}{(k+1-l)!}\int_{\R_+} u^{k+1-l} e^{-\gamma|t-u|} d(x,x')^\gamma \|f\|_{\gamma,0} e^{-u} \di u \\
& \leqslant e^{-\gamma|t|}d(x,x')^\gamma \|f\|_{\gamma,0}  \sum_{l=0}^{m} \binom m l \frac{(k+1)!}{(k+1-l)!}\int_{\R_+} e^{\gamma|u|} u^{k+1-m} e^{-u} \di u
\end{align*}
where we used that for any $v,w\in \R$,
\[
e^{|v|-|v+w|} \leqslant e^{|w|}
\]
In the same way, we get that
\begin{align*}
J_{k,m,x,t} &= \left|\left(\varphi_{k+1} \ast f \right)^{(m)}(x,t) \right| \leqslant C_{k,m} \|f\|_{\gamma,0} e^{-\gamma|t|}
\end{align*}
for some constant $C_{k,m}$ and this finishes the proof of the lemma.
\end{proof}

We can now state our proposition about representation of the renewal kernel in next
\begin{proposition} \label{proposition:representation_renouvellement_fonctions_regulieres}
Under the assumption of theorem~\ref{theorem:renouvellement_general}, for any $\gamma>0$ small enough, there is $K\in \N$ such that for any $f\in \cal E^{\gamma,K}(\X\times \R)$, any $x\in \X$ and any $t\in \R$,
\[
\sum_{n=0}^{+\infty} P^n f(x,t) =\frac 1 {\sigma_\rho}\Pi_0 f(x,t)+ \frac 1 {2\pi} \int_\R e^{i\xi t} U(-i\xi)\widehat f(x,\xi) \di \xi
\]
Where $U$ is the operator defined in proposition~\ref{proposition:controle_inverse_perturbe} and we made the abuse of notations $U(-i\xi) \widehat{f}(x,\xi) = U(-i\xi) \widehat f_{\xi}(x)$ with $\widehat f_\xi = \widehat f(.,\xi)$.
\end{proposition}

\begin{proof}
For $s\in \R_+^\ast$, we note $P_s$ the operator defined by
\[
P_s f(x,t) = \int_\G e^{-s \sigma(g,x)} f(gx,t+\sigma(g,x)) \di\rho(g)
\]
We are first going to prove that if $f$ is non negative,
\[
\sum_{n=0}^{+\infty} P^nf(x,t) =  \lim_{s\to 0^+} \sum_{n=0}^{+\infty} P_s^n f(x,t)
\]
And then that
\[
\sum_{n=0}^{+\infty} P^n_s f(x,t) = \frac 1{2\pi} \int_\R e^{i\xi t} (I_d-P(s-i\xi))^{-1} \widehat f(x,\xi) \di \xi 
\]
To prove the first equality, note that
\[
\int_\G e^{-s\sigma(g,x)} f(g.(x,t))\di\rho^{\ast n}(g) = \int_\G e^{-s\sigma(g,x)}(\un_{\sigma(g,x)\leqslant 0} + \un_{\sigma(g,x)>0})f(g.(x,t)) \di\rho^{\ast n}(g)
\]
And the monotone convergence theorem proves that
\[
\lim_{s\to 0^+} \sum_{n=0}^{+\infty} \int_\G e^{-s\sigma(g,x)} \un_{\sigma(g,x)>0} f(g.(x,t)) \di\rho^{\ast n}(g) = \sum_{n=0}^{+\infty} \int_\G  \un_{\sigma(g,x)>0} f(g.(x,t)) \di\rho^{\ast n}(g)
\]
Moreover, lemma~\ref{lemma:drift} and Bienaymé-Tchebytchev's inequality prove that for any $x\in \X$,
\[
\rho^{\ast n}(g\in \G|\sigma(g,x) \leqslant 0) \leqslant \int_\G e^{-\eta\sigma(g,x)} \di\rho^{\ast n}(g) \leqslant C e^{-t \eta n}
\]
so the dominated convergence theorem proves that
\[
\lim_{s\to 0^+} \sum_{n=0}^{+\infty} \int_\G e^{-s\sigma(g,x)} \un_{\sigma(g,x)\leqslant 0} f(g.(x,t)) \di\rho^{\ast n}(g) = \sum_{n=0}^{+\infty} \int_\G  \un_{\sigma(g,x)\leqslant 0} f(g.(x, t)) \di\rho^{\ast n}(g)
\]
And this finishes the proof of the first equality.

Moreover,
\begin{align*}
I_s(t,x) :&=\sum_{n=0}^{+\infty} P^n_s f(t,x)  = \sum_{n=0}^{+\infty} \int_\G e^{-s\sigma(g,x)} f(gx,t+\sigma(g,x)) \di\rho^{\ast n}(g) \\
&= \frac 1{2\pi} \sum_{n=0}^{+\infty} \int_\G e^{-s\sigma(g,x)} \int_\R e^{i\xi(t+\sigma(g,x)} \widehat f(gx,\xi) \di \xi \di\rho^{\ast n}(g) \\
&= \frac 1{2\pi} \int_\R e^{i\xi t} \sum_{n=0}^{+\infty} \int_\G e^{-(s-i\xi) \sigma(g,x)} \widehat f(gx ,\xi) \di\rho^{\ast n}(g) \di \xi \\
&= \frac 1{2\pi} \int_\R e^{i\xi t} \sum_{n=0}^{+\infty} P(s-i\xi)^n \widehat f(x,\xi) \di \xi = \frac 1{2\pi} \int_\R e^{i\xi t} (I_d-P(s-i\xi))^{-1} \widehat f(x,\xi) \di \xi
\end{align*}
and this computation proves the second equality.

Finally, we noted $U(z)$ the family of operators defined by
\[
(I_d-P(z))^{-1}= \frac 1 {\sigma_\rho z} N_0 +U(z)
\]
and we saw in proposition~\ref{proposition:controle_inverse_perturbe} that $(U(z))$ is an analytic family of continuous operators on $\cal C^{0,\gamma}(\X)$.

And so,
\[
\int_\R e^{i\xi t} (I_d-P(s-i\xi))^{-1} \widehat f(x,\xi) \di \xi = \int_\R  \frac{e^{i\xi t}} {s-i\xi}N_0  \widehat f(x,\xi) \frac{\di \xi }{\sigma_\rho}+ \int_\R e^{i\xi t} U(s-i\xi) \widehat f(x,\xi) \di \xi
\]
But, for any continuous function $f$ on $\X$, using lemma~\ref{lemma:loi_grands_nombres_localement_contracte}, we can write
\[
N_0 f(x) = \sum_{i=1}^r p_i(x) \int f\di\nu_i
\]
where the $p_i$ are $P-$invariant functions and the $\nu_i$ are stationary probability measures on $\X$.
And so,
\begin{align*}
\frac 1 {2\pi}\int_\R  \frac{e^{i\xi t}}{s-i\xi} N_0 \widehat f(x,\xi) \di \xi &=\frac 1 {2\pi}\sum_{i=1}^r p_i(x) \int_\X \int_\R \frac{e^{i\xi t}}{s-i\xi} \widehat f(y,\xi) \di \xi \di\nu_i(y) \\
&= \sum_{i=1}^r p_i(x) \int_\X \int_{0}^{+\infty} f(y,t+u) e^{-su} \di u \di\nu_i(y) \\
&= \int_{0}^{+\infty} N_0 f(x,t+u) e^{-su} \di u
\end{align*}
where we used that
\[
\frac 1{s-i\xi} = \int_{0}^{+\infty} e^{-(s-i\xi) u} \di u
\]
So, for any $f\in \mathrm{L}^1(\R)$ such that $\widehat f \in \mathrm{L}^1(\R)$,
\[
\int_\R \frac{ e^{i\xi t}}{s-i\xi} \widehat f(\xi) \di \xi =\int_{0}^{+\infty} e^{-su}  \int_\R e^{i\xi (t+u)} \widehat f(\xi) \di \xi \di u = 2\pi \int_{0}^{+\infty} e^{-su} f(t+u) \di u
\]
This proves, using the definition of $\Pi_0$, that
\[
\lim_{s\to 0^+} \frac 1 {2\pi} \int_\R \frac{e^{i\xi t}}{s-i\xi} N_0 \widehat f(x,\xi) \di \xi = \int_t^{+\infty} N_0 f(x,u)  \di u = \Pi_0 f(x,t)
\]
Thus
\[
Gf(x,t) = \sum_{n=0}^{+\infty} P^n f(x,t) =\frac 1 {\sigma_\rho} \Pi_0 f(x,t)+ \lim_{s\to 0^+} \frac 1 {2\pi} \int_\R e^{i\xi t} U(s-i\xi) \widehat f(x,\xi) \di \xi
\]
and, as, for fixed $\xi$ we have, according to proposition~\ref{proposition:controle_inverse_perturbe}, that
\[
\|U(s-i\xi) \widehat f(x,\xi)\|_\infty \leqslant \|U(s-i\xi) \|_\gamma \|\widehat f(x,\xi)\|_\gamma \leqslant C(1+|\xi|)^{L+1} \|\widehat f(x,\xi)\|_\gamma 
\]
and as $f\in \cal E^{\gamma,K}(\X\times \R)$, we can conclude with the dominated convergence theorem~\ref{lemma:fourier_transform_vanish} taking $K=L+3$.
\end{proof}

\begin{corollary}\label{corollary:vitesse_renouvellement_fonctions_regulieres}
Under the assumption and notations of proposition~\ref{proposition:controle_inverse_perturbe}, for any $\gamma>0$ small enough there are constants $C,K$ such that for any $f\in \cal E^{\gamma,K}$, any $x,x' \in \X$ and any $t,t' \in \R$ we have that
\[
\left|\left(G-\frac 1{\sigma_\rho} \Pi_0\right)f(x,t) - \left(G-\frac 1{\sigma_\rho} \Pi_0\right) f(x',t')\right| \leqslant C \|f\|_{\gamma,K} \omega_0((x,t),(x',t'))^\gamma 
\]
where,
\[
\omega_0((x,t),(x',t')) =  \left\{\begin{array}{cl} \frac{\sqrt{|t-t'|^2+  d(x,x')^2} }{(1+|t'|)(1+|t|)} & \text{if }\pi_\Abf  (x) = \pi_\Abf (x')\\ 1 & \text{otherwise}  \end{array} \right.
\]
\end{corollary}

\begin{remark}
Letting $t'$ go to $+\infty$ and using Riemann-Lebesgue's lemma and proposition~\ref{proposition:representation_renouvellement_fonctions_regulieres} to get that under the assumption of the corollary,
\[
\lim_{t'\to \pm\infty } (G-\frac 1{\sigma_\rho} \Pi_0)f(x,t') = 0
\]
we find that for any $f\in \cal E^{\gamma,K}(\X\times \R)$, tout $x\in \X$ and any $t\in \R$,
\[
\left| Gf(x,t) - \frac 1{\sigma_\rho} \Pi_0 f(x,t)\right| \leqslant \frac{C}{(1+|t|)^{\gamma}} \|f\|_{\gamma,K}
\]
\end{remark}

\begin{proof}
According to proposition~\ref{proposition:representation_renouvellement_fonctions_regulieres},
\[
\left(G-\frac 1{\sigma_\rho} \Pi_0\right)f(x,t) = \frac 1 {2\pi} \int_\R e^{i\xi t} U(-i\xi) \widehat f(x,\xi) \di \xi
\]
So, integrating by parts and noting $\psi(x,\xi) = U(-i\xi) \widehat f(x,\xi)$, we find that for any $t\in \R^\ast$,
\[
(G-\frac 1{\sigma_\rho} \Pi_0) f(x,t) = \frac 1 {2\pi} \int_\R \frac{e^{i\xi t}}{-t^2} \psi''(x,\xi) \di \xi
\]
So, for any $x,x'\in \X$ such that $\pi_\Abf \circ \pi _\Hb(x) = \pi_\Abf \circ \pi_\Hb(x')$ and any $t,t' \in \R^\ast$,
\begin{align*}
I(x,t,x',t')&:= (G-\frac 1{\sigma_\rho} \Pi_0) f(x,t) - (G-\frac 1{\sigma_\rho} \Pi_0) f(x',t') \\&= \frac 1 {2\pi} \int_\R \frac{e^{i\xi t}}{-t^2} \psi''(x,\xi) \di \xi - \frac 1 {2\pi}\int_\R \frac{e^{i\xi t'}}{-t'^2} \psi''(x',\xi) \di \xi \\
&= \int_\R \left(\frac{e^{i\xi t}}{-t^2} -\frac{e^{i\xi t'}}{-t'^2} \right) \psi''(x,\xi) -\int_\R \frac{e^{i\xi t'}}{t'^2} \left(\psi''(x,\xi) -\psi''(x',\xi)\right) \frac {\di \xi} {2\pi} 
\end{align*}
So,
\[
\left|I(x,t,x',t') \right| \leqslant \int_\R \left|\frac{e^{i\xi t}}{t^2} -\frac{e^{i\xi t'}}{t'^2} \right| |\psi''(x,\xi)| \di \xi + \frac 1 {|t'|^2} \int_\R \left|\psi''(x,\xi) -\psi''(x',\xi)\right| \di \xi 
\]
But, assuming that $|t'| \geqslant |t|\geqslant 1$, we find that
\[
\left|\frac{e^{i\xi t}}{t^2} -\frac{e^{i\xi t'}}{t'^2} \right| \leqslant \frac{ |t^2 - t'^2|}{t^2 t'^2} + |\xi| \frac{ |t-t'|}{t'^2} \leqslant \frac{|t-t'|}{|t| |t'|} \left(2 + |\xi| \right)
\]
And, as we also have for $t\in \R$ with $|t| \geqslant 1$, that
\[
\frac 1 {|t|} \leqslant \frac 2 {1+|t|},
\]
what we get is that
\[
 \int_\R \left|\frac{e^{i\xi t}}{t^2} -\frac{e^{i\xi t'}}{t'^2} \right| |\psi''(x,\xi)| \di \xi \leqslant \frac{ 4 |t-t'| }{(1+|t|)(1+|t'|)} \int_\R (2+|\xi|) |\psi''(x,\xi)| \di\xi
 \]
Moreover,
\[
\frac 1 {|t'|^2} \int_\R \left|\psi''(x,\xi) -\psi''(x',\xi)\right| \di \xi  \leqslant \frac{ 4 d(x,y)^\gamma}{(1+|t|)(1+|t'|)} \int_\R m_\gamma(\psi''(.,\xi)) \di\xi 
\]
So we only need to check the integrability of $ \|\psi''(\,.\,,\xi)\|_\gamma$. But,
\[
\psi''(x,\xi) =-U''(-i\xi) \widehat f(x,\xi) -2i U'(-i\xi) \widehat f'(x,\xi) + U(-i\xi) \widehat f''(x,\xi)
\]
and according to proposition~\ref{proposition:controle_inverse_perturbe} there is a constant $C$ such that for any $m\in\{0,1,2\}$,
\[
\|U^{(m)}(-i\xi) \|_\gamma \leqslant C^{m+1} m! (1+|\xi|)^{(L+1)(m+1)}
\]
And, as we also have, according to lemma~\ref{lemma:fourier_transform_vanish} that for any $k\in \N$, there is a constant $C$ such that for any $l\in \{0,1,2\}$,
\[
\left\|\frac{ \partial^l \widehat f}{\partial \xi^l}(x,\xi) \right\|_\gamma \leqslant C\frac{\|f\|_{\gamma,k}}{1+|\xi|^k}
\]
we get that for some constants $C,K$ and any function $f\in \cal E^{\gamma,K}_0(\X\times \R)$, any $x,x' \in \X$ with $\pi_\Abf \circ \pi _\Hb(x) = \pi_\Abf \circ \pi_\Hb(x')$ and any $t,t' \in \R^\ast$ with $|t|,|t'| \geqslant 1$,
\begin{equation} \label{equation:lemma:renouvellement_regulier}
\left|I(x,t,x',t')\right| \leqslant C\|f\|_{\gamma,K} \frac{|t-t'|+ d(x,x')^\gamma}{(1+|t|)(1+|t'|)}
\end{equation}

We can now use the fact that for some constants $C_\gamma,C$, we have that for any $t,t' \in \R$ and any $x,x' \in \X$,
\[
\frac{|t-t'|+ d(x,x')^\gamma}{(1+|t|)(1+|t'|)} \leqslant C_\gamma \left( \frac{ |t-t'| + d(x,x')}{(1+|t|)(1+|t'|)} \right)^\gamma \leqslant C C_\gamma  \left( \frac{ \sqrt{|t-t'|^2 + d(x,x')^2}}{(1+|t|)(1+|t'|)} \right)^\gamma 
\]
To prove that inequality~\eqref{equation:lemma:renouvellement_regulier} still holds for $|t| \leqslant 1$, we are going to use the fact that the objects we consider behave well with the translations on $\R$.

Indeed, if $f\in \cal E^{\gamma, k}(\X\times \R)$ and if we note, for $s\in \R$, $f_s(t) = f(t-s)$ then for any $x\in \X$ and any $t\in \R$,
\[
(G-\frac 1{\sigma_\rho} \Pi_0) f(x,t) = (G-\frac 1{\sigma_\rho} \Pi_0) f_s(x,t+s)
\]
and
\[
\|f_s\|_{\gamma,k} \leqslant e^{\gamma|s|} \|f\|_{\gamma,k}
\]
So, if $|t| \leqslant 1$, we take $s\in [-10,10]$ such that $1 \leqslant |t+s| \leqslant |t'+s|$ and we get that for some other constant $C$,
\[
|I(x,t,x',t')| \leqslant 4C e^{10\gamma} \|f\|_{\gamma,k}  \left( \frac{ \sqrt{|t-t'|^2 + d(x,x')^2}}{(1+|t|)(1+|t'|)} \right)^\gamma 
\]
and this is finally what we intended to prove.
\end{proof}

\subsection{Renewal theorem for hölder-continuous functions} \subsectionmark{Hölder-continuous functions}
\label{section:renouvellement_holder}
 
Until now, we proved the renewal theorem only for regular functions when we are interested in functions that will just be hölder-continuous on $\R^d$.

To do so, we are going to regularize our functions by convolving them to regular ones and then use tauberian theorems to get the result.

This method is the one used for instance in~\cite{BDP15} and we will use, as they do, a tauberian theorem by Frennemo (see~\cite{Fre65} and appendix~\ref{annexe:frennemo}).

\medskip
We are going to define a new class of functions on $\X\times \R$ for which it is easy to work for the rate of convergence in the renewal theorem.

\begin{example}On $\R^d$, we study functions having (a power of)
\[
\omega'(x,y) = \frac{\|x-y\|}{(1+\|x\|)(1+\|y\|)}
\]
as continuity modulus. Using the application $\Phi:\Sb^d \times \R \to \R^d\setminus\{0\}$ that maps $(x,t)$ onto $e^t x$ and identifies $\R^d \setminus\{0\}$ and $\Sb^d\times \R$, our function $\omega$ becomes
\[
\omega'((x,t),(x',t') ) = \frac{\|e^t x - e^{t'} x'\|}{(1+e^t)(1+e^{t'})}
\]
But,
\begin{align*}
\|e^t x - e^{t'} x'\|^2 &= e^{2t} +e^{2t'} - 2 e^{t+t'} \langle x,x'\rangle = \left( e^t - e^{t'} \right)^2 + e^{t+t'} \|x-x'\|^2 \\
&= e^{t+t'} \left( \left( e^{(t-t')/2} - e^{(t'-t)/2} \right)^2 + \|x-x'\|^2 \right)
\end{align*}
So
\[
\omega'((x,t),(x',t')) = \frac{e^{t/2}}{1+e^t} \frac{e^{t'/2}}{1+e^{t'}} \sqrt{d(x,x')^2 + \left(e^{(t-t')/2} - e^{(t'-t)/2} \right)^2}
\]
\end{example}

This leads to the following definition when $(\X,d)$ is some compact metric space and $(x,t),(x',t') \in \X\times \R$, 
\[
\omega((x,t),(x',t')) = \left\{\begin{array}{cl} e^{-\frac{|t|+|t'|}2} \sqrt{d(x,x')^2 + \left(e^{(t-t')/2} - e^{(t'-t)/2}\right)^2}& \text{si }\pi_\Abf  (x) = \pi_\Abf (x')\\ 1 & \text{si non}  \end{array} \right.
\]
And we note
\[
\cal C^\gamma_\omega(\X\times \R):= \left\{ f\in \cal C^{0}(\X\times \R) \middle|  \sup_{\substack{(x,t),(x',t') \in \X\times \R \\ (x,t) \not= (x',t')}} \frac{|f(x,t) - f(x',t')|}{\omega((x,t),(x',t'))^\gamma} \text{ is finite}\right\}
\]
The following lemma proves that functions in $\cal C^\gamma_\omega$ can be extended to continuous functions on $\X\times \overline{\R}$.
\begin{lemma}
Let $f\in \cal C^\gamma_\omega(\X\times \R)$. Then, there are $p^+,p^- \in \mathrm{L}^{\infty}(\Abf)$ such that for any $x\in \X$ and any $t\in \R$,
\[
|f(x,t) - p^+(f)(\pi_\Abf (x))| \leqslant 2e^{\gamma t} \|f\|_{\gamma,\omega}\text{ et }|f(x,t) - p^-(f)(\pi_\Abf(x))| \leqslant 2e^{-\gamma t} \|f\|_{\gamma, \omega}
\]
(we recall our abuse of notations : we keep the function $\pi_\Abf$ for the function $\pi_\Abf \circ \pi_\Hb$)
\end{lemma}

\begin{proof}
We only do the proof for $p^-$ because it is the same idea for $p^+$.

For any $x\in \X$ ans any $t,t' \in \R$, we have that
\[
|f(x,t) - f(x,t')| \leqslant e^{-\frac{|t|+|t'|}2}  \left|e^{(t-t')/2} - e^{(t'-t)/2}\right| \|f\|_{\gamma, \omega}
\]
So, according Cauchy's criterion for functions, for any $x\in \X$, $(t\mapsto f(x,t))$ has a finite limit at $-\infty$.

So we can note $p^-(f)(x) = \lim_{t\to-\infty} f(x,t)$. And then, we have, by definition of $\cal C^\gamma_\omega$, that for any $(x,t),(x',t')\in \X\times \R$,
\begin{align*}
|f(x,t) - p^-(f)(x')|& \leqslant |f(x,t) - f(x',t')| + |f(x',t') - p^-(f)(x')|\\& \leqslant \|f\|_{\gamma,\omega} \omega((x,t), (x',t'))^\gamma + |f(x',t') - p^-(f)(x')|
\end{align*}
So, letting $t'$ go to $-\infty$, we find that for any $t\in \R$ and any $x,x' \in \X$ with $\pi_\Abf \circ \pi_\Hb (x)  =\pi_\Abf \circ \pi_\Hb (x')$,
\[
|f(x,t) - p^-(f)(x')| \leqslant e^{\gamma (t-|t|)/2} \|f\|_{\gamma,\omega}
\]
And this proves that $p^-(f)(x)= p^-(f)(x')$ and that, for any $t\in \R_-$,
\[
|f(x,t) - p^-(f)(x)| \leqslant e^{\gamma t}\|f\|_{\gamma,\omega}
\]
For $t\in \R_+$, we have that
\[
|f(x,t) - p^-(f)(x)| \leqslant 2 \|f\|_\infty \leqslant 2 e^{\gamma t} \|f\|_{\gamma, \omega}
\]
and this ends the proof of the lemma.
\end{proof}

In the sequel, we will use a regular function $\psi$ on $\R$ such that
\[
\lim_{t\to -\infty} \psi(t) = 1 \text{ and }\lim_{t\to +\infty} \psi(t) =0
\]
The space of functions in $\cal C^\gamma_\omega(\X\times \R)$ such that $p^-(f) = 0= p^+(f)$ has a finite codimension and this function $\psi$ will allow us to make the projections onto this space and it's supplementary explicit.

So, from now on, we note
\begin{equation} \label{equation:fonction_psi}
\psi(t) = \frac{ 1 }{\sqrt{2\pi}} \int_t^{+\infty} e^{-u^2/2} \di u
\end{equation}
The choice of this particular function is arbitrary but will simplify the computations to come.

In next lemma, we prove that the projection of a function in $\cal C^\gamma_\omega(\X\times \R)$ to the space of functions such that $p^+(f) = 0=p^-(f)$ has image in the space $\cal E^{\gamma,0}(\X\times \R)$ that we defined in equation~\eqref{equation:E_gamma_zero}.
\begin{lemma} \label{lemma:decomposition_fonctions_bord}
For any function $f\in \cal C^{0,\gamma}_\omega(\X\times \R)$, let
\[
\varphi(x,t) = f(x,t) - p^-(f) (x) \psi(t) - p^+(f)(x) (1-\psi(t))
\]
where $\psi$ is the function defined in equation~\eqref{equation:fonction_psi}.

Then, for any $\gamma \in \left]0,1\right]$ there is a constant $C$ depending only on $\gamma$ such that for all $f\in \cal C^{0,\gamma}_\omega(\X\times \R)$, we have that $\varphi \in \cal E^{\gamma,0}(\X\times \R)$ and
\[
\|\varphi\|_{\gamma,0} \leqslant C \|f\|_{\gamma,\omega}
\]
\end{lemma}

\begin{proof}
It is clear that $\varphi$ is continuous on $\X\times \R$.

Moreover, using that
\[
\frac1{\sqrt{2\pi}} \int_\R e^{-u^2/2} \di u = 1,
\]
we find that for any $t\in \R$ and any $x\in \X$,
\begin{align*}
|\varphi(x,t)| &\leqslant |f(x,t)-p^-(f)(x)| \psi(t) + |f(x,t) - p^{+}(f)(x)|(1-\psi(t)) \\
& \leqslant \|f\|_{\gamma,\omega} \left(e^{\gamma t} \psi(t) + e^{-\gamma t} (1-\psi(t)) \right)
\end{align*}
So,
\[
e^{\gamma|t|} |\varphi(x,t)| \leqslant \|f\|_{\gamma,\omega} \left( e^{\gamma(|t|+t)}\psi(t) + e^{\gamma(|t|- t)} (1-\psi(t)) \right)
\]
This proves that there exists some $C$ depending only on $\gamma$ such that
\[
\sup_t e^{\gamma|t|} |\varphi(x,t)| \leqslant C\|f\|_{\gamma,\omega}
\]
Moreover, for any $x,x'$ such that $\pi_\Abf \circ\pi_\Hb (x) = \pi_\Abf \circ\pi_\Hb (x')$,
\[
|\varphi(x,t) - \varphi(x',t)| = |f(x,t) - f(x',t)| \leqslant e^{-\gamma|t|} d(x,x')^\gamma \|f\|_{\gamma,\omega}
\]
And this finishes the proof of the lemma. 
\end{proof}

\begin{lemma}
Under the assumptions of theorem~\ref{theorem:renouvellement_general}, for any $f\in \cal C^{0,\gamma}_\omega(\X\times \R)$, any $x \in \X$ and any $t,t' \in \R$,
\[
\left|\sum_{n=0}^{+\infty} P^nf(x,t) - \sum_{n=0}^{+\infty} P^n f(x,t') \right| \leqslant C \|f\|_{\omega,\gamma} \left( e^{|t'-t|} -1\right)^\gamma
\]
\end{lemma}

\begin{proof}
Let $f\in \cal C^{0,\gamma}_\omega(\X\times \R)$, $x \in \X$ and $t,t' \in \R$.

Then,
\begin{align*}
I(t,t',x) : &=\left|\sum_{n=0}^{+\infty} P^nf(x,t) - \sum_{n=0}^{+\infty} P^n f(x,t') \right| \\
& \leqslant \sum_{n=0}^{+\infty} \int_\G |f(gx, t+\sigma(g,x)) - f(gx,t'+\sigma(g,x))| \di\rho^{\ast n}(g) \\
& \leqslant \|f\|_{\omega,\gamma} \sum_{n=0}^{+\infty} \int_\G \omega(gx,t+\sigma(g,x),gx,t'+\sigma(g,x))^\gamma \di\rho^{\ast n}(g) \\
& \leqslant \|f\|_{\omega,\gamma}\left| e^{(t-t')/2} - e^{(t'-t)/2} \right|^\gamma \sum_{n=0}^{+\infty} \int_\G e^{-\gamma|t+\sigma(g,x)|/2 - \gamma |t'+ \sigma(g,x)|/2} \di\rho^{\ast n}(g)  
\end{align*}
Moreover,
\[
|e^{(t-t')/2} - e^{(t'-t)/2}| = |e^{|t-t'|/2} - e^{-|t-t'|/2}| = e^{-|t-t'|/2} \left| e^{|t-t'|}-1 \right| \leqslant e^{|t-t'|} -1
\]
So,
\begin{align*}
I(t,t',x) &\leqslant \|f\|_{\omega,\gamma}  \left| e^{|t-t'|}-1 \right|^\gamma \sum_{n=0}^{+\infty} \int_\G e^{-\gamma |t+\sigma(g,x)|/2} \di\rho^{\ast n}(g) \\
& \leqslant \|f\|_{\omega,\gamma}  \left| e^{|t-t'|}-1 \right|^\gamma \sum_{n=0}^{+\infty} P^n \varphi(x,t)
\end{align*}
with $\varphi(x,t) = e^{-\gamma|t|/2}$.

To prove the lemma, we only need to check that the series is bounded uniformly in $(x,t)$. We would like to apply the renewal theorem (that would prove that the considered sum has finite limits at $\pm\infty$ and so is bounded) to the function $\varphi$ but we can't do it so early because this function is not regular. This is why we use two regular functions that dominate $\varphi$ :
\begin{align*}
\sum_{n=0}^{+\infty} \int_\G  & e^{-\gamma|t+\sigma(g,x)|} \di\rho^{\ast n}(g) \leqslant \sum_{n=0}^{+\infty}P^n \varphi_1(x,t) + \sum_{n=0}^{+\infty} P^n \varphi_2(x,t)
\end{align*}
where we noted
\[\varphi_1(t) = \frac{2}{\sqrt{2\pi}} e^{\gamma t} \int_{t}^{+\infty} e^{-u^2/2}\di u \text{ et } \varphi_2(t) =\frac2{\sqrt{2\pi}}e^{-\gamma t} \int_{-\infty}^t e^{-u^2/2} \di u
\] 
to get that
\[
\varphi_1(t) \geqslant \un_{\R_-} e^{\gamma t} \text{ et } \varphi_2(t)  \geqslant e^{-\gamma t} \un_{\R_+}(t)
\]
Then, the renewal theorem that we proved for regular functions (corollary~\ref{corollary:vitesse_renouvellement_fonctions_regulieres}) and Riemann Lebesgue's lemma, prove that
\[
\lim_{t\to +\infty} \sum_{n=0}^{+\infty} P^n \varphi_1(x,t) =0 \text{ and }\lim_{t\to-\infty} \sum_{n=0}^{+\infty} P^n \varphi_1(x,t) =N_0 1 \int_\R \varphi_1(t) \di t
\]
So, $\sum_{n=0}^{+\infty} P^n \varphi_1$ is bounded on $\X\times \R$. We treat $\sum_{n=0}^{+\infty} P^n \varphi_2$ in the same way and this finises the proof of the lemma.
\end{proof}

For $x,x'\in \X$ such that $\pi_\Abf \circ \pi_\Hb (x) = \pi_\Abf \circ \pi_\Hb(x')$ and $t,t' \in \R$, we note
\[
\omega_0 ((x,t), (x',t')) = \frac{ \sqrt{|t-t'|^2 + d(x,x')^2}}{(1+|t|)(1+|t'|)} 
\]

\begin{corollary} \label{corollaire:renouvellement_fonctions_nulles_bord}
Under the assumptions of theorem~\ref{theorem:renouvellement_general}, there are constants $C,K, \alpha \in \R_+^\ast$ such that for any $f\in \cal C^{0,\gamma}_\omega(\X\times \R) $ with $p^+(f)= 0 = p^-(f)$ any $x,x' \in \X$ and any $t,t'\in\R$,
\[
\left|(G-\frac 1{\sigma_\rho} \Pi_0)f(x,t) - (G-\frac 1{\sigma_\rho} \Pi_0) f(x',t') \right| \leqslant C \|f\|_{\gamma,\omega} \omega_0 ((x,t),(x',t'))^{\alpha}
\]
Moreover, for any $x\in \X$,
\[
\lim_{t \to\pm\infty} (G-\frac 1{\sigma_\rho} \Pi_0)f(x,t) = 0
\]
\end{corollary}

\begin{proof}
Let, for $k\in \N$ and $t\in \R$, $\varphi_k(t) = t^k e^{-t} \un_{\R_+}(t)$.

We already saw in lemma~\ref{lemma:regularizing} that there is a constant $C_k$ such that for any $f\in \cal E^{\gamma,0}(\X\times \R)$, $f\ast \varphi_{k+1} \in \cal E^{\gamma,k}(\X\times \R)$ and
\[
\|f\ast \varphi_{k+1}\|_{\gamma,k} \leqslant C_k \|f\|_{\gamma,0}
\]
In particular, for any $f\in \cal C^{0,\gamma}_\omega$, according to lemma~\ref{lemma:decomposition_fonctions_bord},
\[
\|f\ast \varphi_{k+1}\|_{\gamma,k} \leqslant C_k \|f\|_{\gamma,\omega}
\]

Moreover,
\[
\varphi_{k+1} \ast (G-\frac 1{\sigma_\rho} \Pi_0)f(x,t) = (G-\frac 1{\sigma_\rho} \Pi_0)(f\ast \varphi_{k+1})(x,t)
\]
and so, taking $k$ large enough, corollary~\ref{corollary:vitesse_renouvellement_fonctions_regulieres} proves that
\begin{align*}
I(k,x,t,x',t',f):&=\left| \varphi_{k+1} \ast(G-\frac 1{\sigma_\rho} \Pi_0)f(x,t) - \varphi_{k+1} \ast (G-\frac 1{\sigma_\rho} \Pi_0) f(x',t') \right| \\& \leqslant CC_k \omega_0((x,t),(x',t'))^\gamma \|f\|_{\gamma,\omega}
\end{align*}
And, as we also have that
\[
|(G-\frac 1{\sigma_\rho} \Pi_0) f(x,t) - (G-\frac 1{\sigma_\rho} \Pi_0) f(x,t')| \leqslant C\|f\|_{\gamma,\omega} \left(e^{|t-t'|} -1 \right)^\gamma
\]
we can conclude with the result by Frennemo stated in the appendix~\ref{annexe:frennemo}.
\end{proof}

We can now treat functions in $\cal C^{\gamma}_\omega(\X\times \R)$ that vanish at infinity. However, we would like to study functions on $\R^d\times \Abf$ that doesn't vanish on $\{0\}\times \Abf$ but such that $\sum_{a\in \Abf} f(a,0) = 0$. This is why we prove the

\begin{lemma} \label{lemma:renouvellement_fonctions_bord}
Under the assumptions of~\ref{theorem:renouvellement_general}, there is a constant $C$ such that for any $p\in \mathrm{L}^\infty(\Abf)$ with $\sum_{a\in \Abf} p(a) = 0$, for any $x,x'\in \X$ and any $t,t'\in \R$,
\[
\left|(G-\frac 1{\sigma_\rho} \Pi_0) f(x,t) - (G-\frac 1{\sigma_\rho} \Pi_0)f(x',t')\right| \leqslant C\|p\|_\infty \omega_0 ((x,t),(x',t'))
\]
where we noted
\[
f(x,t) = p(\pi_\Abf \circ \pi_\Hb (x)) \psi(t)
\]
and $\psi$ the function we defined in equation~\eqref{equation:fonction_psi}.

Moreover, if $\sum_{a\in \Abf} p(a) = 0$, then for any $x\in \X$,
\[
\lim_{t \to -\infty}\left(G- \frac 1{\sigma_\rho} \Pi_0\right)f(x,t)= \sum_{n=0}^{+\infty} P^n p(\pi_\Abf(x))
\]
\end{lemma}

\begin{proof}
As $P$ commutes to the derivation,
\[
\sum_{n=0}^{+\infty} P^n f(x,t) = -\int_t^{+\infty} \sum_{n=0}^{+\infty} P^n f'(x,u) \di u
\]

But, according to proposition~\ref{proposition:representation_renouvellement_fonctions_regulieres},
\[
\sum_{n=0}^{+\infty} P^n f'(x,u) = \frac 1 {\sigma_\rho}\int_{u}^{+\infty} N_0 f'(x,s) \di s + \frac 1 {2\pi} \int_\R e^{i\xi u} U(-i\xi) p(x) \widehat {f'}(\xi) \di \xi 
\]
But
\[
\widehat{f'}(\xi) = -e^{-\xi^2/2}
\text{ and }\int_u^{+\infty} N_0 f'(x,s) \di s = -N_0 f(x,u)
\]
So,
\[
\sum_{n=0}^{+\infty} P^n f'(x,u) = \frac{-1}{\sigma_\rho}N_0 f(x,u) + \frac{1}{2\pi}\int_\R e^{i\xi u} U(-i\xi) p(x) e^{-\xi^2/2} \di \xi
\]
This proves that
\[
\sum_{n=0}^{+\infty} P^n f(x,t) = \frac{ 1 }{\sigma_\rho}\int_t^{+\infty} N_0 f(x,u) \di u - \frac 1 {2\pi}\int_{t}^{+\infty} \int_\R e^{i\xi u} U(-i\xi) p(x) e^{-\xi^2/2} \di \xi \di u
\]
Moreover, noting $N_1 = U(0)$ and $V(z)$ such that $U(z)= N_1 + z V(z)$, we have that the second term (that only depend on $\pi_\Abf \circ \pi_\Hb (x)$) is the sum of
\[
\frac 1 {2\pi}  N_1 p(x)\int_t^{+\infty} \int_\R e^{i\xi u} e^{-\xi^2/2} \di \xi \di u = -N_1 p(x) \varphi(t)
\]
and of
\[
\frac 1 {2\pi} \int_{t}^{+\infty} \int_\R e^{i\xi u} (-i\xi)V(-i\xi) p(x) e^{-\xi^2/2} \di \xi \di u = \frac 1 {2\pi} \int_\R e^{i\xi t} V(-i\xi) p(x) e^{-\xi^2/2} \di \xi
\]
(to get this equality, we derive two times). And so, finally, we get that
\[
\sum_{n=0}^{+\infty} P^n f(x,t) = \frac 1 {\sigma_\rho}\int_t^{+\infty} N_0 f(x,u) \di u + N_1 f(x,t) - \frac 1 {2\pi} \int_\R e^{i\xi t} V(-i\xi) p(x) e^{-\xi^2/2} \di \xi
\]
So, according to Riemann-Lebesgue's lemma and using the definition of $\Pi_0$,
\[
\lim_{t \to -\infty} (G-\frac 1{\sigma_\rho} \Pi_0) f(x,t) = \lim_{t\to -\infty} N_1 f(x,t) = N_1 f(x,-\infty)
\]
Now, we need to prove that
\[
N_1 f(x,-\infty) = \sum_{n=0}^{+\infty} P^n f(\pi_\Abf \circ\pi_\Hb(x))
\]
To do so, note that these two functions are solutions of the equation $g-Pg=f$ and, as the random walk on $\Abf$ is irreducible and aperiodic, there is a constant $C\in \R$ such that for any $a\in \Abf$,
\[
N_1 p(a) = \sum_{n=0}^{+\infty} P^n p(a) + C
\]
But,
\[
\sum_{a\in \Abf} N_1 p (a) = \sum_{a\in \Abf}\sum_{n=0}^{+\infty} P^n p(a) + |\Abf| C = |\Abf| C
\]
And so, $C=0$ since $\sum_{a\in \Abf} N_1 p (a)= 0=\sum_{a\in \Abf}\sum_{n=0}^{+\infty} P^n p(a)$.

\medskip
Finally, according to proposition~\ref{proposition:controle_inverse_perturbe},
\[
\left| \int_\R e^{i\xi t} V(-i\xi) p(x) e^{-\xi^2/2} \di \xi\right| \leqslant \|p\|_\infty \int_\R C (1+|\xi|)^L e^{-\xi^2/2} \di \xi
\]
and
\[
|t|^2\left| \int_\R e^{i\xi t} V(-i\xi) p(x) e^{-\xi^2/2} \di \xi\right| \leqslant \|p\|_\infty C \int_\R (1+|\xi|)^{3L} e^{-\xi^2/2} \di \xi
\]
This finishes the proof of the lemma (because $p(x)$ only depend on the projection of $x$ onto $\Abf$).
\end{proof}

\appendix
\section{Remainder terms in Wiener's tauberian theorem}\label{annexe:frennemo}

\begin{miniabstract}
In this appendix, we state Frennemo's result about remainder terms in Wiener's tauberian theorem.
\end{miniabstract}
The aim of this section is the study of the following problem : given two functions $f,\varphi$ on $\R$ for whoch we now the rate of convergence to $0$ of $|f\ast \varphi|$ at infinity, can we get the one of $|f|$ ?

\medskip
The first answer to this question is a corollary of a tauberian theorem of Wiener that says that if $f$ is bounded, $\varphi$ is integrable, the Fourier-Laplace transform of $\varphi$ doesn't vanish on $i\R$ and if $f\ast \varphi$ converges to $0$ at $\pm\infty$, then, for any integrable function $g $ on $\R$, $f\ast g$ also converge to $0$ at $\pm\infty$.

\medskip
This kind of result is interesting for us in the study of the renewal theorem because to make our method work, we have to assume that our test functions are regular so we will have to regularize them by convoling them with regular functions and these remainder terms estimates will allow us to keep the rate of convergence to $0$.

\begin{definition_annexe}
Let $f$ be a uniformly continuous function on $\R$.

We say that a non decreasing function $\omega:\R_+\to\R_+$, that is continuous at $0$ and such that $\omega(0)=0$ is a modulus of uniform continuity for $f$ if for any $x,y\in \R$,
\[
|f(x)-f(y) | \leqslant \omega(|x-y|)
\]
\end{definition_annexe}

\medskip
The following theorem is an adaptation of the second theorem of Frennemo in~\cite{Fre65}.

\begin{theorem_annexe} \label{theorem:frennemo}
Let $k\in \N^\ast$.

Let $\varphi_k$ be the function defined on $\R$ by $\varphi_k(x)= x^{k} e^{-x} \un_{\R_+}(x)$

Then, there is a constant $C$ depending only on $k$ such that for any uniformly continuous and bounded function $f$ on $\R$ and any $x\in \R$,
\[
|f(t)| \leqslant C \inf_{V \in \R_+^\ast}\left( \omega_f\left(\frac 1 V \right) + \frac { \|f\|_\infty } V + (1+V)^k \sup_{t'\in \R }  e^{-|t'|} |\varphi_k \ast f(t-t')|\right)
\]
where $\omega_f$ is a modulus of uniform continuity for $f$.
%
%
\end{theorem_annexe}

%
%
%
\begin{lemma_annexe}
There is a constant $C$ such that for any integrable any uniformly continuous function $f$ on $\R$,
\[
\sup_{x\in \R}|f(x)| \leqslant C\inf_{V \in \R_+^\ast} \omega_f\left(\frac 1 V \right) + \sup_{\tau} \left| \int_{-V}^V e^{i\xi \tau} \left( 1 - \frac{|\xi|}V \right) \widehat f(\xi) \di \xi \right|
\]
where $\omega_f$ is a modulus of uniform continuity for $f$.
\end{lemma_annexe}

\begin{proof}
%
%
This lemma is a corollary of lemma~1 in~\cite{Fre65} when we remark that if $f$ is uniformly continuous then
\[
-\inf_{x\leqslant t \leqslant x+1/V} f(t) - f(x)  \leqslant \sup_{x\leqslant t \leqslant x+1/V} \omega_f(|x-t|) \leqslant \omega_f(1/V) \qedhere
\]
\end{proof}

\begin{proof}[Proof of theorem~\ref{theorem:frennemo}]

For any $s\in \R$, let $u_s$ be the function defined for $t\in \R$ by $u_s(t)= f(t) e^{-\frac 1 2(t-s)^2}$.

Then, for any $t,t'\in \R$,
\begin{align*}
|u_s(t)-u_s(t') |&= \left|e^{-\frac 1 2(t-s)^2} (f(t)-f(t')) + f(t') (e^{-\frac 1 2(t-s)^2} - e^{-\frac 1 2(t'-s)^2}) \right|\\
& \leqslant \omega_f(|t-t'|) + \|f\|_\infty |t-t'| \sup_{u\in \R} |u| e^{-u^2/2} 
\end{align*}
So, the function $u_s$ is uniformly continuous. As it is integrable, we have, according to the lemma, that for any  $V\in \R_+^\ast$ and any $t,s\in \R$,
\begin{align*}
|f(t)|&= |u_s(t)|\\& \leqslant C \left(\omega_f\left(\frac 1 V\right) +  \frac{ \|f\|_\infty} V \sup_{u\in \R} |u|e^{-u^2/2} + \sup_{\tau} \left|\int_{-V}^V e^{i\xi\tau} \left(1 - \frac{|\xi|} V \right) \widehat {u_s}(\xi) \di \xi \right| \right)
\end{align*}
Moreover, Frennemo proves that
\[
\sup_{\tau} \left| \int_{-V}^V e^{i\xi \tau} \left(1 - \frac{|\xi|} V \right) \widehat{u_s}(\xi) \di \xi \right| \leqslant C (1+V)^k \sup_{t'\in \R }  e^{-|t'|} |\varphi_k \ast f(t-t')|
\]
This is what we intended to prove.
\end{proof}

Let $(\X,d)$ be a compact metric space and $\gamma \in \left]0,1\right]$. For any $(x,t),(x',t') \in \X\times \R$, we note
\[
\omega_0((x,t),(x',t')) =  \frac{\sqrt{|t-t'|^2+ d(x,x')^2}}{(1+|t'|)(1+|t|)}
\]

\begin{corollary_annexe}
Let $(\X,d)$ be a compact metric space and $\gamma \in \left]0,1\right]$.

For any $k\in \N$, there is a constant $C_k \in \R_+$ and $\alpha \in \R_+^\ast$ such that for any continuous bounded function $f$ on $\X\times \R$ such that
\[
|f(x,t) - f(x,t')| \leqslant \left( e^{|t-t'|}-1 \right)^\gamma C(f), \;\|f\|_\infty \leqslant C(f)
\]
and
\[
|\varphi_k \ast f(x,t) - \varphi_k \ast f(x',t') |\leqslant C(f) \omega((x,t),(x',t'))
\]
We have that
\[
|f(x,t) - f(x',t')| \leqslant C_k C(f) \omega((x,t),(x',t'))^\alpha
\]
\end{corollary_annexe}

\begin{proof}
For any $x,x'\in \X$ and $s\in \R$, we note $f_{x,x',s}$ the function defined for any $t\in \R$ by
\[
f_{x,x',s}(t) = f(x,t) - f(x',t+s)
\]
Then, for any $t,t' \in \R$,
\begin{align*}
|f_{x,x',s}(t)  - f_{x,x',s}(t') | &= |f(x,t) - f(x',t+s) - f(x,t') + f(x',t'+s)| \\&\leqslant 2 C_0 \left( e^{|t-t'|} -1 \right)^\gamma
\end{align*}
and
\[
|\varphi_k \ast f_{x,x',s} (t)| = |\varphi_k \ast f(x,t) - \varphi_k \ast f(x',t+s)| \leqslant C(f) \omega((x,t),(x',t+s))
\]
And so, according to Frennemo's theorem,
\begin{align*}
|f_{x,x',s}(t)| &\leqslant C C(f) \left(\inf_{ V\in \R_+^\ast} 2 (e^{1/V}-1)^\gamma  + \frac{2}{V} \right. \\
& \retrait \left.+(1+V)^k \sup_{t'\in \R} e^{-|t'|} \omega((x,t-t'),(x',t-t'+s)) \right)
\end{align*}
But, for any $t,t'\in \R$, 
\[
\frac 1 {1+|t-t'|} \leqslant \frac{1+|t'|}{1+|t|}
\]
This proves that
\[
\sup_{t'\in\R}e^{-|t'|} \omega((x,t-t'),(x',t-t'+s)) \leqslant \omega((x,t),(x',t+s)) \sup_{t'\in \R} e^{-|t'|}(1+|t'|)
\]
Thus, for an other constant $C$ that doesn't depend on $f$,
\[
|f_{x,x',s}(t)| \leqslant C C(f) \inf_{V\in [1, +\infty[} \frac 1 {V^\gamma} + (1+V)^k \omega((x,t),(x',t+s))
\]
Noting now, for some $\delta \in \R_+^\ast$ that we will choose later,
\[
V =\omega((x,t),(x',t+s))^{-\delta}
\]
we find that for some constant $C$ that doesn't depend on $f$,
\[
|f(x,x',s)(t)| \leqslant C C(f) \left( \omega^{\gamma \delta}+ (1+\omega^{-\gamma \delta})^k \omega \right) \leqslant C C(f) \omega^{\alpha}
\]
for $\delta$ small enough and some $\alpha$ (where we used that $\omega$ is bounded on $\X\times \R$).

This ends the proof of the corollary.
\end{proof}

\bibliographystyle{amsalpha}
\bibliography{biblio}

\providecommand{\bysame}{\leavevmode\hbox to3em{\hrulefill}\thinspace}
\providecommand{\MR}{\relax\ifhmode\unskip\space\fi MR }
\providecommand{\MRhref}[2]{%
  \href{http://www.ams.org/mathscinet-getitem?mr=#1}{#2}
}
\providecommand{\href}[2]{#2}
\begin{thebibliography}{BFLM11}

\bibitem[BDP15]{BDP15}
Dariusz Buraczewski, Ewa Damek, and Tomasz Przebinda, \emph{On the rate of
  convergence in the {K}esten renewal theorem}, Electron. J. Probab.
  \textbf{20} (2015), no. 22, 35. \MR{3325092}

\bibitem[BFLM11]{BFLM11}
Jean Bourgain, Alex Furman, Elon Lindenstrauss, and Shahar Mozes,
  \emph{Stationary measures and equidistribution for orbits of nonabelian
  semigroups on the torus}, J. Amer. Math. Soc. \textbf{24} (2011), no.~1,
  231--280. \MR{2726604 (2011k:37008)}

\bibitem[BG07]{BG07}
J.~Blanchet and P.~Glynn, \emph{Uniform renewal theory with applications to
  expansions of random geometric sums}, Adv. in Appl. Probab. \textbf{39}
  (2007), no.~4, 1070--1097. \MR{2381589 (2009e:60191)}

\bibitem[BL85]{BL85}
Philippe Bougerol and Jean Lacroix, \emph{Products of random matrices with
  applications to {S}chr\"odinger operators}, Progress in Probability and
  Statistics, vol.~8, Birkh\"auser Boston, Inc., Boston, MA, 1985. \MR{886674
  (88f:60013)}

\bibitem[Bla48]{Bla48}
David Blackwell, \emph{A renewal theorem}, Duke Math. J. \textbf{15} (1948),
  145--150. \MR{0024093 (9,452a)}

\bibitem[BQ14]{BQproj}
Yves Benoist and Jean-Fran{\c{c}}ois Quint, \emph{Random walks on projective
  spaces}, Compos. Math. \textbf{150} (2014), no.~9, 1579--1606. \MR{3260142}

\bibitem[BQ15]{BQred}
\bysame, \emph{Random walks on reductive groups}, Preprint (2015), ~.

\bibitem[Bre05]{Bre05}
E.~Breuillard, \emph{Distributions diophantiennes et th\'eor\`eme limite local
  sur {$\mathbb{R}^d$}}, Probab. Theory Related Fields \textbf{132} (2005),
  no.~1, 39--73. \MR{2136866 (2006b:60093)}

\bibitem[Car83]{Car83}
Hasse Carlsson, \emph{Remainder term estimates of the renewal function}, Ann.
  Probab. \textbf{11} (1983), no.~1, 143--157. \MR{682805 (84e:60127)}

\bibitem[Dol98]{Dol98}
Dmitry Dolgopyat, \emph{Prevalence of rapid mixing in hyperbolic flows},
  Ergodic Theory and Dynamical Systems \textbf{18} (1998), 1097--1114.

\bibitem[Fre65]{Fre65}
Lennart Frennemo, \emph{On general {T}auberian remainder theorems}, Math.
  Scand. \textbf{17} (1965), 77--88. \MR{0193439 (33 \#1659)}

\bibitem[Fur63]{Fur63}
Harry Furstenberg, \emph{Noncommuting random products}, Trans. Amer. Math. Soc.
  \textbf{108} (1963), 377--428. \MR{0163345 (29 \#648)}

\bibitem[GC64]{GC64}
I.~M. Guelfand and G.~E. Chilov, \emph{Les distributions. {T}ome 2: {E}spaces
  fondamentaux}, Traduit par Serge Vasilach. Collection Universitaire de
  Math\'ematiques, XV, Dunod, Paris, 1964. \MR{0161145 (28 \#4354)}

\bibitem[GL12]{GuLe15}
Y.~{Guivarc'h} and E.~{Le Page}, \emph{{Spectral gap properties for linear
  random walks and Pareto's asymptotics for affine stochastic recursions}},
  ArXiv e-prints (2012), ~.

\bibitem[GR85]{GuRa85}
Y.~Guivarc'h and A.~Raugi, \emph{Fronti\`ere de {F}urstenberg, propri\'et\'es
  de contraction et th\'eor\`emes de convergence}, Z. Wahrsch. Verw. Gebiete
  \textbf{69} (1985), no.~2, 187--242. \MR{779457}

\bibitem[ITM50]{ITM50}
C.~T. Ionescu~Tulcea and G.~Marinescu, \emph{Th\'eorie ergodique pour des
  classes d'op\'erations non compl\`etement continues}, Ann. of Math. (2)
  \textbf{52} (1950), 140--147. \MR{0037469 (12,266g)}

\bibitem[Kes74]{Ke74}
Harry Kesten, \emph{Renewal theory for functionals of a markov chain with
  general state space}, Ann. Probab. \textbf{2} (1974), no.~3, 355--386.

\bibitem[Qui05]{Qu05}
Jean-Fran{\c{c}}ois Quint, \emph{Groupes de {S}chottky et comptage}, Ann. Inst.
  Fourier (Grenoble) \textbf{55} (2005), no.~2, 373--429. \MR{2147895
  (2006g:37033)}

\bibitem[Rau92]{Rau92}
Albert Raugi, \emph{Théorie spectrale d'un opérateur de transition sur un
  espace métrique compact}, Annales de l'institut Henri Poincaré (B)
  Probabilités et Statistiques \textbf{28} (1992), no.~2, 281--309 (fre).

\bibitem[Rud87]{Rud87}
Walter Rudin, \emph{Real and complex analysis}, third ed., McGraw-Hill Book
  Co., New York, 1987. \MR{924157 (88k:00002)}

\bibitem[Ser78]{Ser78}
Jean-Pierre Serre, \emph{Repr\'esentations lin\'eaires des groupes finis},
  revised ed., Hermann, Paris, 1978. \MR{543841}

\end{thebibliography}

\end{document}